\documentclass{amsart}
\usepackage[utf8]{inputenc}
\usepackage[margin=1in]{geometry} 

\usepackage[utf8]{inputenc} 
\usepackage[T1]{fontenc} 
\usepackage[autostyle=true]{csquotes} 
\usepackage{amsmath}
\usepackage{mathrsfs}
\usepackage{amsfonts}
\usepackage{dsfont}
\usepackage{amssymb}
\usepackage{graphicx}
\usepackage{accents}
\usepackage{bm}
\usepackage{caption}
\usepackage{tikz-cd}
\usepackage{float}
\usepackage{wasysym}
\usepackage{hyperref}
\usepackage{url}
\usepackage{amsthm}
\usepackage{bbm}
\usepackage{imakeidx}

\hypersetup{
    colorlinks=true,
    citecolor=,
    linkcolor=,
    filecolor=,
    urlcolor=blue}
\theoremstyle{plain}

\newtheorem{theorem}{Theorem}
\newtheorem*{theorem*}{Theorem}
\newtheorem{theoremalpha}{Theorem}

\newtheorem{corollary}[theorem]{Corollary}
\newtheorem{corollaryalpha}[theoremalpha]{Corollary}
\newtheorem*{corollary*}{Corollary}
\newtheorem*{question}{Question}
\newtheorem{lemma}[theorem]{Lemma}
\newtheorem{proposition}[theorem]{Proposition}

\newtheorem{example}[theorem]{Example}
\newtheorem{definition}[theorem]{Definition}
\newtheorem{notation}[theorem]{Notation}
\theoremstyle{remark}
\newtheorem{remark}[theorem]{Remark}

\numberwithin{theorem}{section}
\numberwithin{figure}{section}
\numberwithin{equation}{section}

\newcommand{\N}{\mathbb{N}}
\newcommand{\Z}{\mathbb{Z}}

\newcommand{\R}{\mathbb{R}}
\newcommand{\C}{\mathbb{C}}

\newcommand{\set}[1]{\left\{ #1 \right\}}
\newcommand{\ff}{\mathbb{F}}
\newcommand{\cat}{\mathsf{C}}
\newcommand{\catD}{\mathsf{D}}

\newcommand{\X}{\mathcal {X}}
\newcommand{\Y}{\mathcal {Y}}

\newcommand{\Ltwo}[1]{\text{L}^2(#1)}

\newcommand{\tens}[1]{\mathbin{\mathop{\otimes}\limits_{{#1}}}}
\newcommand{\fuse}[1]{\mathbin{\mathop{\boxtimes}_{#1}}}
\newcommand{\ev}[1]{\mathbin{\mathop{\text{ev}_{#1}}}}
\newcommand{\coev}[1]{\mathbin{\mathop{\text{coev}_{#1}}}}

\newcommand{\cH}{\mathcal{H}}
\newcommand{\cS}{\mathcal{S}}
\newcommand{\cY}{\mathcal{Y}}

\DeclareMathOperator{\id}{\text{id}}
\DeclareMathOperator{\one}{\mathbbm{1}}
\DeclareMathOperator{\Tr}{Tr}
\DeclareMathOperator{\tr}{tr}

\DeclareMathOperator{\Bim}{\mathsf{Bim}}
\DeclareMathOperator{\vect}{\mathsf{Vec}}

\DeclareMathOperator{\Gr}{Gr}

\DeclareMathOperator{\Ind}{Ind}
\DeclareMathOperator{\spann}{span}
\DeclareMathOperator{\End}{End}

\definecolor{OhioRed}{rgb}{0.73047, 0, 0} 

\title{Realizations of rigid C*-tensor categories as bimodules over GJS C*-algebras}
\author{Michael Hartglass}
\author{Roberto Hern\'{a}ndez Palomares }
\date{\today}

\begin{document}

\maketitle
\begin{abstract}
    Given an arbitrary countably generated rigid C*-tensor category, we construct a fully-faithful bi-involutive strong monoidal functor onto a subcategory of finitely generated projective bimodules over a simple, exact, separable, unital C*-algebra with unique trace. The C*-algebras involved are built from the category  using the GJS-construction introduced in \cite{GJS11} and further studied in \cite{BHP12} and \cite{HaPeII}. Out of this category of Hilbert C*-bimodules, we construct a fully-faithful bi-involutive strong monoidal functor into the category of bi-finite spherical bimodules over an interpolated free group factor. The composite of these two functors recovers the functor constructed in \cite{BHP12}.\\
\end{abstract}

\tableofcontents{}

\section{Introduction}
		Since the advent of modern subfactor theory by Jones, there has been considerable interest in aximotizing the standard invariant 
	resulting from a finite-index subfactor $N \subset M.$  The standard invariant has been axiomitized by Jones' \underline{planar algebras}, 
	Ocneanu's \underline{paragroups} (for finite depth subfactors), and Popa's \underline{$\lambda$-lattices}.

		A natural question is given a standard invariant, can one construct a subfactor with that standard invariant?  This question was answered in 
	the affirmative by Popa in \cite{MR1278111} where given a $\lambda$-lattice, a finite-index II$_{1}$ subfactor $N \subset M$ was constructed whose standard invariant was the given $\lambda$-lattice.

		Popa's theorem was reproven utilizing a diagrammatic argument by Guionnet, Jones, and Shlyakhtenko \cite{GJS10}.  The first author and Penneys utilized the construction
	 therein to state an analogous result in a C*-algebraic setting:
	\cite{HaPeII} \textit{Given a standard invariant,} $\mathcal{P}_{\bullet}^{\pm}$, \textit{there is an inclusion } $B_0\subset B_1$ \textit{ with finite Watatani index of simple, unital C*-algebras with 
	unique trace satisfying:}
		 $\mathcal{P}_{n}^{+} = B_{0} \cap B_{n}$ \textit{and} $\mathcal{P}_{n}^{-} = B_{1} \cap B_{n+1}.$ \textit{Here} $(B_n)$ \textit{is the Jones-Watatani tower
	for} $B_0\subset B_1.$

		As the $N-N$ and $M-M$ bimodules generated by $\Ltwo{M}$ as an $N-M$ bimodule form rigid C$^{*}$-tensor categories, it is natural to ask to what degree an arbitrary 
	rigid C*-tensor category can be represented by bimodules over a II$_{1}$ factor.  The first author, along with Brothier and Penneys \cite{BHP12} proved that every rigid C*-tensor category is equivalent to a 
	subcategory of bimodules of $L(\mathbb{F}_{\infty})$. (See Example \ref{BifiniteBimodules}.)

	\begin{theorem*}\cite{BHP12}
		Given a countably generated rigid C*-tensor category, $\cat$, there is a  a full and faithful bi-ivolutive strong monoidal functor 
		$$
			\mathbb G: \cat \hookrightarrow\Bim_{\mathsf{bf}}^{\mathsf{sp}}L(\mathbb{F}_{\infty}).
		$$
	\end{theorem*}

		It is natural to ask what sort of reconstruction theorems can be obtained in the context of projective Hilbert C*-bimodules over C*-algebras. 
 	One immediate difficulty is that due to K-theoretic obstructions, it is impossible to find a single separable unital C*-algebra $B$ 
	so that every countably generated rigid C*-tensor category is fully realized as a subcategory of the projective C*-bimodules over $B$ (see the discussion before Corollary \ref{cor::universalFusion}).  
	This observation leads to the following natural question:
	\begin{question}
		Given a rigid C*-tensor category, $\cat$, is there a separable, simple, unital C*-algebra $B$ (depending on $\cat$)
		 so that $\cat$ is realized as a full subcategory of the projective C*-bimodules of $B$?
	\end{question}

	In \cite{Yu19}, Yuan answered the above question in the affirmative under the weaker assumption of not requiring norm separability on the C*-algebra, even in the case where $\cat$ is countably generated. 
	 In this article, we answer the above question in the affirmative: 

	\begin{theoremalpha}\label{thm::A}
		Given a countably generated rigid C*-tensor category with simple unit, $\cat$, there is a  unital, simple, separable, exact C*-algebra, $B_0$ with unique trace, 
		and a fully-faithful bi-involutive strong monoidal functor 
		$$
			\ff:\cat\hookrightarrow \Bim_{\mathsf{fgp}}^{\mathsf{tr}}(B_0),
		$$	
		 into the rigid C*-tensor category of finitely generated projective  $B_0$-bimodules compatible with the (unique) trace of $B_0$. 
		(See Example \ref{HappyBims}, Equation \ref{eqn::ComTrace}  and \cite[Definition 5.7]{MR1624182})
		Moreover, the $K_{0}$ group of $B_0$ is the free abelian group on the classes of simple objects in $\cat$, 
		and the image of the simple objects of $\cat$ are precisely the canonical generators of $K_{0}(B_0).$
	\end{theoremalpha}

		Our proof of this theorem can be outlined as follows:

			$\clubsuit$ We take off from Chapter \ref{sec::diagAlgs} by choosing a \textit{symmetrically self-dual} object $x\in\cat$
		 (Definition \ref{dfn::SSD}) of \textit{quantum dimension} $\delta_x > 1,$ and use it to
		 construct a graded \textit{ground }C*-\textit{algebra}  $A_\infty$ (Equation \ref{eqn::GroundAlg}) from the endomorphism C*-algebras $\cat(x^{\otimes n})$. We then build		
		 an $A_\infty$-module $X_1$ and consider the \textit{creation and annihilation operators by real vectors} inside the C*-algebra of \textit{adjointable} 
		endomorphisms of $\mathcal{F}({X_1}_{A_\infty}),$ the \textit{full Fock space of} ${X_1}_{A_\infty}.$ 
		(See Definition \ref{dfn::Fock}, the discussion following Corollary \ref{cor:fockspace} and \cite{MR1426840}.) These operators together with $A_\infty$
		generate the simple C*-algebra $B_\infty$  (\textit{a.k.a. Voiculescu semi-circular system}), which comes equipped with a faithful conditional expectation
		onto $A_\infty$ and tracial weight $\Tr$. We then consider certain corners $B_n\subset B_\infty$(Notation \ref{notation::corners}), for $n\geq 0,$ which are unital, separable, simple and
		have a unique trace, given by $ \tr_{B_n} := \Tr|_{B_n}.$ 

			$\clubsuit\clubsuit$ The algebra $B_0$ is of central importance to us, since it acts on the \textit{off-corners} $\{_lB_r\}_{l,r\in\N\cup 0}$ of $B_\infty$ by means of
		 \textit{multiplication in the ambient} C*-\textit{algebra} $B_\infty$. In this framework all our actions are automatically bounded.
		 We then construct a rigid C*-tensor category of \textit{finitely generated projective, minimal and normalized} Hilbert C*-bimodules over $B_0$, denoted 
		$\Bim_{\mathsf{fgp}}(B_0)$ (see Definitions \ref{index} and \ref{minimal}), which includes the bimodules $_0B_r,$ for $r\geq 0.$ The simplicity of each $B_n$ allows us to 
		describe the \textit{Connes' fusion} relative to $B_0$ of these bimodules in simple terms (Lemma \ref{lemma::isomorphism}) as:
		 $$_{B_0}({_0B_n}\fuse{B_0}{_0B_m})_{B_0} \cong _{B_0}({_0B_{n+m}})_{B_0}.$$ 
		This construction then induces the functor 
		\begin{equation*}
		        \begin{tikzpicture}[scale=1/13 pt, thick,baseline={([yshift=-\the\dimexpr\fontdimen22\textfont2\relax] current bounding box.center)}] 	        
		        \node at(-30,17){$\ff:\cat_x\hookrightarrow \Bim_{\mathsf{fgp}}(B_0)$};    
			\node at (-75,4){$x^{\otimes n}\longmapsto {_{B_0}\left({_0B_n}\right)} _{B_0},$};
			\node at(-15,4){$\cat( x^{\otimes n}\longrightarrow x^{\otimes m})\owns f \longmapsto \ \ \ff(f):$};
				            
		            \draw[rounded corners, ultra thick] (20,-1) rectangle (30,9);
		            \node at (25,9){$\bullet$};
		                \draw[ultra thick] (29,9) --(29,16);
		                    \node at(31,14){$n$};
		                \draw[densely dashed] (21,9) --(21,16);
		                \draw[ultra thick] (25,-1) -- (25,-6);

		            \node at (35,4){$\longmapsto$};

		            \draw[rounded corners, ultra thick] (41,-4) rectangle (49,4);
		            \node at (45,4){$\bullet$};
		                \draw[ultra thick] (47,4) --(47,8);
		                    \node at(49,6){$n$};
		                \draw[densely dashed] (43,4) --(43,19);
		                \draw[ultra thick] (45,-4) -- (45,-8);

		            \draw[rounded corners, ultra thick] (44.5,8) rectangle (49.5,14);
		            \node at (47,11){$f$};
		                \draw[ultra thick] (47,14) -- (47,19);
		                    \node at(49.5,17.5){$m$};
		        \end{tikzpicture}.
	    	\end{equation*}
		Here, $\cat_x$ is the full subcategory of $\cat$ generated by tensor powers of $x$ .
		By construction, $\ff$ will be a faithful bi-involutive strong monoidal functor, as will be shown in Chapter \ref{Sec::Representing}.
		 However, showing that it is full takes some more work, and the proof is delayed to the last chapter.

     		     $\clubsuit\clubsuit\clubsuit$ The next step of the proof consists of turning Hilbert C*-bimodules into \textit{bifinite and spherical/extremal bimodules over an interpolated} $\rm{II}_1$\textit{-factor} $M_0$.
		This is done and explained in Chapter \ref{section:Hilbert}.
		The factor $M_0$ is in fact isomorphic to the \textit{von Neumann algebra generated by} $B_0$, and we shall find the tools to compute it over several different representations. 
		(Lemma \ref {lemma::Hilbertifying})
		In order to do so, one has to further restrict $\Bim_{\mathsf{fgp}}(B_0)$ to those bimodules which are \textit{compatible
		 with the trace} $\tr_{B_0}$, and we denote this rigid C*-tensor subcategory by $\Bim_{\mathsf{fgp}}^{\mathsf{tr}}({B_0}).$
		From $\Y\in\Bim_{\mathsf{fgp}}^{\mathsf{tr}}({B_0})$ we obtain an \textit{honest} Hilbert space  by fusing it with $\Ltwo{B_0}$, obtaining $\Y\fuse{B_0}\Ltwo{B_0}.$
		 Afterwards, we extend both left and right $B_0$-actions on  $\Y$ to normal left and right actions of $M_0$. This process in fact defines a full and faithful bi-involutive strong monoidal functor
		(shown in Propositions \ref{prop::HilbBiInv} and \ref{prop::propertiesHilbertify})
		$$
			_{M_0}(-\fuse{B_0}\Ltwo{B_0})_{M_0}: \Bim_{\mathsf{fgp}}^{\mathsf{tr}}(B_0)\hookrightarrow \Bim_{\mathsf{bf}}^{\mathsf{sp}}(M_0)
		$$
		\begin{equation*}
		    	\Y\longmapsto\ _{M_0}(\Y\fuse{B_0}\Ltwo{B_0})_{M_0}, \hspace{1cm}
		   \	\Bim_{\mathsf{fgp}}^{\mathsf{tr}}(B_0) (\Y\rightarrow \mathcal{Z})\owns f\longmapsto f\boxtimes\id_{\Ltwo{B_0}}.
		\end{equation*}

		$\clubsuit\clubsuit\clubsuit\clubsuit$ Finally, we recover the functor $\mathbb G$ from \cite{BHP12} via the commuting 2-cell
		\begin{equation*}\label{eqn::triangle}
	           	\begin{tikzcd}[column sep=large, row sep=large]
	        	& \cat_x \arrow[dl, "\ff"', hook] \arrow[dr,hook, two heads, "\mathbb G"{name=G}]  & \\
	        	|[alias=B]|\Bim_{\mathsf{fgp}}^{\mathsf{tr}}(B_0) \arrow["_{M_0}(-\fuse{B_0}\Ltwo{B_0})_{M_0}"', hook]{rr}\arrow[Rightarrow, from = B, to = G, "T", "\cong" below, shorten >=7mm, shorten <=7mm]  & & \Bim_{\mathsf{bf}}^{\mathsf{sp}}(M_0), 
	           	\end{tikzcd}
	      \end{equation*} 
	     	where $T$ is a monoidal unitary natural isomorphism.  It is preciselly this commuting 2-cell that grants us the grounds for showing that $\ff$ is full, and this is indeed proven by a simple finite-dimensional
		linear algebraic trick. (See Proposition \ref{prop::CommTrian}.) We consider this step the most important part of the proof, since it provides a \textbf{conceptual and concrete connection between the realization
		 functors} $\ff$ and $\mathbb G.$

		As a consequence of Theorem \ref{thm::A} (see also Remark \ref{remark:countable}), we found a way around the aforementioned K-theoretical obstruction to the existence of
	a  C*-algebra $B$ such that every rigid C*-tensor category can be realized into $\Bim_{\mathsf{fgp}}(B),$ by restricting our scope to unitary fusion categories: (Corollary \ref{cor::universalFusion})
	\begin{corollaryalpha}
		There exists a unital simple exact separable C*-algebra $B$ with unique trace over which we can full and faithfully realize every unitary
		 fusion category $\cat$ into $\Bim_{\mathsf{fgp}}(B),$ in the spirit of Theorem \ref{thm::A}.
	\end{corollaryalpha}

\section*{Aknowledgements}
	The authors would like to thank Corey Jones and David Penneys for a number of useful conversations and their intense sharing of ideas and comments.
	Hern\'{a}ndez Palomares was partially supported by CONACyT.
\section{Background}
\subsection{Rigid C*-tensor categories}\ \\
We shall provide a brief description\index{description} and relevant examples of rigid C*-tensor categories (RC*TC). All categories in this article are assumed to be essentially small (isomorphism classes of objects form a set) and $\vect$ enriched (not necessarily finite dimensional). We let ($\cat$, $\otimes,$ $\alpha,$  $\one,$ $\lambda,$ $\rho$) denote an arbitrary $\vect$ enriched semi-simple tensor category, where $-\otimes-:\cat\otimes\cat\longrightarrow\cat$ is a bilinear functor, $\alpha$ is the \textit{associator}, $\one$ is the \textit{monoidal/tensor unit}, and $\lambda$ and $\rho$ are the \textit{left and right unitors}, respectively. We refer to the monoidal category simply as $\cat$. Furthermore, \textbf{we assume the monoidal unit is \textit{simple}}; i.e. its endomorphism space is one dimensional: $\cat(\one)\cong \C.$ This assumption together with semi-simplicity imply that hom spaces are finite dimensional. (For a more detailed description of (monoidal) categories see \cite{EGNO15} and \cite{Mac71}.) All the categories in this paper are assumed to admit direct sums and contain sub-objects; i.e. they are \textit{Cauchy-complete}.

We further endow $\cat$ with more structure:
\begin{definition}\label{daggerCat}\cite{MR3663592}, \cite{GLR85}
   We say that {\upshape $\cat$} is a \textbf{dagger category} if and only if for each $a,b\in$ {\upshape $\cat$} there is an anti-linear map 
   $$\dagger:\mathrm{\cat}(a\rightarrow b)\rightarrow\cat(b\rightarrow a)$$ 
   such that
   \begin{itemize}
       \item The map $\dagger$ is an involution; i.e. for each $a,b\in\cat$ and every 
       $f\in \cat(a\rightarrow b)$ we have $(f^\dagger)^\dagger = f.$
       \item For composable morphisms $f$ and $g$ in $\cat,$ we have $(f\circ g)^\dagger = g^\dagger\circ f^\dagger.$ 
    \item Moreover, we have the identity $(f\otimes g)^\dagger = f^\dagger\otimes g^\dagger.$
   \end{itemize}
   Furthermore, we say that $f\in\cat(a\rightarrow b)$ is unitary if it is invertible with $f^{-1} = f^\dagger.$
\end{definition}

Having a dagger structure on $\cat$ allows us to introduce important analytical properties, as one could ask whether endomorphism $*$-algebras in $\cat$ are C*-algebras. We then introduce the following definition:

\begin{definition}\label{def::C*Cat}\cite{MR3663592}, \cite{GLR85}  
    We say a dagger category is a \textbf{C*-category} if and only if the following conditions hold for each $a,b\in\cat$:
    \begin{itemize}
        \item For every morphism $f\in\cat(a\rightarrow b)$ there exists an endomorphism $g\in \cat(a)$ such that $f^\dagger\circ f = g^\dagger\circ g.$
        \item The map $||\cdot||: \cat(a\rightarrow b)\longrightarrow [0,\infty]$ defined by
        $$||g||^2 := \text{sup}\big\{|\lambda|\geq 0\ |\ (g^\dagger\circ g -\lambda\cdot\id_a) \not\in \text{GL}(\cat(a))\big\}$$ defines a (submultiplicative) C*-norm. This is, the map defines a submultiplicative norm and for each morphism $f$ as above we have $||f^\dagger\circ f || = ||f||^2.$
    \end{itemize}
    Notice that being C* is a property of a dagger category and not extra structure.
\end{definition}

\begin{remark}
    Theorem 1.2 in \cite{HaYam00} asserts that any C*-tensor category is equivalent to a strict one. (In a strict category the associator and unitors are all identities.)  We can therefore \textbf{assume that} $\cat$ \textbf{is strict}, with no loss of generality.
    In lights of this result, we will often omit the associator and the unitors in our computations.
\end{remark}

\begin{definition}\label{rigidCat}\cite{EGNO15}  
	   We say that $\cat$ is a \textbf{rigid category} if each object has a dual and a predual. This is, if for each $c\in\cat$ there is an object $c^\vee,$ the dual of $c$, and evaluation and coevaluation maps 
	   $$\ev{c}:c^{\vee}\otimes c \longrightarrow \one \text{ and } \coev{c}:   \one\longrightarrow c\otimes c^{\vee},$$
	   which we draw as follows:
	   \begin{equation}
	        \begin{tikzpicture}[scale=1/10 pt, thick,baseline={([yshift=-\the\dimexpr\fontdimen22\textfont2\relax] current bounding box.center)}] 
	        \node at (-42,8){$\ev{c} = \ $};
	            \draw (-20,0) arc(0:180:6);
	            \draw[dashed] (-26,6) -- (-26,16);
	            \node at (-35,2){$c^\vee$};
	            \node at (-17,2){$c$};
	            \node at (-24,14){$\one$};
	            
	            \node at (-8,9){$\text{ and }$};
	        \node at (8,8){$\coev{c}= \ $};
	            \draw (20,16) arc (180:360:6);
	            \draw[dashed] (26,10) -- (26,0);
	            \node at (18,15){$c$};
	            \node at (34,15){$c^{\vee}$};
	            \node at (28,2){$\one$};
	       \end{tikzpicture}
	   \end{equation}
	   (From this point on, we relax the graphical calculus notation, so that we will not necessarily indicate the presence of the unit
	   object $\one$ nor the dashed strand $\id_{\one}$.) For each object $c\in\cat$, the evaluation and coevaluation maps satisfy \textbf{the Zig-Zag equations} (a.k.a. \textit{the conjugate equations} in \cite{LR97}, Sec. 2):
	   \begin{equation}\label{eqn::zz}
		        \begin{tikzpicture}[scale=1/10 pt, thick,baseline={([yshift=-\the\dimexpr\fontdimen22\textfont2\relax] current bounding box.center)}] 
			            \draw (-60,-10) -- (-60,10);
			            \node at (-63,-8){$c^\vee$};
			            \node at (-63,8){$c^{\vee}$};
			        \node at (-48,0){$=$};
			            \draw (-40,-10) -- (-40,0);
			            \node at (-43,-8){$c^{\vee}$};
			            \draw (-30,0) arc (0:180:5);
			            \node at (-32,0){$c$};
			            \draw (-30,0) arc(180:360:5);
			            \draw (-20,0) -- (-20,10);
			            \node at (-17,8){$c^{\vee}$};
			            
			            \node at (0,0){$\text{ and }$};
			            
			            \draw (60,-10) -- (60,10);
			            \node at (58,-8){$c$};
			            \node at (58,8){$c$};
			        \node at (48,0){$=$};
			            \draw (40,-10) -- (40,0);
			            \node at (17,8){$c$};
			            \draw (40,0) arc (0:180:5);
			            \node at (27,0){$c^{\vee}$};
			            \draw (20,0) arc(180:360:5);
			            \draw (20,0) -- (20,10);
			            \node at (37,-8){$c$};
		        \end{tikzpicture}
	   \end{equation}
	   Which more succinctly can be equivalently stated as
	   \begin{equation*}
		       \id_{c^\vee} = (\ev{c}\otimes\id_{c^\vee})\circ (\id_{c^\vee}\otimes\coev{c})\\
		       \text{ and }\ \
		       \id_{c} = (\id_{c}\otimes\ev{c})\circ (\coev{c}\otimes\id_{c}).
	   \end{equation*}
	\noindent We moreover require that there exists a predual object to $c$, denoted by $c_\vee,$ such that $(c_v)^\vee \cong c.$	
	   
	  For arbitrary $c,b\in\cat$, \textit{the dual of a map} $f\in \cat(a\rightarrow b)$ can be computed by composing with the evaluation and coevaluation maps as follows:
	   \begin{equation}\label{eqn::dualmap}
	        \begin{tikzpicture}[scale=1/15 pt, thick,baseline={([yshift=-\the\dimexpr\fontdimen22\textfont2\relax] current bounding box.center)}] 
		        \node at (-33,0){$f^\vee $};
	       \end{tikzpicture}
	       =
	       \begin{tikzpicture}[scale=1/15 pt, thick,baseline={([yshift=-\the\dimexpr\fontdimen22\textfont2\relax] current bounding box.center)}] 
		            \node at (-20,0){$f$};
		            \draw[rounded corners, ultra thick](-25,-5) rectangle (-15,5);
		            \draw[ultra thick] (-20,7)--(-20,5);
		            \draw[ultra thick] (-20,7)arc(0:180:5);
		            \draw[ultra thick] (-30,7) -- (-30,-15);
		            \node at (-34,-13){$b^\vee$};
		            \draw (-20,-7)--(-20,-5);
		            
		            \draw (-20,-7)arc(180:360:5);
		            \draw (-10,-7)--(-10,15);
		            \node at (-6,13){$a^\vee$};
	        \end{tikzpicture}
	   \end{equation}
	   where $f^\vee\in\cat (b^\vee\rightarrow a^\vee)$ is given by $(\ev{b}\otimes \id_{a^\vee})\circ(\id_{b^\vee}\otimes f\otimes \id_{a^\vee})\circ(\id_{b^\vee}\otimes \coev{a}).$ As a result, these choices of duals for objects, can be arranged into a strong-monoidal \textbf{dual functor} $$(\bullet)^\vee:\cat\longrightarrow\cat^{\rm{mop}},$$ where $\cat^{\rm{mop}}$ is $\cat$ considered with reversed arrows and tensor product. The associated \textit{tensorator} is given by the natural isomorphism whose components are given by
	    \begin{equation}
	        \nu_{a,b}
	        =
	        \begin{tikzpicture}[scale=1/15 pt, thick,baseline={([yshift=-\the\dimexpr\fontdimen22\textfont2\relax] current bounding box.center)}] 
	            \node at (-24,-8){$b^\vee$};
	            \node at (-7,-8){$a^\vee$};
	            \draw (-20,-10) -- (-20,0) arc (180:0:10) arc (180:360:5) -- (10,10);
	            \draw (-11,-10) -- (-11,0) arc (180:0:5) arc(180:360:6) -- (11,10);
	            \node at (21,8){$(a\otimes b)^\vee$};
	        \end{tikzpicture}
	        =
	        (\ev{b}\otimes\id_{(a\otimes b)^\vee}) \circ (\id_{b^\vee}\otimes \ev{a}\otimes \id_{(a\otimes b)^\vee}) \circ (\id_{b^\vee}\otimes \id_{a^\vee}\otimes \coev{a\otimes b}).
	    \end{equation}
\end{definition}

\begin{remark}\label{rmk::BalancedDual}
    Neither a dual functor nor its tensorator need to be unitary; i.e., $(f^\dagger)^\vee \neq (f^\vee)^\dagger$, nor $\nu^{-1} = \nu^{\dagger}$. However, for a RC*TC with simple unit, there exists a \textbf{balanced dual} for each object (see Lemma 3.9 in \cite{MR2091457} or Prop. 3.24 in \cite{Pen18}); i.e., for each $c\in \cat,$ there exists a \textit{choice} of dual $(\bar {c},\ \ev c,\coev c)$ for which the Zig-Zag equations hold and moreover satisfies \textbf{the balancing condition}: for an arbitrary endomorphism $f\in\cat(a\rightarrow a),$ its \textbf{left and right traces} match; i.e
       \begin{equation}\label{eqn::balancing}
       \ev{a}\circ(\id_{\overline{a}}\otimes f)\circ{\ev{a}}^\dagger 
       =
           \begin{tikzpicture}[scale=1/15 pt, thick,baseline={([yshift=-\the\dimexpr\fontdimen22\textfont2\relax] current bounding box.center)}] 
                \draw[rounded corners, ultra thick] (-5,-5) rectangle (5,5);
                \node at(0,0){$f$};
                \draw[ultra thick] (0,5) arc(0:180:5);
                    \node at (-13,5){\footnotesize$\overline a$};
                \draw[ultra thick] (-10,5) -- (-10,-5);
                \draw[ultra thick] (-10,-5) arc(180:360:5);
           \end{tikzpicture}
        =
            \begin{tikzpicture}[scale=1/15 pt,             thick,baseline={([yshift=-\the\dimexpr\fontdimen22\textfont2\relax] current bounding box.center)}] 
                \draw[rounded corners, ultra thick](-5,-5) rectangle (5,5);
                    \node at (0,0){$f$};
                    \draw[ultra thick] (10,5) -- (10,-5);
                    \draw[ultra thick] (10,5) arc(0:180:5);
                        \node at(13,5){\footnotesize$\overline a$};
                    \draw[ultra thick] (0,-5) arc(180:360:5);
            \end{tikzpicture}
        =
        {\coev{a}}^\dagger\circ(f\otimes \id_{\overline{a}})\circ \coev{a}
            \in \C.
       \end{equation}
   	This choice of balanced dual assembles into a \textit{unitary dual functor} for which for every morphism $f$ we obtain ${f^\dagger}^\vee = {f^\vee}^\dagger$, and $\nu$ is unitary.
	 We remark that this choice of dual functor is unique up to a unique natural isomorphism.
   
   	Associated to the \textit{balanced dual functor} there is a \textbf{canonical pivotal structure} $$\varphi:\id_\cat\Rightarrow (\bullet)^{\vee\vee},$$ which is a \textbf{natural monoidal unitary} \textit{spherical} (i.e. Eqn. \ref{eqn::balancing} holds) isomorphism whose components are given by
    \begin{equation}
    (\coev{c}^\dagger \otimes\id_{{c^\vee}^\vee})\circ (\id_c \otimes\coev{c^\vee})
    =
        \begin{tikzpicture}[scale=1/15 pt,             thick,baseline={([yshift=-\the\dimexpr\fontdimen22\textfont2\relax] current bounding box.center)}] 
        \node at (-12,-8){\footnotesize$c$};
        \node at (-3,-2){\footnotesize$c^\vee$};
        \node at (15,8){\footnotesize${c^\vee}^\vee$};
              \draw (-10,-10) -- (-10,0) arc(180:0:5) arc(180:360:5) --(10,10);   
        \end{tikzpicture}
    =
         \varphi_c
    =
        \begin{tikzpicture}[scale=1/15 pt,             thick,baseline={([yshift=-\the\dimexpr\fontdimen22\textfont2\relax] current bounding box.center)}] 
            \node at (13,-8){\footnotesize$c$};
            \node at (3,-2){\footnotesize$c^\vee$};
            \node at (-15,8){\footnotesize${c^\vee}^\vee$};   
            \draw (-10,10) -- (-10,0) arc (180:360:5) arc (180:0:5) --(10,-10);
        \end{tikzpicture}
    =  
    (\id_{{c^\vee}^\vee}\otimes \ev{c})\circ(\ev{c^\vee}^\dagger\otimes\id_c).
    \end{equation}
\end{remark}  
   
\begin{remark}
	    Using the unitary balanced dual functor in $\cat$, the naturality of $\varphi$ and the canonical spherical pivotal structure, one can show that the mirrored dual of a morphism matches Eqn. \ref{eqn::dualmap}: 
	    (see \cite{BHP12}, pp. 8-10)
	\begin{equation}\label{eqn::dualmap}
		        \begin{tikzpicture}[scale=1/15 pt, thick,baseline={([yshift=-\the\dimexpr\fontdimen22\textfont2\relax] current bounding box.center)}] 
			            \node at (-20,0){$f$};
			            \draw[rounded corners, ultra thick](-25,-5) rectangle (-15,5);
			            \draw[ultra thick] (-20,7)--(-20,5);
			            \draw[ultra thick] (-20,7)arc(0:180:5);
			            \draw[ultra thick] (-30,7) -- (-30,-15);
			            \node at (-34,-13){$b^\vee$};
			            \draw (-20,-7)--(-20,-5);
			            
			            \draw (-20,-7)arc(180:360:5);
			            \draw (-10,-7)--(-10,15);
			            \node at (-6,13){$a^\vee$};
		       \end{tikzpicture}
		        =
		        f^\vee 
		        = 
		        \begin{tikzpicture}[scale=1/15 pt, thick,baseline={([yshift=-\the\dimexpr\fontdimen22\textfont2\relax] current bounding box.center)}] 
			            \node at (-20,0){$f$};
			            \draw[rounded corners, ultra thick](-25,-5) rectangle (-15,5);
			            \draw[ultra thick] (-20,7)--(-20,5);
			            \draw[ultra thick] (-10,7)arc(0:180:5);
			            \draw[ultra thick] (-10,7) -- (-10,-15);
			            \node at (-6,-13){$b^\vee$};
			            \draw (-20,-7)--(-20,-5);
			            
			            \draw (-30,-7)arc(180:360:5);
			            \draw (-30,-7)--(-30,15);
			            \node at (-34,13){$a^\vee$};
		        \end{tikzpicture}
		        = (\id_{a^\vee}\otimes \coev{b}^\dagger) \circ (\id_{a^\vee}\otimes f\otimes \id_{b^\vee}) \circ (\ev{a}^\dagger \otimes \id_{b^\vee}) .
	    \end{equation}
	
	 Notice that, in light of the naturality of $\varphi,$ \textit{rotation by} $2\pi$ in $\cat$ satisfies the identity ${f^\vee}^\vee\circ\varphi_a = \varphi_b\circ f,$ for an arbitrary $f\in\cat(a\rightarrow b).$
\end{remark}
   
\begin{definition}\label{def::involution}\cite{JP17}, \cite{Pen18}, \cite{MR3663592} 
	    Using the dagger structure $\dagger$ and the unique unitary (balanced) dual functor from Remark \ref{rmk::BalancedDual} $(\bullet)^\vee:\cat \longrightarrow \cat^{\rm{mop}},$ 
	we can define an \textbf{involutive structure} $((\overline{\, \cdot \,}, \nu), \varphi, r)$ on $\cat$ as follows:
	\begin{align*}
	        \overline{\, \cdot \,}:\cat \longrightarrow\ &\cat^{\rm{mp}}\\
	        c\mapsto \bar c\ (= c^\vee), \hspace{5mm} f\mapsto &\overline f := (f^\dagger)^\vee\ (= (f^\vee)^\dagger).
	\end{align*}
	Here, $\cat^{\rm{mp}}$ is the monoidal category $\cat$ considered with the reversed tensor product. Notice that if $f\in\cat(a\rightarrow b),$ then $\overline f\in \cat(\overline a \rightarrow \overline b).$ 
	The \textbf{involution} $\overline{\, \cdot \,}$ is a conjugate-linear strong-monoidal functor, which comes equipped with the canonical unitary pivotal structure $\{\varphi_c: c\xrightarrow[]{\sim} \overline{\overline c}\}_{c\in \cat}$ 
	satisfying $\varphi_{\overline c} = \overline{\varphi_c}.$ The monoidal structure is given by natural unitaries $\{\nu_{a,b}: \overline a \otimes \overline b\xrightarrow[]{\sim} \overline{a\otimes b}\},$
	 and an isomorphism $r:\one\xrightarrow[]{\sim} \overline{\one}$ given by $\lambda_{\overline {\one}} \circ \coev{\one}.$ Here $\lambda$ is the \textit{left unitor}. 
	We limit ourselves to remark there are associativity and unitality axioms these maps must satisfy, and a more detailed view can be found in \cite[Definition 2.4]{JP17}.
	    
	    A tensor category with an involution is called \textbf{involutive}. And a tensor category is called \textbf{bi-involutive} (\cite[Definition 2.3]{MR3663592} and \cite[Definition 3.35]{Pen18}) if it is a dagger category with an involution and moreover, 
	the involution $\overline{\, \cdot \,}$ is a dagger functor (i.e. for every $f\in\cat(a\rightarrow b)$, we have $\overline{f^\dagger} = {\overline f}^\dagger$) with unitary structure isomorphisms $\varphi,\nu$ and $r$.
\end{definition}

\begin{remark}\label{rmk::biinv}
	Notice that the \textbf{bi-involutive structure} described above is \textit{canonical}. (The interested reader can go to \cite{JP17}, Example 2.12 and the references therein for an expanded description of the structure isomorphisms.)
	It was shown in \cite{Pen18} that this bi-involutive structure is canonical and independent of the choice of unitary dual functor.
\end{remark}

The following definition will be of importance in the last section, so we include them for the convenience of the reader.

\begin{definition}\label{dfn::BiInvolutive} \cite{Pen18}
	    A \textbf{bi-involutive functor} between bi-involutive RC*TCs is a dagger tensor functor $(\ff, \mu): \cat
	    \longrightarrow \catD$ together with a natural unitary isomorphism $\chi_c: \ff(\overline c)\longrightarrow
	    \overline{\ff(c)}$ which is \textit{monoidal} with respect to $\mu, \nu^{\cat}$ and $\nu^{\catD},$ and \textit{involutive} with respect to $\varphi^{\cat}$ and $\varphi^{\catD}.$ 
	    These properties are captured by the following commutative diagrams:
	    \begin{equation}\label{eqn::Chi}
	        \begin{tikzcd}[column sep=large, row sep=large]
	            \ff(\overline a)\boxtimes\ff(\overline b) \arrow["\chi_a \boxtimes\chi_b"']{d} \arrow["\mu_{\overline a, \overline b}"]{r}
	            & \ff(\overline a\otimes\overline b)  \arrow["\ff(\nu_{a,b})"]{r}  & \ff(\overline{b\otimes a}) \arrow["\chi_{b,a}"]{d}\\
	            \overline{\ff(a)}\boxtimes\overline{\ff(b)} \arrow["\nu_{\ff(a),\ff(b)}"']{r}
	            & \overline{\ff(b)\boxtimes \ff(a)} \arrow["\mu_{a,b}"']{r} &\overline{\ff(b\otimes a)}
	        \end{tikzcd}
	        \hspace{2cm}
	        \begin{tikzcd}[column sep=large, row sep=large]
	            \ff(c) \arrow["\ff(\varphi_c)"]{r} \arrow["\varphi_{\ff(c)}"']{d} &\ff(\overline{\overline c}) \arrow["\chi_{\overline a}"]{d}\\
	             \overline{\overline{\ff(c)}}  &\overline{\ff(\overline c)} \arrow["\overline{\chi_c}"']{l}
	        \end{tikzcd}
	    \end{equation}
\end{definition}

We now introduce the tensor category of \textit{bimodules over a }$\rm{II}_1$ \textit{factor}. (See \cite{PopaII1} for basic facts about $\rm{II}_1$ factors.) Beyond its great significance in the theory of subfactors, this example is of central importance to us as it will serve as a grounds to construct a more general RC*TC of bimodules over a C*-algebra. 

\begin{example}\label{BifiniteBimodules}
	For this example we closely follow \cite{Bisch} and  
	\cite{JP17}. For a type $\rm{II}_1$ factor $(N,\tau),$ consider the (W*-)\textbf{tensor category of 
	bimodules over }$N$. The tensor product is given by the \textit{Connes fusion relative tensor product}, denoted $-\fuse{N}-$. By restricting to the full subcategory of bifinite bimodules $\Bim_{\mathsf{bf}}(N)$, we obtain a 
	RC*TC whose structure we now describe. 
	Let $\Omega\in \text{L}^2(N,\tau)$ be a cyclic vector in the GNS-construction. 
	Given $H\in \Bim_\mathsf{bf}(N)$, we say that $\xi\in H$ is (left/right-)-bounded if and only if the map $N\Omega\rightarrow H$ given by
	 $n\Omega\mapsto \xi\lhd n,$ extends to a bounded map $L_\xi:\Ltwo{N}\rightarrow H.$ The set of \textbf{bounded vectors} in
	 $H$ is denoted by $H^\circ,$ and defines a dense subset of $H$. Bounded vectors help us define an $N$-valued inner product; i.e., for $\eta,\xi\in H^\circ,$ we have
	 $\langle \eta\ |\ \xi\rangle_N := L_\eta^*L_\xi$. This product is indeed $N$-valued, as $L_\eta^*L_\xi$ commutes with the right $N$-action.
	
	\textbf{The dual} of $H\in\Bim_\mathsf{bf}(N)$ is given by $\overline {H},$ as a Hilbert space. To refer to an element in $\overline{H}$ we write 
	$\overline{\xi}\in \overline H$. Right and left $N$-actions are given by 
	$n\rhd\overline{\xi}\lhd m := \overline{m^*\rhd\xi\lhd n^*},$ and for $\overline \xi, \overline \eta\in \overline H$ 
	their inner product is given by $\langle \overline\xi\ |\ \overline\eta\rangle_N := \langle \eta\ |\ \xi \rangle_N.$
	 We are now able to define an evaluation map
	$$\ev{H}:\overline{H}\fuse{N}H\rightarrow \Ltwo{N} \text{ given on } \overline {H^\circ}\tens{N} H^\circ \text{ by} $$
	$$\overline{\eta}\boxtimes \xi\mapsto \langle \eta\ |\ \xi \rangle_N,$$ which is $N$-bilinear and bounded.
	
	To define coevaluation maps, we need the notion of a (finite) right $N$-basis, which we introduce in a more general context in Definition
	 \ref{dfn::basis}. By \cite{MR561983}, such a finite set exists and we denote it by $\set{\beta}\subset H^\circ.$ We then have that for every
	 $\xi\in H^\circ,$  $\xi = \sum \beta\lhd\langle \beta\ |\ \xi\rangle_N.$ For a chosen basis $\set{\beta}\subset H^\circ,$ we then have the map
	$$\coev{H}: \Ltwo{N} \longrightarrow  H\fuse{N}\overline{H} \text{ given by}$$
	   $$ n\Omega \mapsto \sum (n\rhd\beta)\boxtimes\overline\beta,$$
	which can be seen not to depend on the choice of basis. The maps $\ev{H}$ and $\coev{H}$ satisfy the Zig-Zag Equations (\ref{eqn::zz}). Thus
	 $\Bim_{\mathsf{bf}}(N)$ is indeed a RC*TC, where our choice of evaluation and coevaluation maps is usually referred to as the \textbf{tracial evaluation
	 and coevaluation}. One may need to renormalize these maps on irreducible bimodules so that they satisfy the balancing condition
	 described in Equation \ref{eqn::balancing}. We further restrict to the full rigid C* tensor subcategory of \textit{spherical/extremal} bimodules, 
	denoted $\Bim_{\mathsf{bf}}^{\mathsf{sp}}(N),$ so that the tracial dual matches the canonical unitary dual functor.
\end{example}

\subsection{Hilbert C*-bimodules}\label{sec::Bimodules}  \ \\
We recall some useful facts about Hilbert C*-bimodules that we will use in the following sections. Throughout this section, we assume that $C$ and $D$ are unital C*-algebras. 
\begin{definition}\label{dfn::bimodule}\cite{MR1624182, HaPeI}\\
   A $C$-$D$ \textbf{Hilbert C*-bimodule} $\Y$ is a vector space together with commuting left and right actions denoted by $-\rhd-: C\times\Y\rightarrow \Y$ and $-\lhd-: \Y\times D\rightarrow \Y$. Moreover, we have the following structures:
   \begin{itemize}
       \item A $C$-valued form $_C\langle-,\ \cdot\ \rangle$ which is $C$-linear on the left and conjugate linear on the right, and
       \item a $D$-valued  form $\langle\ \cdot\ | - \rangle_D$ which is $D$-linear on the right and conjugate linear on the left.
       \item Compatibility of inner products with the involution: $_C\langle\xi,\eta \rangle = (_C\langle\eta,\xi \rangle)^*$, and $\langle\xi\ |\ \eta \rangle_D = (\langle\eta\ |\ \xi \rangle_D)^*$;
   \end{itemize}
   These forms provide canonical norms defined as follows:\\ for each $y\in \Y,$ $$_C||y||^2 :=\ ||_C\langle y, y\rangle||_C,  \text{ and } ||y||_D^2 :=   || \langle y\ |\  y\rangle_D||_D.$$
   We also require the following identities and properties to hold for every $\xi,\eta\in\Y$:
   \begin{itemize}
        \item The two norms above are complete and equivalent;
	  \item The two forms are \textit{positive definite}; i.e. $_C\langle \xi,\ \xi\rangle \geq 0$ on $\Y$ and $_C\langle \xi,\ \xi\rangle = 0$ if and only if $\xi = 0$, and similarly for the $D$-valued form.
       \item For each $c\in C$ and each $d\in D$, we have $_C\langle \xi\lhd d,\eta \rangle =\ _C\langle\xi,\eta\lhd d^* \rangle$, and $\langle c\rhd\xi\ |\ \eta \rangle_D = \langle\xi\ |\ c^*\rhd\eta \rangle_D.$
   \end{itemize}
 For Hilbert $C^*$-bimodules $_C\Y_D$ and $_C\mathcal {Z}_D$, denote by $\mathcal B(\Y, \mathcal Z)$ the Banach space of bounded $C$-$D$ bilinear operators. We denote the set of \textbf{left-adjointable operators} by $^*\mathcal B(\Y,\mathcal Z)$, and $\mathcal B^*(\Y, \mathcal Z)$ denotes the set of \textbf{right-adjointable operators.} We notice that under this notation, the last bullet points amounts to saying that there is an embedding of the left $C$-action into the Banach space of right-adjointable operators and also that there exists an embedding of the right $D$-action into the space of left-adjointable operators. 
\end{definition}

Notice that the left adjoint of a given operator need not match its right adjoint. Often times, it will be sufficient for our purposes to restrict to those operators for which these notions match, denoted \textbf{bi-adjointable operators.} In the following definitions we seek for sufficient conditions for our bimodules to form a RC*TC. We will closely follow \cite{MR1624182} for the description of such categorical/analytic structures and properties.

\begin{definition}\label{dfn::basis}\cite{MR1624182}  
   Given a Hilbert $C^*$-bimodule $_C\Y_D$, we say that $\{u_i\}_{i = 1}^{m}\subset \Y$ is a \textbf{right $D$-basis} if and only if for each $\xi\in\Y$ we have
                    $$\xi = \sum_{i=1}^{m} u_i \lhd\langle u_i\ |\ \xi\rangle_D.$$
    Similarly, we say $\{v_j\}_{j = 1}^{n}\subset \Y$ is a \textbf{left $C$-basis} if and only if for each $\xi\in\Y$ we have
                    $$\xi = \sum_{j=1}^{n}\ _C\langle \xi, v_j\rangle\rhd v_j.$$
\end{definition}

\noindent If a bimodule $\Y$ has both a left and a right basis, we say it is \textbf{finitely generated projective} (fgp). 

\begin{definition}\label{index} \cite[Definitions 1.14 and 1.16]{MR1624182}   
  	 Define the \textbf{right index} the and \textbf{left index} of a fgp $C$-$D$ Hilbert C*-bimodule $\Y$ as 
	$$
		\text{r-Ind}(\Y):= \sum_{i =1}^{m}\ _C\langle u_i,u_i\rangle, \text{ and } \text{l-Ind}(\Y) = \sum_{j=1}^{n} \langle v_j\ |\ v_j\rangle_D.
	$$
	By Proposition 1.13 in \cite{MR1624182}, these sums do not depend on the choices of bases, and moreover that $\text{r-Ind}\in \text{Z}(C)^+$
	 and $\text{l-Ind}\in\text{Z}(D)^+,$ define positive central elements in their respective C*-algebras. Thus, if either $C$ or $D$ have trivial center, the
	 corresponding index is then a positive real number. Moreover, we define \textbf{the index} of $\Y$ as 
	$$\text{Ind}(\Y) = \text{l-Ind}(\Y) \cdot \text{r-Ind}(\Y).$$
		We say the bimodule $\Y$ is \textbf{normalized} if $\text{r-Ind}(\Y) = \text{l-Ind}(\Y).$ By \cite[Lemma 1.15 ]{MR1624182}
	 fgp bimodules over C*-algebras with trivial centers can always be normalized and therefore we will only consider these, without loss of generality.
\end{definition}

We shall next introduce another property for an fgp bimodule $_C\Y_D$ that will become of crucial importance when defining a tensor structure for bimodules, obtaining a multiplicative notion of (Jones') index.

\begin{definition}\label{minimal}  \cite[Proposition 3.3 and Definition 3.5]{MR1624182}   
	  Let $_C\Y_D$ be as above, and further assume that Z$(C) = \C =\ $Z$(D)$. 
	For $K := \frac{\text{r-Ind}(\Y)}{\text{l-Ind}(|y)} > 0$ and each $T\in \text{End}(_C\Y_D)$ we have
		      $$\sum_i\ _C\langle Tu_i , u_i\rangle = K\cdot \sum_j \langle v_j| Tv_j\rangle_D  \in \C,$$
	 we then say that $\Y$ is a \textbf{minimal bimodule}.\\
\end{definition}

\begin{example}\label{HappyBims}	
	We now describe a RC*TC, which is a C*-algebraic analog of the category $\Bim_\mathsf{bf}^{\mathsf{sp}}(N)$ of spherical bifinite bimodules over a
	 $\rm{II}_1$-factor $N$, outlined earlier in Example \ref{BifiniteBimodules}. In describing this example, we rely heavily on the results 
	in \cite{MR1624182}. Let $B$ be a fixed unital C*-algebra with trivial center. Consider the category of fgp normalized and minimal
	 Hilbert C*-bimodules over $B$, denoted by $\Bim_\mathsf{fgp}(B).$ Neither minimality nor being normalized are \textit{extra properties} 
	of a bimodule, since one can always \textit{remetrize} to simultaneously force minimality and normality. 
	For further details, see \cite[Lemma 3.6]{MR1624182}. Since there is no loss of generality by doing so, from this point on, 
	we will only consider \textbf{fgp normalized and minimal bimodules}. The greater advantage of reducing the category to 
	normalized and minimal bimodules is that the index is well behaved with respect to sums and products in the category, 
	which we shall next define, and that we can structure $\Bim_\mathsf{fgp}(B)$ as a RC*TC.

	There is a canonical way to direct sum objects in $\Bim_\mathsf{fgp}(B)$, 
	where we emphasize that the resulting object belongs to $\Bim_{\mathsf{fgp}}(B),$ and the index behaves additively
	 (Lemma 3.9 \cite{MR1624182}). The tensor structure is given by $B$-\textbf{fusion} $-\fuse{B}-$, where if
	 $\Y,\mathcal Z\in \Bim_{\mathsf{fgp}}(B)$ then $\Y\fuse{B}\mathcal Z \in \Bim_{fgp}(B)$ without completion
	 (Proposition 1.23 in \cite{MR1624182}). We remark that Ind$(\Y\fuse{B}\mathcal Z)$ = Ind$(\Y)\cdot$ Ind$(\mathcal{Z})$ holds true on
	 $\Bim_{\mathsf{fgp}}$ (Proposition 1.30, \cite{MR1624182}),  and that fusion is distributive with respect to sums (Lemma 1.24, \cite{MR1624182}).

	For $\Y,\mathcal Z \in \Bim_{\mathsf{fgp}}(B),$ by \cite[Lemma 1.10]{MR1624182}, any right (resp. left) $B$-module map
	 $T: \Y\rightarrow \mathcal Z$ is automatically a finite rank operator, is bounded with respect to the right (resp. left) $B$-norm
	 (and hence both norms) and is	 \textit{adjointable with respect to the right (resp. left) inner product},
	 with right-adjoint $T^*$ (resp. left-adjoint $^*T$). 
	Moreover, minimality together with Proposition 3.14 in \cite{MR1624182} ensure that if $T$ is also
	 \textit{adjointable with respect to the left inner product}, it satisfies $^*T = T^*.$ 
	 Therefore $\Bim_{\mathsf{fgp}}(B)(\Y\rightarrow \mathcal Z)$ consists of \textbf{bi-adjointable bounded} 	
	$B$-bimodule maps, with common adjoint denoted $T^*$. This defines a \textbf{dagger structure} on $\Bim_{\mathsf{fgp}}(B)$.


	We now define the \textbf{conjugate bimodule} of $\Y$ denoted $\overline \Y$, where $\Y = \overline \Y$ as sets. 
	The bimodule structure is as follows: if $\overline{y}, \overline {y'}\in\overline{\Y}$ and $a, b\in B$ the left and right actions are given by
	 $a\rhd \overline{y}\lhd b := \overline{b^* \rhd y \lhd a^*},$  $_B\langle \overline{y}, \overline{y'}\rangle := \langle y\ |\ y' \rangle_B$ 
	and $\langle \overline{y}\ |\ \overline{y'}\rangle_B :=\ _B\langle y,\ y'\rangle.$ Clearly, $\overline \Y\in\Bim_{\mathsf{fgp}}(B)$. We are now in
	 position to define \textbf{evaluation and coevaluation}: for a given finite right $B$-basis $\{u_i \}_{i=1}^m$ we have maps
	\begin{align*}
    		\ev{\Y}:&\overline{\Y}\fuse{B}\Y\longrightarrow\ _BB_B, \hspace{20 mm} \coev{\Y}: {_BB_B}\longrightarrow \Y\fuse{B}\overline{\Y},\\
	    &\overline{\eta}\boxtimes\xi\mapsto  \langle\eta\ |\ \xi \rangle_B \hspace{30mm} b\mapsto b\rhd\sum_{i=1}^{m} u_i\boxtimes \overline{u_i}
	\end{align*}
	One must check these maps belong to our category. For example, $\ev{\Y}$ is manifestly right $B$-linear so it is a bounded right-adjointable
	 $B$-bimodule map. By direct computation, $\ev{\Y}^*(b) = (\sum_{j=1}^{n}\overline{v_j}\boxtimes v_j) \lhd b$ and moreover
	$^*\ev{\Y}(b) = b\rhd\sum_{j=1}^{n}\overline{v_j}\boxtimes v_j,$ where $\{v_j\}_{j=1}^n$ 
	is a finite left $B$-basis and $b\in B$ is arbitrary. Therefore, $ ^*\ev{\Y}$ is indeed the left-adjoint of the evaluation map. Then, necessarily, 
	$^*\ev{\Y}(b) = b\rhd\sum_{j=1}^{n}\overline{v_j}\boxtimes v_j = (\sum_{j=1}^{n}\overline{v_j}\boxtimes v_j)\lhd b = \ev{\Y}^*,$
	 showing that $\ev{\Y}\in\Bim_{\mathsf{fgp}}(B)(\overline{\Y}\boxtimes\Y\rightarrow B)$. Similarly for
	 $\coev{\Y}$, where $\coev{\Y}^*(y\boxtimes\overline{y'}) = {_B\langle y,\ y' \rangle} =\ ^*\coev{\Y}(y\boxtimes\overline{y'}).$
	We now define \textbf{the dual module} $\Y^\vee := \overline{\Y}$ and \textbf{the dual of a morphism} $T^\vee$
	is given by the canonical dual. By Proposition 4.1 in \cite{MR1624182}, under these structural maps, the category of fgp $B$-bimodules
	 $\Bim_{\mathsf{fgp}}(B)$ is a rigid category. In fact, since we restricted to normalized bimodules, $\Bim_{\mathsf{fgp}}(B)$ has the C*-property, 
	by Corollary 4.3 in \cite{MR1624182}. 
	
	This choice of dual functor corresponds to the balanced unitary dual. Indeed the canonical spherical structure on $\Bim_{\textsf{fgp}}(B)$
	 defines an honest pivotal structure so, by Proposition 3.9 in \cite{Pen18}, the result follows. (Remarkably, that $^*T = T^*$ 
	for every biadjointable map is an equivalent condition to the naturality of the pivotal structure.) Thus, we can consider 
	the canonical \textbf{bi-involutive structure} given on objects by conjugation $\overline \Y$ and a morphism
	 $T:\Y\rightarrow \mathcal Z$, its
	 conjugate $\overline T : \overline \Y\rightarrow \overline{\mathcal Z}$ is given by $\overline T := (T^*)^\vee = (T^\vee)^*$
	 and satisfies $\overline T(\overline y) = \overline{T(y)}.$
\end{example}

\section{Diagramatic Algebras}\label{sec::diagAlgs}
\subsection{GJS construction}\label{section:GJS}\ \\
In this section we will assume that $\cat$ is a fixed essentially small RC*TC. We start with an auxiliary lemma, which helps us find a special object $x\in \cat.$ From the morphisms between tensor powers of this object we will construct certain diagrammatic $*$-algebras, and an \textit{ambient} C*-algebra, denoted $B_{\infty}$ (Definition \ref{dfn::BInfty}), which we will heavily rely on. The corners $B_n$ of $B_\infty$ will be separable simple unital C*-algebras with unique traces. Of particular importance will be the corner $B_0$ (Notation \ref{notation::corners}), as we will use it to construct a concrete rigid C*-tensor subcategory of $\Bim_{\mathsf{fgp}}(B_0),$ (Example \ref{HappyBims}) where we \textit{fully and faithfully} represent $\cat.$ 

\begin{notation}\label{notation::delta}
    For an object $c\in\cat,$ we define $\delta_c$ to be the \textbf{quantum dimension} of $c$. This is, $\delta_c := \coev{c}^*\circ\coev{c} = \ev{c}\circ\ev{c}^*,$ is the (complex) value of the $c$-\textit{loop}.
\end{notation}

\begin{definition}\label{dfn::SSD}
    	We say that an object $x\in \cat$ is \textbf{self-dual}  if and only if there exists a (unitary) isomorphism $\psi:x\longrightarrow\overline x$. 
	(This isomorphism can be chosen to be unitary by polar decomposition.)
	Moreover, we say that $x$ is \textbf{symmetrically self-dual} (ssd) (\cite{BHP12}, Definition 2.10) if and only if is self-dual and 
	$\ev{x}\circ(\psi\otimes \id_x) = \coev{x}^*\circ(\id_x \otimes \psi).$ (For more detailed treatments of self-dualities, 
	the reader can go to \cite[Theorem 3.4]{MR1010160}, \cite{LR97}, and \cite{HaPeI}.)
\end{definition}

Intuitively, in the graphical calculus, strings labeled by a ssd object need not be oriented. In the following lemma, we show there are always ssd objects.
We can \textbf{fix a ssd object} $x\in \cat$, since
\begin{lemma}\label{lemma::ssdObject}
    Every RC*TC has a ssd object $x,$ with $\delta_x >1$ and $x = \overline x$.
\end{lemma}
\begin{proof} 
	Let $\sigma\in\cat$ be an arbitrary object (not isomorphic to the tensor unit), and define $x := \sigma\oplus \overline \sigma.$ Clearly, $\delta_x > 1$. There is a canonical unitary map
	 $\psi:(\sigma\oplus\overline \sigma)\longrightarrow (\overline{\overline \sigma}\oplus \overline \sigma) \cong \overline{x}$ given by the diagonal matrix $\text{diag}(\varphi_\sigma, \id_{\overline \sigma}),$ 
	where $\varphi$ is the canonical unitary pivotal structure on $\cat.$ To establish symmetric self-duality, one simply observes that this condition is equivalent to $\ev{x} = \coev{\overline x}^\dagger\circ(\id_{\overline x}\otimes \varphi_x)$
	 and $\ev{\overline x} \circ(\varphi_x\otimes \id_{\overline{x}}),$ both of which hold since we are using the canonical pivotal structure. Finally, we can choose $x = \overline x$ by using $\psi$ to \textit{unitarily} 
	redefine the balanced dual functor, if need be.
\end{proof}

From this point on, we reserve the symbol $x$ to denote a fixed ssd object in $\cat$ as in the statement of the previous Lemma.

\begin{definition}\label{Grinfty} \cite{GJS11, BHP12, HaPeI, HaPeII}     Given  non-negative integers $l$, $r$, $b$,  we let $V_{b, l, r}$ be $\cat(x^{\otimes b}\rightarrow x^{\otimes l}\otimes x^{\otimes r}).$  We define $\Gr_\infty$ whose underlying vector space is given by the external \textit{algebraic} direct sum
    $$\Gr_\infty := \bigoplus_{l,r,b \geq 0}V_{b, l, r}.$$
    
    An element  $\xi\in V_{b, l, r}$ can be pictorially visualized as follows:
    \begin{equation}
        \begin{tikzpicture}[scale=1/15 pt, thick,baseline={([yshift=-\the\dimexpr\fontdimen22\textfont2\relax] current bounding box.center)}] 
            \draw[rounded corners, ultra thick] (-5,-5) rectangle (5,5);
                \node at (0,0){$\xi$};
                \node at (0,5){$\bullet$};
                \draw[ultra thick] (-3,5) -- (-3,15);
                \draw[ultra thick] (3,5) -- (3,15);
                \draw[ultra thick] (0,-5) -- (0,-15);
                    \node at (-6,12){$l$};
                    \node at (6,12){$r$};
                    \node at (2,-13){$b$};
        \end{tikzpicture}
        \text{   or    }
        \begin{tikzpicture}[scale=1/15 pt, thick,baseline={([yshift=-\the\dimexpr\fontdimen22\textfont2\relax] current bounding box.center)}]
            \draw[rounded corners, densely dashed, blue] (-9,-9) rectangle (9,9);
            \node at (0,9){$\bullet$};
                \draw[ultra thick, rounded corners] (-5,-5) rectangle (5,5);
                    \node at (0,0){$\xi$};
                    \draw[ultra thick] (3,5) -- (3,15);
                        \node at (5,12){$r$};
                    \draw[ultra thick] (-3,5) -- (-3,15);
                        \node at (-5,12){$l$};
                    \draw[ultra thick] (0,-5) -- (0,-15);
                        \node at (2,-12){$b$};
        \end{tikzpicture}.
    \end{equation}
    Here, the dot indicates how the lateral strings split on each side; i.e. there are $l$ strings \textit{on the left} and $r$ strings \textit{on the right}. We warn the reader that the same element $\xi\in V_{b, l, r}$ appears as many times as the number $(l + r)$ can be decomposed as a sum of two non-negative integers. 
\end{definition}

There is an \textbf{involution on} $\Gr_\infty$ denoted by $*$, which is determined by the canonical conjugate structure in $\cat$, given by its \textit{bi-involutive structure} (See Remark \ref{rmk::biinv}):\\  
$$*:V_{b, l, r} \longrightarrow V_{b, r, l}$$
\begin{equation}
    \begin{tikzpicture}[scale=1/15 pt, thick,baseline={([yshift=-\the\dimexpr\fontdimen22\textfont2\relax] current bounding box.center)}] 
        \draw[densely dashed,black,rounded corners, blue] (-22,-7) rectangle (-8, 7);
            \draw[rounded corners, ultra thick] (-20,-5) rectangle (-10,5);
            \node at (-15,7){$\bullet$};
                \node at (-15,0){$\xi$};
                    \draw[ultra thick] (-18,5) -- (-18,20);
                        \node at (-21,17){$l$};
                    \draw[ultra thick] (-12,5) -- (-12,20);
                        \node at (-9,17){$r$};
                    \draw[ultra thick] (-15,-5) -- (-15,-20);
                        \node at (-13,-18){$b$};
            \node at (-5,7){$^*$};
            \node at (0,0){$:=$};
            \draw[densely dashed,black,rounded corners, blue] (5,-17) rectangle (30, 12);
            \node at (25,12){$\bullet$}; 
                \draw[rounded corners, ultra thick] (10,-5) rectangle (20,5);
                    \node at (15,0){$\xi^{\dagger}$};
                        \draw[ultra thick] (15,5) -- (15,7);
                        \draw[ultra thick] (15,7) arc(0:180:4);
                        \draw[ultra thick] (7,7) -- (7,-21);
                            \node at (9,-19){$b$};
                        \draw[ultra thick] (12,-5) -- (12,-7);
                        \draw[ultra thick] (12,-7) arc (180:360:8);
                        \draw[ultra thick] (18,-7) arc (180:360:2);
                        \draw[ultra thick] (22,-7) -- (22,20);
                            \node at (30,16){$l$};                            
                        \draw[ultra thick] (18,-5) -- (18,-7);
                        \draw[ultra thick] (28,-7) -- (28,20);
                            \node at (19,16){$r$};
    \end{tikzpicture}
	= (\xi^\dagger)^\vee = (\xi^\vee)^\dagger.
\end{equation}
Notice how the roles of $l$ and $r$ switched.  For notational convenience, we will denote this diagram as 
\begin{equation*}
    \begin{tikzpicture}[scale=1/15 pt, thick,baseline={([yshift=-\the\dimexpr\fontdimen22\textfont2\relax] current bounding box.center)}] 
        \draw[rounded corners, ultra thick] (-5,-5) rectangle (5,5);
        \node at (0,0){$\xi^{*}$};
        \node at (0,5){$\bullet$};
            \draw[ultra thick] (-3,5) -- (-3,15);
                \node at (-6,12){$r$};
            \draw[ultra thick] (3,5) -- (3,15);
                \node at (6,12){$l$};
            \draw[ultra thick] (0,-5) -- (0,-15);
                \node at (2,-13){$b$};
    \end{tikzpicture}.
\end{equation*}

\begin{definition}
    Define $\Gr_{0,\infty}$ by 
    $$
    \Gr_{0,\infty} := \bigoplus_{l \geq 0, r \geq 0} V_{0, l, r}.
    $$ 
    The space $\Gr_{0,\infty}$ becomes an algebra when endowed with the multiplication
    \begin{equation}
        \begin{tikzpicture}[scale=1/11.5 pt, thick,baseline={([yshift=-\the\dimexpr\fontdimen22\textfont2\relax] current bounding box.center)}] 
            \draw[rounded corners, ultra thick] (-35,-5) rectangle (-25,5);
            \node at (-30,5){$\bullet$};
                \node at (-30,0){$\xi$};
                \draw[ultra thick] (-33,5) -- (-33,15);
                \draw[ultra thick] (-27,5) -- (-27,15);
                    \node at (-36,12){$l$};
                    \node at (-24,12){$r$};
                    
           \node at (-20,0){$\wedge$};
            
            \draw[rounded corners, ultra thick] (-15,-5) rectangle (-5,5);
            \node at (-10,5){$\bullet$};
                \node at (-10,0){$\eta$};
                \draw[ultra thick] (-13,5) -- (-13,15);
                \draw[ultra thick] (-6,5) -- (-6,15);
                    \node at (-16,12){$l'$};
                    \node at (-3,12){$r'$};
        \end{tikzpicture}        
         := \delta_{r = l'}\cdot
        \begin{tikzpicture}[scale=1/15 pt, thick,baseline={([yshift=-\the\dimexpr\fontdimen22\textfont2\relax] current bounding box.center)}] 
            \draw[densely dashed, rounded corners, blue] (10,-4) rectangle (40,16);
            \node at (25, 16){$\bullet$};
                \draw[rounded corners, ultra thick] (38,-2) rectangle (28,8);
                \node at (33,8){$\bullet$};
                    \node at (33,3){$\eta$};
                    \draw[ultra thick] (36,8) -- (36,23);
                    \node at (39,20){$r'$};
                    
                \draw[rounded corners, ultra thick] (22,-2) rectangle (12,8);
                \node at (17,8){$\bullet$};
                    \node at (17,3){$\xi$};
                    \draw[ultra thick] (14,8) -- (14,23);
                        \node at (11,20){$l$};
                        \draw[ultra thick] (31,8) arc (0:180:6);
        \end{tikzpicture}
    \end{equation}
    Notice that if $l+r = l'+r'$ with $l \neq l'$, this multiplication makes $V_{b,l,r}$ orthogonal to $V_{b',l',r'}$, and this is also true for elements that are counted more than once in $\Gr_\infty.$
    
    Define $\Gr_{0,n}$ by  
    \begin{equation}
        \Gr_{0,n} := \bigoplus_{n \geq l,  r \geq 0}V_{0, l, r}. 
    \end{equation} 
    
    
    Note that $\Gr_{0,n}$ is a finite dimensional unital $*$-subalgebra of $\Gr_{0,\infty}$ which we describe in the following Proposition whose proof is straightforward:
    \begin{proposition}\label{prop::AFalgs}
        For each $n, m\in\N$, we have an injective $*$-algebra homomorphism (given by \textbf{Frobenius Reciprocity} on $\cat$): 
        \begin{align}
            FR: \cat^{op}(x^{\otimes m}\rightarrow x^{\otimes n})\xrightarrow[]{\sim} \Gr_{0,\text{max}(m,n)}
            \hspace{5mm}\text{given by}\hspace{5mm}
            \begin{tikzpicture}[scale=1/15 pt, thick,baseline={([yshift=-\the\dimexpr\fontdimen22\textfont2\relax] current bounding box.center)}] 
                \draw[ultra thick] (0,5) -- (0,15);
                    \node at (3,12){$n$};
                \draw[rounded corners, ultra thick] (-5,-5) rectangle (5,5);
                    \node at (0,0){$f$};
                \draw[ultra thick] (0,-5) -- (0,-15);            
                    \node at (3,-12){$m$};
            \end{tikzpicture}
            \longmapsto
            \begin{tikzpicture}[scale=1/15 pt, thick,baseline={([yshift=-\the\dimexpr\fontdimen22\textfont2\relax] current bounding box.center)}] 
                \draw[densely dashed, rounded corners, blue] (-15,-14) rectangle(8,10);
                \node at (-5,10){$\bullet$};
                    \draw[ultra thick] (0,5) -- (0,15);
                    \draw[rounded corners, ultra thick] (-5,-5) rectangle (5,5);
                        \node at (0,0){$f$};
                    \draw[ultra thick] (0,-5) -- (0,-7);
                    \draw [ultra thick] (0,-7) arc(360:180:5);
                    \draw[ultra thick] (-10,-7) -- (-10,15);
            \end{tikzpicture}
        \end{align}
    When $n = m,$ the map $FR$ becomes an injective $*$-algebra isomorphism.
    \end{proposition}
    Therefore, using Frobenius Reciprocity, we endow $\Gr_{0,n}$ with a norm, giving it the structure of a unital C*-algebra. Note that the inclusion $\Gr_{0, n} \hookrightarrow \Gr_{0, n+1}$ is nonunital but preserves minimal projections. As such, $\Gr_{0,\infty}$ (being the inductive limit of the $\Gr_{0,n}$ ) can be completed into an AF C$^*$-algebra which we shall denote as $A_{\infty}$ with multiplication still denoted by $ \wedge.$ This is, as a vector space
    \begin{equation}\label{eqn::GroundAlg}
        A_\infty := \overline{\lim_n \Gr_{0,n}}.
    \end{equation}
\end{definition}

We shall now complete $\Gr_\infty$ into a right Hilbert C*-module over $A_\infty$ that will become of central importance in the sequel: For any $n,m \in\N$ we have left and right actions of $\Gr_{0,n}$ and $\Gr_{0,m}$ on $\Gr_\infty$ defined as follows: for $a\in \Gr_{0,n}$, $a'\in\Gr_{0,m}$ and $\xi\in\Gr_\infty,$ we have 
\begin{equation}
    \begin{tikzpicture}[scale=1/15 pt, thick,baseline={([yshift=-\the\dimexpr\fontdimen22\textfont2\relax] current bounding box.center)}]
        \draw[ultra thick, rounded corners] (-30,-5) rectangle (-20,5);
        \node at (-25,5){$\bullet$};
            \node at (-25,0){$a$};
                \draw[ultra thick] (-28,5) -- (-28,15);
                    \node at (-31,12){$l$};
                \draw[ultra thick] (-22,5) -- (-22,15);
                    \node at (-20,12){$r$};
        
        \node at (-16,0){$\rhd$};
        
        \draw[ultra thick, rounded corners] (-10,-5) rectangle (0,5);
        \node at (-5,5){$\bullet$};
            \node at (-5,0){$\xi$};
                \draw[ultra thick] (-8,5) -- (-8,15);
                    \node at (-11,12){$l''$};        
                \draw[ultra thick] (-2,5) -- (-2,15);        
                    \node at (2,12){$r''$};
                \draw[ultra thick] (-5,-5) -- (-5,-15);
                    \node at (-2,-13){$b$};
        
        \node at (7,0){$\lhd$};
        
        \draw[ultra thick, rounded corners] (11,-5) rectangle (21,5);
        \node at (16,5){$\bullet$};
            \node at (16,0){$a'$};
                \draw[ultra thick] (19,5) -- (19,15);
                    \node at (23,12){$r'$};
                \draw[ultra thick] (13,5) -- (13,15);
                    \node at (10,12){$l'$};
    \end{tikzpicture}
    :=\ \ \delta_{r = l''}\cdot\delta_{r'' = l'}\cdot
    \begin{tikzpicture}[scale=1/15 pt, thick,baseline={([yshift=-\the\dimexpr\fontdimen22\textfont2\relax] current bounding box.center)}]
        \draw[densely dashed, rounded corners, blue] (-25,-9) rectangle (25,12);
        \node at (0,12){$\bullet$};
            \draw[ultra thick, rounded corners] (-20,-5) rectangle (-10,5);
            \node at (-15,5){$\bullet$};
                \node at (-15,0){$a$};
                    \draw[ultra thick] (-18,5) -- (-18,17);
                        \node at (-21,15){$l$};
            
            \draw[ultra thick, rounded corners] (-5,-5) rectangle (5,5);
            \node at (0,5){$\bullet$};
                \node at (0,0){$\xi$};
                    \draw[ultra thick] (-3,5) arc(0:180:5);
                    \draw[ultra thick] (3,5) arc (180:0:5);        
                    \draw[ultra thick] (0,-5) -- (0,-15);
                        \node at (3,-13){$b$};
            
            \draw[ultra thick, rounded corners] (11,-5) rectangle (21,5);
            \node at (16,5){$\bullet$};
                \node at (16,0){$a'$};
                    \draw[ultra thick] (19,5) -- (19,17);
                        \node at (22,15){$r'$};
    \end{tikzpicture}
\end{equation}
The module $\Gr_{\infty}$ then becomes endowed with an $A_\infty$-valued inner product (taking values in $\Gr_{0,\infty}$) given by the sesquilinear extension of
\begin{equation}
    \langle \xi\ |\ \eta \rangle_{A_\infty} := \delta_{l = l'} \cdot \delta_{b = b'}\cdot\  
    \begin{tikzpicture}[scale=1/15 pt, thick,baseline={([yshift=-\the\dimexpr\fontdimen22\textfont2\relax] current bounding box.center)}]
        \draw[densely dashed, rounded corners, blue] (12,-16) rectangle (42,13);
        \node at (27, 13){$\bullet$};        
            \draw[rounded corners, ultra thick] (24,-5) rectangle (14,5);
            \node at (19,5){$\bullet$};
            \node at (19,0){$\xi^{*}$};    
                \draw[ultra thick] (15,5) -- (15,20);
                    \node at (12,17){$r$};
        
                    \draw[ultra thick] (35,-5) arc(0:-180: 8);
                        \node at (35,-12){$b$};
                    \draw[ultra thick] (33,5) arc (0:180:6);
            
            \draw[rounded corners, ultra thick] (40,-5) rectangle (30,5);
            \node at (35,0){$\eta$};
            \node at (35,5){$\bullet$};
                \draw[ultra thick] (38,5) -- (38,20);
                    \node at (41,17){$r'$};
    \end{tikzpicture}
\end{equation}
for $\xi \in V_{b, l, r}$ and $\eta \in V_{b', l', r'}$.  Note that this inner product is linear in the right variable and conjugate linear in the left.  

It is a straightforward computation to prove the next Proposition, following from the pivotal spherical structure on $\cat$ (see Remark \ref{rmk::biinv}):
\begin{proposition}\label{prop:prehilbert}
    The following statements hold for all diagrams $\xi, \eta \in \Gr_{\infty}$ and $a \in \Gr_{0, \infty}:$
    \begin{itemize}
        \item  Right $A_\infty$-linearity: $\langle \xi\ |\ \eta \lhd a \rangle_{A_\infty} = \langle \xi\ |\ \eta \rangle_{A_\infty} \wedge a;$ 
        
        \item The left action is right-adjointable: $\langle \xi\ |\ a^*\rhd\eta \rangle_{A_\infty} = \langle a\rhd \xi\ |\ \eta \rangle_{A_\infty};$  and 
        
        \item  Compatibility with adjoints: $(\langle \xi\ |\ \eta \rangle_{A_\infty})^{*} = \langle \eta\ |\ \xi \rangle_{A_\infty}.$ 
    \end{itemize}
\end{proposition}

\begin{definition}
    The vector space $\Gr_{\infty}$ can now be endowed with a C*-norm, denoted $\| \cdot \|_{A_\infty},$ given by $\| \xi \|_{A_\infty}^2 := \| \langle \xi\ |\ \xi \rangle_{A_\infty} \|_{A_\infty},$ where the latter norm is the norm in the C$^*$-algebra $A_{\infty}$. (We remind the reader that this norm was introduced previously in Definition \ref{bimodule}.) 
    We therefore define $\mathcal{X}_{\infty}$ to be the completion of $\Gr_{\infty}$ under $\| \cdot \|_{A_\infty};$ this is
    $$\X_\infty := \overline{\Gr_\infty}^{||\,\cdot\,||_{A_\infty}}.$$
    We also extend the $A_\infty$-valued inner product to all of $\X_\infty\times\X_\infty$ in the obvious way. We will keep the same notation for this extended product. 
\end{definition}

To equip $\mathcal{X}_{\infty}$ with a right Hilbert $A_{\infty}$-module structure (see Definition \ref{bimodule} on Hilbert bimodules and Definition 1.3.2 in \cite{MR2125398} for Hilbert  modules), we must first extend the actions to all of $\X_{\infty}$ and all of $A_\infty.$ 
The right $\Gr_{0,\infty}$-action extends to $\X_\infty$ since for each sequence $\xi_n\in \Gr_{\infty}$ converging to $\xi\in \X_\infty$, and for each $a\in\Gr_{0,\infty},$ a simple diagrammatic computation reveals that 
$||\xi_n \lhd a||_{A_\infty}^2 =
\delta_x\cdot ||a^*\wedge \langle \xi_n | \xi_n \rangle_{A_\infty}\wedge a||_{A_\infty} \leq
\delta_x\cdot ||\xi_n||_{A_\infty}^2\cdot ||a||_{A_\infty}^2,$
which is bounded above in $A_\infty$-norm, since the $\xi_n$ are bounded. Using the same inequality we obtain a (bounded) right $A_\infty$-action on $\X_\infty.$ Moreover, observe that the first and last items in Proposition \ref{prop:prehilbert} apply to the extended $A_\infty$-action over the module $\X_\infty$. Thus, $(\X_\infty,\langle \cdot \ |\ \cdot \rangle_{A_\infty})$ is a right Hilbert C*-module over $A_\infty.$ 

We shall now show there is a left action of $A_\infty$ on $\X_\infty$ by right adjointable operators. To do so, we first bring an auxiliary lemma:
\begin{lemma}\label{lemma::norm}
    For $n,m\geq 0$, the collection of maps 
    \begin{align}
        \Gr_{0,n}\longrightarrow \Gr_{0,n+m}
        \hspace{3mm} \text{given by the linear extension of} \hspace{5mm}
        \begin{tikzpicture}[scale=1/15 pt, thick,baseline={([yshift=-\the\dimexpr\fontdimen22\textfont2\relax] current bounding box.center)}]
            \draw[rounded corners, ultra thick](-5,-5) rectangle (5,5);
            \node at (0,5){$\bullet$};
                \node at (0,0){$\xi$};
                \draw[ultra thick] (-3,5) -- (-3,15);
                \draw[ultra thick] (3,5) -- (3,15);                
        \end{tikzpicture}
        \mapsto
        \begin{tikzpicture}[scale=1/15 pt, thick,baseline={([yshift=-\the\dimexpr\fontdimen22\textfont2\relax] current bounding box.center)}]
            \draw [rounded corners, densely dashed, blue] (-10,-16) rectangle (10,10);
            \node at (0,10){$\bullet$};
                \draw[ultra thick] (-7,-7) arc(180:360:7);
                \draw[ultra thick] (-7,-7) -- (-7,15);
                \draw[ultra thick] (7,-7) -- (7,15);
                    \draw[rounded corners, ultra thick](-5,-5) rectangle (5,5);
                    \node at (0,5){$\bullet$};
                        \node at (0,0){$\xi$};
                        \draw[ultra thick] (-3,5) -- (-3,15);
                        \draw[ultra thick] (3,5) -- (3,15);                
        \end{tikzpicture}
    \end{align}
    extends to an isometric C*-algebra homomorphism $A_\infty\longrightarrow A_\infty.$
\end{lemma}
\begin{proof}
    This map is clearly a densely defined injective $*$-algebra homomorphism. This map is bounded because it is a monomorphism between the finite dimensional algebras $\Gr_{0,n}$ and $\Gr_{0,n+m}$.
\end{proof}
For arbitrary $\eta\in V_{b,l,r}\subset\Gr_\infty$ and $a\in \Gr_{0,\infty}$ we have the following diagrammatic identity:
\begin{align}
    \langle a\rhd\eta | a\rhd\eta\rangle_{A_\infty}
    = \delta_{l = r'}\cdot
    \begin{tikzpicture}[scale=1/15 pt, thick,baseline={([yshift=-\the\dimexpr\fontdimen22\textfont2\relax] current bounding box.center)}]
        \draw[ultra thick](-14,10) arc(0:180:4);
            \draw[densely dashed, rounded corners, blue] (-37,-15) rectangle (-20,10);
            \node at(-25,10){$\bullet$};
            \draw[ultra thick](-22,10) -- (-22,-7);
                \draw[rounded corners, ultra thick](-35,-5) rectangle (-25,5);
                    \node at(-30,5){$\bullet$};
                    \node at(-30,0){$\eta^*$};
                    \draw[ultra thick](-33,5) -- (-33,23);
                        \node at(-36,20){$r$};
                    \draw[ultra thick](-28,5) -- (-28,11);
                    \draw[ultra thick](-30,-5) -- (-30,-7);
                        \draw[ultra thick] (-22,-7) arc (360:180:4);                    
                            \node at(-32,-8){$b$};
        \draw[ultra thick](-10,10) arc(0:180:9);
            \node at(9,14){$l'$};
        \draw[ultra thick](10,10) arc(180:0:9);
            \node at(-8,14){$l'$};
           \draw[densely dashed, rounded corners, blue] (-18,-17) rectangle (18,11);
           \node at(0,11){$\bullet$};
                \draw[ultra thick](4,5) arc(0:180:4);
                    \draw[rounded corners, ultra thick](-12,-5) rectangle (-2,5);        
                    \node at(-7,5){$\bullet$};
                        \node at(-7,0){$a^*$};
                            \draw[ultra thick](-10,5) -- (-10,10);
                    \draw[rounded corners, ultra thick](2,-5) rectangle (12,5);
                    \node at(7,5){$\bullet$};
                        \node at(7,0){$a$};
                            \draw[ultra thick](10,5) -- (10,10);
            \draw[ultra thick] (-14,10)-- (-14,0) arc(180:360:14) -- (14,10);   
        \draw[ultra thick](14,10) arc(180:0:4);
            \node at(25,10){$\bullet$};
            \draw[densely dashed, rounded corners, blue] (37,-15) rectangle (20,10);
            \draw[ultra thick] (22,10) -- (22,-7);
                \draw[rounded corners, ultra thick](25,-5) rectangle (35,5);        
                \node at(30,5){$\bullet$};
                    \node at(30,0){$\eta$};
                    \draw[ultra thick] (33,5) -- (33,23);
                        \node at(36,20){$r$};
                    \draw[ultra thick] (28,5) -- (28,11);
                    \draw[ultra thick](30,-5) -- (30,-7);
                        \draw[ultra thick] (22,-7) arc (180:360:4);
                            \node at(32,-8){$b$};
    \end{tikzpicture}    
\end{align}
Thus, denoting the rightmost element above as $\xi\in\Gr_{0,\infty}$ we then obtain that $\langle a\rhd\eta\ |\ a\rhd\eta \rangle_{a_\infty} = \xi^*\wedge [(a^*\wedge a)] \wedge \xi.$ Here, $[a^*\wedge a]$ denotes the middle element in the diagram above. By Lemma \ref{lemma::norm} we obtain that $||a\rhd\eta||_{A_\infty}^{2} = ||\langle a\rhd\eta\ |\  a\rhd\eta \rangle_{A_\infty}||_{A_\infty} = ||\ \xi^*\wedge [(a^*\wedge a)] \wedge \xi||_{A_\infty}\leq  ||\xi\ ||_{A_\infty}^{2}\cdot  ||a||_{A_\infty}^2  =||\eta||_{A_\infty}^{2}\cdot ||a||_{A_\infty}^2.$ Similarly as done for the right action, this inequality allows us to define a bounded left action of $A_\infty$ on $\X_\infty$.  Finally, notice that the second item in Proposition \ref{prop:prehilbert} (which extends to the left action of $A_\infty$ on $\X_\infty$ by a density argument), 
implies that the left $A_\infty$-action is given by right-adjointable operators. Therefore, there is a faithful embedding $([A_\infty]\rhd -) \hookrightarrow \mathcal B^*(({\X_\infty})_{A_\infty}).$

There is an $(\N\cup\{0\})$-graded family of right $A_\infty$-modules of importance:
\begin{definition}
    For arbitrary $b\geq 0$, set $\Gr_{b, \infty}$ to be the following subspace of $\Gr_{\infty}:$  $$\Gr_{b,\infty} := \bigoplus_{l \geq 0, r \geq 0} V_{b, l, r},$$ and define $X_{b}$ as its completion in $\X_\infty$. Note that each $X_{b}$ is naturally a right $A_{\infty}$-module with a canonical left $A_{\infty}$ action (by right-adjointable operators). Also note that $X_{0} = A_{\infty}$.
\end{definition}

The treatment below relies heavily on the machinery in section 4 of \cite{HaPeI} and is essentially a translation of that section into our diagrammatic language.  In that article, planar algebras were used as a substitute for the graphical picture used in this article.  Below is a dictionary that the reader can use when parsing the changes in notation and diagrams between \cite{HaPeI} and this article.
\begin{center}
    \begin{tabular}{c|c}
       \cite{HaPeI}  & This manuscript \\
       \hline
        \begin{tikzpicture}[scale=1/17 pt, thick,baseline={([yshift=-\the\dimexpr\fontdimen22\textfont2\relax] current bounding box.center)}] 
            \draw[rounded corners, ultra thick] (-30,-5) rectangle (-20,5);
            \node at (-25,0){$\xi$};
            \draw[ultra thick] (-40,0) -- (-30,0);
            \draw[ultra thick] (-20,0) -- (-10,0);
            \draw[ultra thick] (-25,5) -- (-25,20);
                \node at (-35,-3){$l$};
                \node at (-15,-3){$r$};
                \node at (-23,15){$b$};
        \end{tikzpicture} & \begin{tikzpicture}[scale=1/17 pt, thick,baseline={([yshift=-\the\dimexpr\fontdimen22\textfont2\relax] current bounding box.center)}] 
            \draw[rounded corners,ultra thick] (-30,-5) rectangle (-20,5);
            \node at (-25,0){$\xi$};
            \node at (-25,5){$\bullet$};
                \draw[ultra thick] (-28,5) -- (-28,15);
                \draw[ultra thick] (-22,5) -- (-22,15);
                \draw[ultra thick] (-25,-5) -- (-25,-15);
                    \node at (-31,13){$l$};
                    \node at (-19,13){$r$};
                    \node at (-23,-12){$b$};
        \end{tikzpicture}\\ 
        \hline
            $\mathcal{P}_{l, n, r}$ & $V_{b, l, r}$\\    
        \hline
            $\mathcal{F}(\mathcal{P}_{\bullet})$ & $\mathcal{X}_{\infty}$\\
        \hline
            $B_{n}(\mathcal{P}_{\bullet})$ & $\Gr_{0, n}$\\
        \hline
            $B(\mathcal{P}_{\bullet})$ & $A_{\infty}$\\
        \hline
            $X_{b}$ & $\Gr_{b, \infty}$\\
        \hline
            $\X_{b}$ & $X_{b}$\\
        \hline
            $L(\xi)$ & $L(\xi) (\text{creation operators})$
    \end{tabular}
\end{center}

We have the following proposition, appearing as \cite[Proposition 4.11]{HaPeI}:
\begin{proposition}\label{prop:easytensor}\cite{HaPeI}
    For arbitrary $b,b'\geq 0$, the mapping 
    $$U_{b,b'}: \Gr_{b,\infty}\odot \Gr_{b',\infty}\rightarrow \Gr_{(b+b'),\infty}$$
    defined by
    \begin{equation}
        \begin{tikzpicture}[scale=1/15 pt, thick,baseline={([yshift=-\the\dimexpr\fontdimen22\textfont2\relax] current bounding box.center)}] 
            \draw[rounded corners, ultra thick] (-35,-5) rectangle (-25,5);
            \node at (-30,5){$\bullet$};
                \node at (-30,0){$\xi$};
                \draw[ultra thick] (-33,5) -- (-33,20);
                    \node at (-36,16){$l$};
                \draw[ultra thick] (-27,5) -- (-27,20);
                    \node at (-24,16){$r$};
                \draw[ultra thick] (-30,-5) -- (-30,-20);
                    \node at (-28,-13){$b$};
                
            \node at(-20,0){$\ \ \odot\ \ $};
                
            \draw[rounded corners, ultra thick] (-14,-5) rectangle (-4,5);
            \node at (-9,5){$\bullet$};
                \node at (-9,0){$\eta$};
                \draw[ultra thick] (-12,5) -- (-12,20);
                    \node at (-15,16){$l'$};
                \draw[ultra thick] (-5,5) -- (-5,20);
                    \node at (-2,16){$r'$};
                \draw[ultra thick] (-9,-5) -- (-9,-20);
                    \node at (-6,-13){$b'$};
                    
            \node at (3,0){$\longmapsto$};
            \node at(15,0){$\delta_{r = l'}\cdot \ $};

            \draw[densely dashed, rounded corners, blue] (22,-13) rectangle (52,13);
            \node at (37, 13){$\bullet$};
                \draw[rounded corners, ultra thick] (34,-5) rectangle (24,5);
                \node at (29,5){$\bullet$};    
                    \node at (29,0){$\xi$};
                    \draw[ultra thick] (26,5) -- (26,20);
                        \node at (23,17){$l$};
                \draw[rounded corners, ultra thick] (50,-5) rectangle (40,5);
                \node at (45,5){$\bullet$};
                    \node at (45,0){$\eta$};
                    \draw[ultra thick] (48,5) -- (48,20);
                        \node at (51,17){$r'$};

            \draw[ultra thick] (43,5) arc (0:180:6);
                \draw[ultra thick] (29,-5) .. controls (29,-10) and (35,-9) .. (35,-20);
                \draw[ultra thick] (45,-5) .. controls(45,-10) and (39,-9) .. (39,-20);
                    \node at (50,-18){$(b + b')$};

        \end{tikzpicture}
    \end{equation}
     extends to a unitary isomorphism of right $A_\infty$ Hilbert C*-bimodules:
    $$U_{b,b'}: X_{b} \otimes_{A_\infty} X_{b'} \rightarrow X_{b+b'}.$$
    Here, the relative (algebraic) tensor product $\odot$ is balanced over diagrams with no strings going up; i.e. over $\bigoplus_{b''\geq 0}\Gr_{b'',0,0}$.
\end{proposition}

We now introduce a construction due to Pimsner
\begin{definition}\label{dfn::Fock}\cite{MR1426840}
    Given an arbitrary Hilbert C*-module $\Y$ over a C$^{*}$-algebra $A$ with a left action of $A$ as bounded, adjointable operators on $\Y$, one can form the \textbf{full Fock space} $\mathcal{F}(\Y)$ given by 
    $$ \mathcal{F}(\Y) = A \oplus \bigoplus_{n=1}^{\infty} \Y^{\otimes^{n}_{A}}.$$
\end{definition}

An immediate consequence from Proposition \ref{prop:easytensor} and the previous definition gives:
\begin{corollary}\label{cor:fockspace}
    The unitary operators in Proposition \ref{prop:easytensor} induce a unitary
    $$U: \mathcal{F}(X_1) \rightarrow \mathcal{X}_{\infty}.$$
\end{corollary}

Recall that by \cite{MR1426840}, given $\mathcal{F}(\Y)$ as above, and $\xi \in \Y$, one can form the \textbf{creation operator} $\ell(\xi)$ on $\mathcal{F}(\Y)$ defined as as the linear extension of
\begin{align*}
    \ell(\xi)a &= \xi a &\text {for all } a \in A\\
    \ell(\xi)(\xi_{1} \otimes \cdots \otimes \xi_{n}) &= \xi \otimes \xi_{1} \otimes \cdots \otimes \xi_{n} &\text{for all } \xi_{1}, \dots, \xi_{n} \in \Y
\end{align*}
        the operator $\ell(\xi)$ is bounded and right-adjointable, and its adjoint is given by
\begin{align*}
    \ell(\xi)^{*}a &= 0 &\text {for all } a \in A\\
    \ell(\xi)^{*}(\xi_{1} \otimes \cdots \otimes \xi_{n}) &= \langle\xi\ |\  \xi_{1} \rangle_{A} \xi_{2} \otimes \cdots \otimes \xi_{n} &\text{for all } \xi_{1}, \dots, \xi_{n} \in \Y
\end{align*}

In particular, for $\Y = X_1$ and for an arbitrary $\xi \in V_{1, l, r}$, consider the operator $L(\xi)$ 
on $\Gr_\infty$, which we define using the spaces $V_{b',l',r'}$ by: 
\begin{equation}
    L(\xi): 
    \begin{tikzpicture}[scale=1/15 pt, thick,baseline={([yshift=-\the\dimexpr\fontdimen22\textfont2\relax] current bounding box.center)}] 
    
        \draw[rounded corners, ultra thick] (-14,-5) rectangle (-4,5);
        \node at (-9,0){$\eta$};
        \node at (-9,5){$\bullet$};
        \draw[ultra thick] (-12,5) -- (-12,20);
        \draw[ultra thick] (-6,5) -- (-6,20);
        \draw[ultra thick] (-9,-5) -- (-9,-20);
            \node at (-15,14){$l'$};
            \node at (-3,14){$r'$};
            \node at (-6,-17){$b'$};
    
        \node at (2,0){$\mapsto\ $};
        
        \node at(10,0){$\delta_{r = l'}\cdot \ $};
        
        \draw[densely dashed, rounded corners, blue] (17,-13) rectangle (47,13);
        \node at (32, 13){$\bullet$};
            \draw[rounded corners, ultra thick] (29,-5) rectangle (19,5);
            \node at (24,0){$\xi$};
            \node at (24,5){$\bullet$};
                \draw[ultra thick] (21,5) -- (21,20);
                    \node at (19,17){$l$};
                    
            \draw[rounded corners, ultra thick] (45,-5) rectangle (35,5);
            \node at (40,0){$\eta$};
            \node at (40,5){$\bullet$};
                \draw[ultra thick] (43,5) -- (43,20);
                    \node at (46,17){$r'$};
                
                \draw[ultra thick] (40,-5) .. controls (40,-11) and (34,-12) .. (34,-20);
                \draw (24,-5) .. controls (24,-11) and (29,-12) .. (29,-20);
                    \node at (44,-17){$(1 + b')$};
                \draw[ultra thick] (38,5) arc (0:180:6);
    \end{tikzpicture}
\end{equation}

For $\xi \in \Gr_{1,\infty}$, one defines $L(\xi)$ by linearity.  The following is now immediate:
\begin{proposition}\label{prop::intertwiner}
The unitary operator $U$ in Corollary \ref{cor:fockspace} satisfies $U^{*}L(\xi)U = \ell(\xi)$ for all $\xi\in\Gr_{1,\infty}$. 
\end{proposition}

Since one has $\|L(\xi)\| = \|\xi\|$ for all $\xi \in \Gr_{1,\infty}$  \cite{MR1426840}, it follows immediately that $L(\xi)$ is bounded for all  $\xi \in \Gr_{1,\infty}$, and that one can use continuity to define $L(\xi)$ for all $\xi \in X_{1}$. It follows from the arguments in \cite[Propositions 4.7--4.9]{HaPeI} that the $*$-operation extends continuously on each $X_{n}$.

We now introduce the C*-algebra $B_{\infty}$, a corner of which (the unital simple C*-algebra $B_0$ from Notation \ref{notation::corners}) will be the main C$^{*}$-algebra in this article.

\begin{definition}\label{dfn::BInfty}
    Let $X_{\R}$ be the real subspace $\{\xi \in X_1|\ \xi = \xi^{*}\}$.  Then we define $B_{\infty}$ to be the C*-algebra generated by
            $$A_\infty \cup \{L(\xi) + L(\xi)^{*} |\ \xi \in X_{\R}\} \subset\mathcal B^*(\X_\infty).$$
    We will often denote the (operator) norm on $B_\infty$ by $||\cdot||_{B_\infty}.$       
\end{definition}

There will be a convenient way to diagrammatically realize $B_{\infty}.$ Given $\xi \in V_{b, l, r}$, define $\pi(\xi)$ to be the operator on $\Gr_{\infty}$ given by (the linear extention of): 
\begin{equation}\label{eqn::WalkerProduct}
    \pi(\xi)\eta := \delta_{r = l'}\cdot \sum_{k=0}^{\min\{b, b' \}}
    \begin{tikzpicture}[scale=1/15 pt, thick,baseline={([yshift=-\the\dimexpr\fontdimen22\textfont2\relax] current bounding box.center)}] 
        \draw[blue, densely dashed, rounded corners] (12,13) rectangle (42,-17);
        \node at (27, 13){$\bullet$};
        \draw[ultra thick] (21, -5) arc (180:360: 6);
            \node at (27,-14){$k$};

            \draw[ultra thick] (33,5) arc (0:180:6);
                \draw[rounded corners, ultra thick] (24,-5) rectangle (14,5);
                \node at (19,5){$\bullet$};
                    \node at (19,0){$\xi$};
                    \draw[ultra thick] (16,5) -- (16,20);
                        \node at (14,17){$l$};
            
                \draw[rounded corners, ultra thick] (40,-5) rectangle (30,5);
                \node at (35,5){$\bullet$};
                    \node at (35,0){$\eta$};
                    \draw[ultra thick] (38,5) -- (38,20);
                        \node at (41,17){$r'$};
            
            \draw[ultra thick] (16,-5) .. controls (16,-11) and (24,-17) .. (24,-23);
                \node at (13,-21){$(b - k)$};
            \draw[ultra thick] (38,-5) .. controls (38,-11) and (30,-17) .. (30,-23);  
                \node at (41,-21){$\scriptsize{(b' - k)}$};
    \end{tikzpicture}
\end{equation}

for $\eta \in V_{b', r', l'}$.  We extend the definition of $\pi(\xi)$ linearly in $\xi$ to all of $\Gr_\infty$. By the following Proposition (\cite[Proposition 4.16]{HaPeI}), we obtain that $\pi(\xi)$ is bounded: 

\begin{proposition}
For each $\eta \in \Gr_{\infty}$, $\pi(\eta)$ is a polynomial in $A_{\infty}$ and $\{L(\xi) + L(\xi)^{*} |\ \xi \in X_{\R}\}$.  Consequently, $\pi(\eta)$ is bounded for each $\eta \in \Gr_{\infty}$.  
\end{proposition}

%
%
%
\begin{remark}
    	We notice that under this representation, $\pi,$ the composition of operators satisfies $\pi(\xi')\circ \pi(\xi) = \pi(\xi'\star\xi),$ where $\xi'\in V_{b',l',r'}$, $\xi\in V_{b,l,r}$ and $-\star-$ 
	is the \textbf{Walker Multiplication}, which diagrammatically is given by Equation \ref{eqn::WalkerProduct} (\cite{JSw10}). We moreover have that $\Gr_\infty$ 
	with the Walker product $-\star-$ is isomorphic to $\Gr_\infty$ with the \textbf{wedge/graded product} $-\wedge-$ as unital (tracial) $*$-algebras. (See \cite[Lemmata 5.2 and 5.3]{JSw10},  for further details on the traces and the isomorphism).
   	It is straightforward to see that if $\xi \in X_{\R}$, then $\pi(\xi) = \xi\star- = L(\xi) + L(\xi)^{*}$. From this observation, we see that $B_{\infty}$ is generated by $A_{\infty}$ and $\{\pi(\eta) |\ \eta \in \Gr_{\infty}\}$.  
\end{remark}

\subsection{A weight and a conditional expectation on $B_{\infty}$} \label{section:weight}\ \\
The AF $*$-algebra $\Gr_{\infty} \cap A_{\infty}$ can be endowed with the following positive linear functional given by the linear extension of
\begin{align}
    \Phi:\Gr_\infty \cap A_\infty \longrightarrow \C
    \hspace{5mm}\text{given by}\hspace{5 mm}
    \begin{tikzpicture}[scale=1/15 pt, thick,baseline={([yshift=-\the\dimexpr\fontdimen22\textfont2\relax] current bounding box.center)}] 
        \draw[rounded corners, ultra thick] (24,0) rectangle (14,10);
        \node at (19,10){$\bullet$};
            \node at (19,5){$a$};
            \draw[ultra thick] (16,10) -- (16,17);
            \draw[ultra thick] (22,10) -- (22,17);
    \end{tikzpicture}
    \longmapsto \delta_{\tiny{l = r}}\cdot
    \begin{tikzpicture}[scale=1/15 pt, thick,baseline={([yshift=-\the\dimexpr\fontdimen22\textfont2\relax] current bounding box.center)}] 
        \draw[rounded corners, ultra thick] (24,0) rectangle (14,10);
        \node at (19,10){$\bullet$};
            \node at (19,5){$a$};
            \draw[ultra thick] (16,10) -- (16,12);
            \draw[ultra thick] (22,10) -- (22,12);
            \draw[ultra thick] (16,12) arc(180:0: 3);
    \end{tikzpicture}.
\end{align}
Proposition \ref{prop::AFalgs}, allows us to see this map is simply the trace on $\cat.$ (See Equation \ref{eqn::balancing}.) Therefore, the map $\Phi$ extends to a faithful, positive, semifinite, tracial weight on $A_{\infty}.$

If we if we let $\Pi_{0}\in \mathcal{B}(\X_\infty)$ be the projection from $\mathcal{X}_{\infty}$ onto $A_{\infty}$, then from (Lemma 2.10 in \cite{MR1704661}), $\Pi_{0}$ induces a faithful \textbf{conditional expectation}, $\mathbb{E}$ from $B_{\infty}$ onto $A_{\infty}$. Diagrammatically, this map is given by
\begin{align}
    \begin{split}
        &\mathbb E:B_\infty \longrightarrow A_\infty,\\
        &b\longmapsto \Pi_{0}\ b\ \Pi_{0}.
    \end{split}
    \hspace{5mm}
    \hspace{5mm}
    \begin{tikzpicture}[scale=1/15 pt, thick,baseline={([yshift=-\the\dimexpr\fontdimen22\textfont2\relax] current bounding box.center)}]
        \draw[rounded corners, ultra thick](-5,-5) rectangle (5,5);
        \node at (0,5){$\bullet$};
            \node at (0,0){$b$};
            \draw[ultra thick] (-3,5) -- (-3,15);
            \draw[ultra thick] (3,5) -- (3,15);                
            \draw[ultra thick] (0,-5) -- (0,-15);
    \end{tikzpicture}
    \longmapsto
    \begin{tikzpicture}[scale=1/15 pt, thick,baseline={([yshift=-\the\dimexpr\fontdimen22\textfont2\relax] current bounding box.center)}]
        \draw[rounded corners, ultra thick](-5,-5) rectangle (5,5);
        \node at (0,5){$\bullet$};
            \node at (0,0){$b_\emptyset$};
            \draw[ultra thick] (-3,5) -- (-3,15);
            \draw[ultra thick] (3,5) -- (3,15);                
            \draw[densely dashed] (0,-5) -- (0,-15);
    \end{tikzpicture}
\end{align}

A straightforward computation shows that if  $x \in \Gr_{\infty},$ then $\Pi_{0}x\Pi_{0} \in A_{\infty}$, so $\mathbb{E}$ is well defined by continuity. From \cite[Lemma 2.10 and Remark 2.13]{MR1704661} we obtain the following proposition: 

\begin{proposition}\label{prop:weights}
	The following assertions hold:    
	 	\begin{itemize}
	      	  	\item The map $\mathbb{E}$ is a faithful conditional expectation from $B_{\infty}$ onto $A_{\infty}$. \cite{HaPeI}
	        	\item The composite map $\Phi \circ \mathbb{E}$ is a faithful, semifinite tracial weight on $B_{\infty},$ which is finite on $\{\pi(\xi) | \xi \in \Gr_{\infty}\}$. \cite{HaPeI}	        
            \end{itemize}
\end{proposition}

As Hilbert spaces, we have that $L^{2}(B_{\infty}, \Phi \circ \mathbb{E}) \cong \mathcal{X}_{\infty} \otimes_{A_{\infty}} L^{2}(A_{\infty}, \Phi).$ In fact, one can check diagrammatically that this equivalence is realized by the unitary $W$ given by $W(\pi(\xi)) = \xi$ for all $\xi \in \Gr_{\infty}$.  Moreover, $W$ satisfies $W(b\eta) = bW(\eta)$ for all $b \in B_{\infty}$ and each $\eta\in \Gr_\infty$.

\begin{remark}\label{remark::density}
	The above proposition allows us to canonically identify $\Gr_{\infty}$ as a dense subspace of $B_{\infty},$ and we will suppress the $\pi$ in $\pi(x)$ when the context is clear.
\end{remark}

\begin{remark}\label{rmk::GradProd}
	In order to present a more visually appealing picture of the multiplication in $B_{\infty}$ we endow $\Gr_{\infty}$ with the following multiplication and tracial weight 
	(sometimes refered to as the \textit{Voiculescu trace})
	$$
	x \wedge y = 	\begin{tikzpicture}[scale=1/15 pt, thick,baseline={([yshift=-\the\dimexpr\fontdimen22\textfont2\relax] current bounding box.center)}] 
	            \draw[rounded corners, ultra thick] (-35,-5) rectangle (-25,5);
	            \node at (-30,5){$\bullet$};
	                \node at (-30,0){$x$};
	                \draw[ultra thick] (-33,5) -- (-33,15);
	                    \node at (-36,13){$l$};
	                \draw[ultra thick] (-27,5) -- (-27,15);
	                    \node at (-24,13){$r$};
	                \draw[ultra thick] (-30,-5) -- (-30,-15);
	                    \node at (-28,-13){$b$};
	                
	            \node at(-20,0){$\ \ \wedge\ \ $};
	                
	            \draw[rounded corners, ultra thick] (-14,-5) rectangle (-4,5);
	            \node at (-9,5){$\bullet$};
	                \node at (-9,0){$y$};
	                \draw[ultra thick] (-12,5) -- (-12,15);
	                    \node at (-15,13){$l'$};
	                \draw[ultra thick] (-5,5) -- (-5,15);
	                    \node at (-2,13){$r'$};
	                \draw[ultra thick] (-9,-5) -- (-9,-15);
	                    \node at (-6,-13){$b'$};
	                    
	            \node at (3,0){$:=$};
	            \node at(15,0){$\delta_{r = l'}\cdot \ $};
	
	            \draw[densely dashed, rounded corners, blue] (22,-13) rectangle (52,13);
	            \node at (37, 13){$\bullet$};
	                \draw[rounded corners, ultra thick] (34,-5) rectangle (24,5);
	                \node at (29,5){$\bullet$};    
	                    \node at (29,0){$x$};
	                    \draw[ultra thick] (26,5) -- (26,20);
	                        \node at (23,17){$l$};
	                \draw[rounded corners, ultra thick] (50,-5) rectangle (40,5);
	                \node at (45,5){$\bullet$};
	                    \node at (45,0){$y$};
	                    \draw[ultra thick] (48,5) -- (48,20);
	                        \node at (51,17){$r'$};
	
	            \draw[ultra thick] (43,5) arc (0:180:6);
	                \draw[ultra thick] (29,-5) .. controls (29,-10) and (35,-9) .. (35,-20);
	                \draw[ultra thick] (45,-5) .. controls(45,-10) and (39,-9) .. (39,-20);
	                    \node at (50,-18){$(b + b')$};

	        \end{tikzpicture}
	$$
	and
	$$
		\Tr(x) = \begin{tikzpicture}[scale=1/15 pt, thick,baseline={([yshift=-\the\dimexpr\fontdimen22\textfont2\relax] current bounding box.center)}] 
		            \draw[rounded corners, ultra thick] (-35,-5) rectangle (-25,5);
		            \node at (-30,5){$\bullet$};
		                \node at (-30,0){$x$};
		                \draw[ultra thick] (-33,5) -- (-33,15);
		                    \node at (-36,13){$l$};
		                \draw[ultra thick] (-27,5) -- (-27,15);
		                    \node at (-24,13){$r$};
		                \draw[ultra thick] (-30,-5) -- (-30,-15);
		                    \node at (-28,-10){$b$};
		                    \draw[rounded corners, ultra thick] (-40,-25) rectangle (-20,-15);
		                    \node at (-30, -20){$\sum \mathsf{NC}_2$};
		\end{tikzpicture},
	$$
	where $\sum \mathsf{NC}_2$ represents the sum over all non-crossing pairings of the $b$ strings on the bottom of the diagram.
	  It was shown in \cite{JSw10}, that there is a tracial weight preserving
	 $*$-isomoprhism $\Psi$ from $\pi(\Gr_{\infty})$ to $\Gr_{\infty}$ satisfying
	$$
		\Phi (\pi (x)) \circ \mathbb{E} = \Tr(\Psi(\pi(x))) \text{ and } \Psi(\pi(x)\pi(y)) = \Psi(\pi(x)) \wedge \Psi(\pi(y)).
	$$
	This means that the C*-algebra completion of $\Gr_{\infty}$ in the GNS representation of $\Tr$ is also isomorphic to $B_{\infty}$.  
	 We will use this multiplication below to declutter many diagrams.
\end{remark}

\begin{notation}\label{notation::corners}
    	For each $n\in \N$ we let $p_{n}\in \Gr_{0,n,n}$ be the element 
    	\begin{equation*}
	        p_{n} := 
	        \begin{tikzpicture}[baseline= -.1cm]
		            \node at (0.5, 0) {$n$};
		            \draw[ultra thick] (-0.3, 0) -- (-0.3, 0.6);
		            \draw[ultra thick] (0.3, 0) -- (0.3, 0.6);
		            \draw [ultra thick] (-.3, 0) arc(180:360: .3);
		            \draw[densely dashed, rounded corners, blue] (-0.7,-0.5) rectangle (0.7,0.4);
		            \node at (0, 0.4) {$\bullet$};
        	\end{tikzpicture}.
    	\end{equation*}
    	We note that $(\Phi \circ \mathbb{E})$ descends to a positive, faithful linear \textbf{tracial functional} $\Tr_{n}(\bullet) = (\Phi\circ\mathbb{E})(p_n\bullet p_n)$
	 on the unital C*-algebra $$B_n := p_{n}\wedge B_{\infty}\wedge p_{n}\subset B_\infty,$$ which inherits a subalgebra C*-norm, denoted by $||\cdot||_{B_n}.$
    
	By defining also $p_0 := \emptyset\in V_{0,0,0}$ to be the empty diagram, we can consider $\Tr_0$ and the tracial unital (simple) C*-algebra
	 $$
		B_{0} :=p_0\wedge B_\infty\wedge p_0\subset B_\infty,
	$$
	 which is of special interest to us. This subalgebra also inherits a C*-norm, 
	denoted by $||\cdot||_{B_0}.$ Notice that $B_0$ contains all diagrams with no strings on the top; i.e. $l = 0 = r$. Moreover, we observe that
	 $1_{B_0} = \emptyset.$ Finally, for $m,n\geq 0$, we define the $(m,n)$-th \textbf{corner of} $B_\infty$ by $_n B_m := p_n\wedge B_\infty\wedge p_m.$ 
	Observe that $_nB_n = B_n$ for every $n.$ 

	Finally, on each $B_{n}$, we define $\tr_{B_n} := \tr_{n}$ to be $\tr_{n}(b) = \frac{1}{\delta_{x}^{n}}\Tr_{n}(b)$.
	 (We remind the reader that $\delta_x > 1$ was introduced earlier on in Notation \ref{notation::delta}.) 
	Notice that $\tr_{n}(p_{n}) = 1$.
\end{notation}

\begin{notation}\label{notation::vNCorners}
	We denote the von Neuman closure of $B_\infty$ by $M_\infty := B_\infty ''\subset \mathcal B^*(\Ltwo{B_\infty, \Tr})$. Moreover, for $m, n\geq 0$,
	 we denote the corners $_nM_m := p_n\wedge M_\infty\wedge p_m$ 
	and define the Hilbert space $\Ltwo{{_nM_m}} := p_n\rhd\Ltwo{B_\infty}\lhd p_m $.
\end{notation}


The following lemma will be crucial to our analysis.  It requires $\delta_{x} > 1$, which we have assumed:
\begin{lemma}\label{lemma::simplicity}
The following statements hold:
    \begin{itemize}
        \item There exists a separable and infinite-dimensional Hilbert space $\cH$ such that $B_{\infty} \cong B_0 \otimes K(\cH)$ (\cite{HaPeII}, Corollary 4.15)
        \item For each $n\in\N\cup\{0\}$, $B_{\infty}$ is simple as a C*-algebra and consequently, the corner $B_{n}$ is simple too. (See Section 5 in \cite{HaPeII}.)
        \item For every $n\in \N\cup\{0\}$, the unital C*-algebra $B_{n}$ has a unique trace given by $\Tr_n$. (Theorem 5.5 and Corollary 5.7 in \cite{HaPeII}.)
        \item Given $n\in \N \cup\{0\}$, the von Neumann algebra defined by $$M_{n} := B_{n}''\subseteq \mathcal{B}(\Ltwo{B_n},\Tr_n)$$ is 
		an interpolated free group factor $L(\mathbb F_t),$ for $t \in (1, \infty]$ whose unique trace we denote by $\tau_n$. (\cite{GJS11}, \cite{BHP12}, \cite{MR3110503})
        \item For $\cS$ any fixed set of isomorphism classes of simple objects in $\cat_x,$ the full RC*TC generated by $x = \overline x\in \cat,$ we have $K_{0}(B_0) = \Z[\cS]$.  Moreover in this isomorphism, $\one_{\cat} \mapsto B_{0}$. \cite{HaPeI}, \cite{HaPeII}
    \end{itemize}
\end{lemma}

\begin{remark}\label{rmk::cheat}
	Independently of any other result of the following sections other than Lemma \ref{lemma::Hilbertifying}, one can prove
	 that there is an isomorphism of von Neumann algebras	 $M_n \cong {_nM_n},$ eliminating any possible ambiguity in the definitions. 
	We choose not to bring this result and the required notational remarks here as it would slow down the exposition.
	This result becomes of central importance, as it will allow us to work with all the corners $_nM_m$ as embedded in the same
	 \textit{ambient} von Neumann algebra $M_\infty$. 
\end{remark}


\subsection{The Watatani C*-tower of GJS algebras}\ \\
We now consider
\begin{align}
  	\iota_{n}: \Gr_{\infty} \cap B_{0} \rightarrow \Gr_{\infty} \cap B_{n}  
        \hspace{2mm}\text{by}\hspace{5mm}
  	\iota_{n}
        \left(\begin{tikzpicture}
        [scale=1/15 pt, thick,baseline={([yshift=-\the\dimexpr\fontdimen22\textfont2\relax] current bounding box.center)}]
            	\draw[rounded corners, ultra thick] (-5,-5) rectangle (5,5);
            	\node at (0,5){$\bullet$};
                \node at (0,0){$\xi$};
                \draw[ultra thick] (0, -5)--(0, -10);
                \draw[densely dashed] (3,5) -- (3,10);
                \draw[densely dashed] (-3,5) -- (-3,10);
        \end{tikzpicture}   \right) 
        :=
        \begin{tikzpicture}
        [scale=1/15 pt, thick,baseline={([yshift=-\the\dimexpr\fontdimen22\textfont2\relax] current bounding box.center)}]
            	\draw[rounded corners, densely dashed, blue] (-7,10) rectangle (7,-6);
            	\node at (0, 9) {$\bullet$};
                \draw[rounded corners, ultra thick] (-5,-5) rectangle (5,5);
                    \node at (0,0){$\xi$};
                    \draw[ultra thick] (0, -5) -- (0, -10);
                \draw[ultra thick] (-3, 9) arc(180:360: 3);
                \draw[ultra thick] (-3,9) -- (-3,12);
                \draw[ultra thick] (3,9) -- (3,12);
                \node at (6, 12) {$n$};
        \end{tikzpicture}
\end{align}

We now state some of the basic properties of these maps:
\begin{proposition}
    	For each $n\in\N$ we have that
        \begin{itemize}
            \item  the map $\iota_{n}$ is a $*$-algebra homomorphism, and for all $\xi \in \Gr_{\infty} \cap B_{0}$ we obtain $\tr_{0}(\xi) = \tr_{n}(\iota_{n}(\xi))$, and 
            \item the map $\iota_{n}$ extends to an injective C*-algebra homomorphism 
			$$\iota_n: B_{0}\hookrightarrow B_{n}.$$
        \end{itemize}
\end{proposition}
\begin{proof}
	The first bullet point is an immediate diagrammatic calculation.  As for the second bullet point, note that $\tr_{0}$ and $\tr_{n}$ are faithful 
	states on $B_{0}$ and $B_{n}$, respectively.  Using the functional calculus, if $a$ is any self adjoint element in a C$^{*}$-algebra with faithful 
	state $\omega$, it follows that $\displaystyle\|a\| = \lim_{p \rightarrow \infty}[\omega(|a|^{p})]^{1/p}$. Indeed, the map $f: t\mapsto|t|$ is continuous 
	on the spectrum of $a$ and moreover, $f\in\text{L}^p$ for every $p\in[1,\infty].$ Thus, by the Riesz Representation Theorem there exists a positive 
	Radon measure $\mu$ defined on the spectrum of $a$ such that $\int f^p d\mu  = \omega(|a|^p),$ 
	and hence $\omega(|a|^p)^{1/p} = ||f||_p \rightarrow ||f||_\infty = ||a||,$ as $p\rightarrow \infty.$
  	Therefore,
	 for all $\xi \in \Gr_{\infty} \cap B_{0},$
	$$ 
		\|\xi\|^{2}_{B_{0}} = \lim_{p \rightarrow \infty} \tr_{0}((\xi^{*}\wedge\xi)^{p})^{1/p} = \lim_{p \rightarrow \infty} \tr_{n}(\iota_{n}(\xi^{*}\wedge\xi)^{p}))^{1/p} = \|\iota_{n}(\xi)\|^{2}_{B_{n}}.
	$$
	The second bullet point follows.
\end{proof}

 
 The inclusion $i_{n}$ extends to a normal inclusion of von Neumann algebras $M_{0} \hookrightarrow M_{n}$ since it is an isometry in the 2-norm, $\|\cdot \|_{2}$ induced by the respective traces on $M_{0}$ and $M_{n}$.  Indeed, note that in a tracial von Neumann algebra, the 2-norm on bounded sets (in the operator norm) induces the strong operator topology, so it follows that $i_{n}$ is strongly continuous on bounded sets.
 
Here are some useful facts about the inclusion of $B_{0}$ into $B_{n}$ and $M_{0}$ into $M_{n}$
\begin{proposition}\label{RelCommutants}
    The following statements hold for every $n\in\N$:
        \begin{enumerate}
            \item The image set $i_{n}[M_{0}] \subset M_n$ is a subfactor of index $\delta_{x}^{2n}$. \cite{GJS10}
            \item The inclusion $i_{n}[B_{0}] \subset B_{n}$ is of finite type, and has Watatani index $\delta_{x}^{2n}.$ \cite{HaPeII}
            \item Relative commutants (inside the GNS-spaces $\mathcal {B}(\Ltwo{B_n})\cong \mathcal{B}(\Ltwo{M_n})$) are characteized as
            $$(i_{n}[M_{0}])' \cap M_{n} = \Gr_{0, n} = (i_{n}[B_{0}])' \cap B_{n}.$$  
        \end{enumerate}
\end{proposition}
\begin{proof}
	We need only prove (3) above.  Indeed, from \cite{GJS10}, $(i_{n}[M_{0}])' \cap M_{n} = \Gr_{0, n} $, and a quick diagrammatic computation 
	shows that $\Gr_{0, n} \subset (i_{n}[B_{0}])' \cap B_{n}$.  Therefore, we have:
	$$
		\Gr_{0, n} \subseteq i_{n}(B_{0})' \cap B_{n} = i_{n}(M_{0})' \cap B_{n} \subseteq  i_{n}(M_{0})' \cap M_{n} = \Gr_{0, n} 
	$$
	so every set containment above is an equality.  Note that the first equality holds since $i_{n}(B_{0})$ is weakly dense in $i_{n}(M_{0})$, and hence $i_{n}(B_{0})' = i_{n}(M_{0})'$.
	We should warn the reader that this equality relies on the fact that $\mathcal(\Ltwo{B_0}\supset B_0'' = M_0 \cong \iota_n[B_0]''\subset \mathcal(\Ltwo{B_n})).$ This is 
	rigorously proven in Corollary \ref{corollary::representations} and the discussion before it, using Lemma \ref{lemma::Hilbertifying}, independently of the results between here and there. (See also Remark \ref{rmk::cheat}.)
\end{proof}

\begin{definition}
    For $n\in \N$, define $\mathbb{E}_{n}: M_{n} \rightarrow i_{n}[M_{0}]$ to be the canonical  conditional expectation.  Diagrammatically, we have: 
    \begin{equation}
        \mathbb{E}_{n}
        \left(\begin{tikzpicture}
        [scale=1/15 pt, thick,baseline={([yshift=-\the\dimexpr\fontdimen22\textfont2\relax] current bounding box.center)}]
            \draw[ultra thick, rounded corners] (-5,-5) rectangle (5,5);
            \node at (0,5){$\bullet$};
                \node at (0,0){$m$};
                \draw[ultra thick](0, -5)--(0, -10);
                    \node at (2,-8){$b$};
                \draw[ultra thick] (-3, 5) -- (-3, 10);
                    \node at (6,8){$n$};
                \draw[ultra thick] (3, 5) -- (3, 10);
                    \node at (-6,8){$n$};
        \end{tikzpicture}
        \right) = \frac{1}{\delta_x^{n}} \cdot 
        \begin{tikzpicture}
        [scale=1/15 pt, thick,baseline={([yshift=-\the\dimexpr\fontdimen22\textfont2\relax] current bounding box.center)}]
            \draw[rounded corners, densely dashed, blue] (-8,-7) rectangle (8,12);
            \node at (0,12){$\bullet$};
                \draw[rounded corners, ultra thick] (-5,-5) rectangle (5,5);
                    \node at (0,0){$m$};
                \draw[ultra thick](0, -5)--(0, -12);
                    \node at (3,-10){$b$};
                \draw[ultra thick] (-3, 5) arc(180:0: 3);
                \draw[ultra thick] (-3,18) -- (-3,12);
                    \node at(6,15){$n$};
                \draw[ultra thick] (3,18) -- (3,12);
                    \node at(-6,15){$n$};
                \draw[ultra thick] (-3, 12) arc(180:360: 3);
        \end{tikzpicture}
    \end{equation}
\end{definition}

We immediately see that $\mathbb{E}_{n}$ restricts to a conditional expectation from $B_{n}$ onto $i_{n}[B_{0}]$. Furthermore, Popa's entropic condition for finite index \cite{MR860811} implies that for all positive $b \in B_{n}$ we have $\mathbb{E}_{n}(b) \geq \frac{1}{\delta_x^{n}}b$ .
This means that for each positive $b \in B_{n}$ we get $\|\mathbb{E}_{n}(b)\|_{B_n} \geq \frac{1}{\delta_x^{n}}\|b\|_{B_n}$.

\subsection{Concrete bimodules}\label{section:bimodule}\ \\
Recall that as in Notation \ref{notation::corners}, for $n,m\geq 0$ we denote the corner $p_{m}\wedge {B_\infty}\wedge p_{n}$ by $_m {B}_n$. We shall now describe the $B_0$ bimodules we will use to construct the representation of $\cat.$ 

\begin{definition}\label{bimodule}
	Let $\X =\ _{0}B_{1}$. Then $\X$ is a $B_{0}-B_{0}$ bimodule under the actions
	$$b \rhd \xi \lhd b' = b\wedge\xi \wedge i_{1}(b'),$$
	which diagrammatically respectively places $b$ and $b'$ to the left and right of $\xi.$. We place the following left and right $B_{0}$-valued inner products on $\X$: 
	$$
	    	_{B_0}\langle \xi, \eta \rangle := \xi\wedge\eta^{*}, \text{ and }
	    	\langle \xi\ |\ \eta \rangle_{B_0} := \delta_x\cdot \mathbb{E}_{1}(\xi^{*}\wedge\eta).
	$$
In the right inner product, it is understood that we are canonically identifying $B_{0}$ with $i_{1}[B_{0}]$.
	
	We will diagrammatically depict our $B_{0}$ actions using the Graded product (Remark \ref{rmk::GradProd}).  In terms of diagrams the above inner products are represented as follows: 
	\begin{equation}
		_{B_0}\langle \xi, \eta \rangle = 
	    \begin{tikzpicture}
	    [scale=1/15 pt, thick,baseline={([yshift=-\the\dimexpr\fontdimen22\textfont2\relax] current bounding box.center)}] 
	     		 \node at (27,13){$\bullet$};
	      	 \draw[rounded corners, densely dashed, blue] (12,-10) rectangle (42,13);
	            \draw[rounded corners, ultra thick] (40,-5) rectangle (30,5);
	            \node at (35,0){$\eta^{*}$};
	            \node at (35,5){$\bullet$};
	            \draw[ultra thick] (35,-5) -- (35,-15);
	            \draw[densely dashed] (37,5) -- (37,18);
	            \draw[rounded corners, ultra thick] (24,-5) rectangle (14,5);
	            \node at (19,0){$\xi$};
	            \node at (19,5){$\bullet$};
	            \draw[densely dashed] (17,5) -- (17,18);
	            \draw[ultra thick] (19,-5) -- (19,-15);
	            \draw[ultra thick] (33,5) arc (0:180:6);
	    \end{tikzpicture}
	        \ \ \ \text{   and   }\ \ \
	    \langle \xi\ |\ \eta \rangle_{B_0} =
	    \begin{tikzpicture}
	    [scale=1/15 pt, thick,baseline={([yshift=-\the\dimexpr\fontdimen22\textfont2\relax] current bounding box.center)}] 
	        \node at (27,30){$\bullet$};
	        \draw[rounded corners, densely dashed, blue] (10,-10) rectangle (44,30);
	                \draw[rounded corners, densely dashed, blue] (12,-7) rectangle (42,17);
	                \node at (27,17){$\bullet$};
	                \draw[rounded corners, ultra thick] (40,-5) rectangle (30,5);
	                \node at (35,0){$\eta$};
	                \node at (35,5){$\bullet$};
	                \draw[ultra thick] (35,-5) -- (35,-15);
	                
	                \draw[rounded corners, ultra thick] (24,-5) rectangle (14,5);
	                \node at (19,0){$\xi^{*}$};
	                \node at (19,5){$\bullet$};
	                \draw[ultra thick] (19,-5) -- (19,-15);
	                \draw[ultra thick] (37,15) arc (0:180:10);
	                \draw[ultra thick] (37,5) -- (37,15);
	                \draw[ultra thick] (17,5) -- (17,15);
	                \draw[densely dashed] (33,5) arc (0:180:6);
	                
	                \draw[gray] (19,35) arc(180:360:8);
	                    \node at (37,34){1};
	    \end{tikzpicture}
	\end{equation}
\end{definition}

\begin{proposition}\label{prop:basis}
    	The left and right inner products on $_{B_0}\X_{B_0}$  give uniformly equivalent norms, and $_{B_0}\X_{B_0}$ is complete under both norms. 
\end{proposition}
\begin{proof}
    	Let $_{B_0}\|\cdot\|$ be the norm under the left $B_{0}$ inner product and $\|\cdot\|_{B_0}$ be the norm under the right inner product, according to the notation established in Definition \ref{bimodule}.
	 Note that $_{B_0}\|\xi\|$ coincides with the norm of $\xi$ in the C*-algebra $B_{\infty},$ and also that from \cite{MR860811},
    	$$
        	_{B_0}\| \xi \|^{2} =\ \| \xi^{*}\wedge\xi \|_{B_\infty} =\ \| \xi\wedge\xi^{*}\|_{B_\infty} \leq \delta^{2}_x\cdot\|\mathbb{E}_{1}(\xi\wedge\xi^{*})\|_{B_\infty} \leq \delta^{2}_x\cdot {_{B_0}\|\xi\wedge \xi^{*}\|} = \delta^{2}_x\cdot\|\xi^{*}\wedge\xi\| = \delta^{2}_x\cdot {_{B_0}\| \xi \|}^2.
    $$
    As $\|\delta_x\cdot \mathbb{E}_{1}(\xi\wedge\xi^{*})\| = \|\xi\|_{B_0}^{2}$, it follows that 
    $$
        	 \frac{1}{\sqrt{\delta_{x}}} \, _{B_0}\|\xi\| \leq \|\xi\|_{B_0} \leq \sqrt{\delta_x}\cdot _{B_0}\|\xi\|,
    $$
    hence the norms are equivalent. It follows that $_{0}B_{1} = \X$ is complete in both norms.
\end{proof}

\begin{notation}\label{notation::copies}
	More generally, given $n,l,r \geq 0$ satisfying $l + r = n$, one can define the $B_{0}$ Hilbert bimodule 
    	$_{l}B_{r}$ with left and right $B_{0}$ actions given by
    	$$
       		x \rhd \xi \lhd y = i_{l}(x)\wedge \xi\wedge  i_{r}(y),
    	$$
    	where once more, the action is given by the multiplication inside $B_\infty.$
	 As usual, $_lB_r$ is endowed with left and right $B_{0}$ inner products given by:
	\begin{itemize}
	        \item $_{B_0}\langle \xi, \eta \rangle = \delta_x^{l}\cdot\mathbb {E}_{l}(\xi\wedge\eta^{*})$ and
	        \item $\langle \xi| \eta \rangle_{B_0} = \delta_x^{r}\cdot\mathbb{E}_{r}(\xi^{*}\wedge\eta).$
	\end{itemize}
    	Again, the two inner products give equivalent complete norms on $_lB_r$. (See Proposition \ref{prop:basis}.) 
	Up to $B_0$-valued inner product preserving $B_{0}$ Hilbert C*-bimodule isomorphisms, $_lB_r$ 
	\textbf{is independent of the decomposition of} $n$ \textbf{as} $(l + r)$.

	We can also describe sub-objects of the $_lB_r$: if $q_l\leq p_l = 1_{B_l}$ and $q_r\leq p_r = 1_{B_r}$ are orthogonal projections in $B_l$ and $B_r$, respectively, 
	the span of diagrams of the form
	
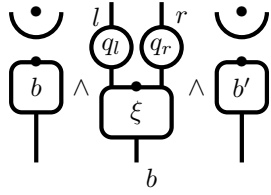
\begin{equation*}
		\begin{tikzpicture}[scale=1/15 pt, thick,baseline={([yshift=-\the\dimexpr\fontdimen22\textfont2\relax] current bounding box.center)}]
		        \draw[ultra thick, rounded corners] (-30,0) rectangle (-20,10);
		        \node at (-25,10){$\bullet$};
		            \node at (-25,5){$b$};
		                \draw[ultra thick] (-25,0) -- (-25,-10);
			 \node at (-25,20){$\bullet$};
			\draw[ultra thick] (-30,20) arc (180:360:5);
		                    		        
		        \node at (-16,5){$\wedge$};
		        
		        \draw[ultra thick, rounded corners] (-12,-5) rectangle (2,5);
		        \node at (-5,5){$\bullet$};
		            \node at (-5,0){$\xi$};
		                \draw[ultra thick] (-10,5) -- (-10,9);
		                	\node at (-13,19){$l$};       
			 	\draw[ultra thick] (-10,13) circle (4);
					\node at (-10,13){$q_l$};
				\draw[ultra thick] (-10,17) -- (-10,22);

		                \draw[ultra thick] (-0,5) -- (-0,9);        
		                	\node at (4,19){$r$};
				\draw[ultra thick] (0,13) circle (4);
					\node at (0,13){$q_r$};
				\draw[ultra thick] (0,17) -- (0,22);

		                \draw[ultra thick] (-5,-5) -- (-5,-10);
		                    \node at (-2,-13){$b$};
		        
		        \node at (7,5){$\wedge$};
		        
		        \draw[ultra thick, rounded corners] (11,0) rectangle (21,10);
		        \node at (16,10){$\bullet$};
		            \node at (16,5){$b'$};
		                \draw[ultra thick] (16,0) -- (16,-10);
			\draw[ultra thick] (11,20) arc (180:360:5);
			 \node at (16,20){$\bullet$};
		\end{tikzpicture}
	\end{equation*}
	defines a pre-Hilbert C* $B_0$-bimodule $(q_l\wedge{_lB_r}\wedge q_r).$ The left and right $B_0$-valued inner products 
	and $B_0$ actions are the obvious ones.
\end{notation}

\begin{notation}\label{remark:changeX}
   	 We shall now describe the \textbf{the conjugate bimodule}, $\overline{\X}.$ 	(See Definition \ref{def::involution}, Example \ref{HappyBims}, and also \cite[p 3443]{MR1624182}.)
    	To do so, we will first show that $\X = {_0B_1}$ is isomorphic to ${_1B_0}$ by defining the map $\phi,$ which \textbf{pushes the dot to the right}:
    	\begin{align}\label{eqn::MoveDot}
        	\phi: \bigoplus_{n = 0}^{\infty} V_{n, 0, 1} \longrightarrow \bigoplus_{n = 0}^{\infty} V_{n, 1, 0} 
	        \hspace{3 mm}\text{ by }\hspace{10 mm} 
	        \phi\left( 
	        \begin{tikzpicture}[scale=1/15 pt, thick,baseline={([yshift=-\the\dimexpr\fontdimen22\textfont2\relax] current bounding box.center)}]
	            \draw[ultra thick, rounded corners] (5,-5) rectangle (-5,5);
	            \node at (0, 5) {$\bullet$};
	                \node at (0,0){$\xi$};
	                \draw[ultra thick] (0,-5) -- (0,-10); 
	                \draw (3, 5) -- (3, 12);
	                \draw[densely dashed] (-3, 5) -- (-3, 12);
	        \end{tikzpicture}\right) 
	        =
	        \begin{tikzpicture}[scale=1/15 pt, thick,baseline={([yshift=-\the\dimexpr\fontdimen22\textfont2\relax] current bounding box.center)}]
	            \draw[rounded corners, ultra thick] (5,-5) rectangle (-5,5);
	            \node at (0, 5) {$\bullet$};
	                \node at (0,0){$\xi$};
	                \draw[ultra thick] (0,-5) -- (0,-10); 
	                \draw (-3, 5) -- (-3, 12);
	                \draw[densely dashed] (3, 5) -- (3, 12);
	        \end{tikzpicture}
    	\end{align}
	 It then follows that for all $\xi, \eta \in _{0}B_{1},$
    	\begin{itemize}
	        \item The inner products are related as $ {_{B_0}\langle\phi(\xi),\ \phi(\eta)\rangle} = i_1\left[_{B_0}\langle \xi,\ \eta \rangle\right]$ and $\langle \xi\ |\ \eta \rangle_{B_0} = i_1\left[\langle\phi(\xi)\ |\ \phi(\eta)\rangle_{B_0}\right],$
        \item the map $ \phi$ extends to a unitary $B_{0}$-bimodule isomorphism of $_{0}B_{1}$ onto $_{1}B_{0}$, and
        \item for any $\xi\in \X,$ we have that $\phi(\xi)^* = \phi^{-1}(\xi^*).$
    \end{itemize}
    
    	More generally, observe that an obvious definitions of $\phi,$ analogous to that in Equation \ref{eqn::MoveDot} will exhibit $B_0$-bilinear
	 unitaries with similar properties realizing
		 $_0B_n\cong\hdots\cong\ _lB_r\ \cong\hdots \cong\ \ _nB_0,$ where $l+r = n,$ which we still denote by $\phi$.
	Thus, for every $n\in \N$, we have 
		$$\phi^n:{_0B_n}\longrightarrow {_nB_0}$$
	is a unitary $B_0$-bimodule isomorphism. Additionally, we introduce the \textbf{congujation map}
	\begin{align*}
		(\,\overline{\, \cdot \,}\,) :\ &{_0B_n}\longrightarrow \overline{_0B_n}\\
		&\xi\mapsto \phi^{-n}(\xi^*) := \overline{\xi}.
	\end{align*} 
	Notice that this defines an $B_0$-conjugate linear unitary isomorphism.
	Using an analogous property to the third bullet point, it follows that $\overline{\overline \xi} = \xi$ on $_0B_n,$ so $_0B_n$ is identical to its 
	double dual as $B_0$-bimodules.
	Finally, the isomorphism 
	\begin{align*}
		\phi^n : \overline {_0B_n}&\longrightarrow {_nB_0}\\
		  &\overline\xi\mapsto \xi^*,
	\end{align*}
	is also a $B_0$-bimodule unitary isomorphism.
\end{notation}


The following propositions will establish the most important algebraic properties of the bimodule $\X:$

\begin{proposition}\label{prop:bases}
    The bimodule $_{B_0}\X_{B_0}$ has finite left and right $B_{0}$-bases; i.e. $\X$ is fgp. (See Definition \ref{dfn::basis}.)
\end{proposition}

\begin{proof}
    	We first show that $\X$ has a finite left $B_{0}$-basis. 
	Note that for $\xi, v\in \X,$ an expression of the form $_{B_{0}}\langle \xi, v \rangle \rhd v$ is the same as the product 
	$ \xi\wedge v^{*}\wedge v$.  Observe that $\xi\wedge v^{*} \in B_{0}$ and $v^{*}\wedge v \in B_{1}$.

    Consider $I := \, _{1}B_{0}\cdot \, _{0}B_{1} := \spann\{ b'b |\ b' \in{_{1}B_{0}} \text{ and } b \in  {_{0}B_{1}} \}$. Note that $I$ is a two-sided ideal in $B_{1}$  
    containing all diagrams of the form:
        \begin{equation}
            \begin{tikzpicture}[scale=1/15 pt, thick,baseline={([yshift=-\the\dimexpr\fontdimen22\textfont2\relax] current bounding box.center)}] 
                \draw[rounded corners, densely dashed, blue] (-7,-10) rectangle (20,12);
                \node at (6,12){$\bullet$};
                    \draw (-3,5) -- (-3, 15);
                    \draw[densely dashed] (10,5) arc (0:180:4);
                    \draw[rounded corners, ultra thick] (5,-5) rectangle (-5,5);
                    \node at (0,5){$\bullet$};
                        \node at (0,0){$b'$};
                        \draw[ultra thick] (0,-5) ..controls (0,-10) and (4,-11).. (4,-15);
                    
                    \draw[rounded corners, ultra thick] (18,-5) rectangle (8,5);
                    \node at (13,5){$\bullet$};
                        \draw (16,5) -- (16, 15);
                        \node at (13,0){$b$};
                        \draw[ultra thick] (13,-5) .. controls (13,-10) and (9,-11) .. (9,-15);
                    
            \end{tikzpicture}
        \end{equation}
     	If $I$ were proper, then since $B_{1}$ is unital, it follows that the closure of $I$, $\overline{I}$ would also be a proper ideal in $B_{1}$. 
	 However, since $B_1$ is simple (Lemma \ref{lemma::simplicity}), it follows that $\overline{I} = B_1,$ hence $I = B_{1}$.  
	Therefore, there exists finite sets $\{w_j\}_{j = 1}^{n}, \{w'_j\}_{j = 1}^{n}\subset \X$ such that
     	\begin{equation}\label{eqn::one}
        	p_1 = 1_{B_{1}} = \sum_{j=1}^{n} w_{j}^{*}\wedge w'_{j}.
     	\end{equation}
     	From this equation, it follows that for any $\xi \in \X$, 
     	$$
        	\xi = \sum_{j=1}^{n} \xi\wedge w_{j}^{*}\wedge w'_{j} = \sum_{j=1}^{n}\, _{B_{0}}\langle \xi, w_{j} \rangle \rhd w'_{j}.
    	 $$
		By Lemma 1.7 in \cite{MR1624182}, these two sets can be made the same. This is,  there exists  a finite set $\{v_j\}_{j = 1}^{n}\subset \X$ such that  for any $\xi \in \X$, 
     	$$
        	\xi = \sum_{j=1}^{n} \xi\wedge v_{j}^{*}\wedge v_{j} = \sum_{j=1}^{n}\, _{B_{0}}\langle \xi, v_{j} \rangle \rhd v_{j}.
    	 $$
     	showing $\{v_j\}_{j=1}^{n}$ defines a left $B_0$-basis for $\X$. Moreover, one can easily show that this set also satisfies
	\begin{equation}\label{eqn::ONEone}
		\sum_{j=1}^n v_j^*\wedge v_j = p_1 = 1_{B_1}.
	\end{equation}
	      Now it suffices to exhibit a right $B_{0}$-basis for $\X$. To do so, we will first show that $\{\zeta_j := \phi[\phi(v_j)^*]\}_{j =1}^{n}$ is a right basis for ${_1B_0}.$
	 By the last property of the map $\phi$, that \textit{pushes the dot to the right} introduced in Remark \ref{remark:changeX},
	 we obtain that $\zeta_j = v_j^*.$ We then have that
     	$$
		1_{B_1} = \sum_{j=1}^{n}\zeta_j\wedge \zeta_j^*, 
	$$
     	so for an arbitrary $\xi\in \X = {_0B_1}$ we obtain that
     	\begin{equation*}
	         \phi(\xi) 
	         =
	         1_{B_1}\wedge \phi(\xi)
	         =
	         \sum_{j=1}^n \zeta_j \lhd \left\langle \zeta_j\ |\ \phi(\xi) \right\rangle_{B_0}
	        =
	        \sum_{j=1}^n
	        \begin{tikzpicture}[scale=1/15 pt, thick,baseline={([yshift=-\the\dimexpr\fontdimen22\textfont2\relax] current bounding box.center)}]
	            \draw[densely dashed, rounded corners, orange](-18,-13) rectangle (36,16);
	            \node at(9,16){$\bullet$};
	                \draw[rounded corners, ultra thick] (-15,-5) rectangle (-5,5);
	                \node at (-10,5){$\bullet$};
	                    \node at (-10,0){$v_j^*$};
	                    \draw[ultra thick] (-10,-5) -- (-10,-15);
	                    \draw (-13,5) -- (-13,20);
	                    \draw[densely dashed](-7,5) arc(180:0:8);
	                
	                \draw[rounded corners, densely dashed, blue] (5,-10) rectangle (33,12);
	                \node at (19,12){$\bullet$};
	                    \draw[rounded corners, ultra thick] (8,-5) rectangle (18,5);
	                    \node at (13,5){$\bullet$};
	                        \node at(13,0){$v_j$};
	                        \draw[ultra thick] (13,-5) ..controls (13,-11) and (18,-11)..  (18,-15);
	                        
	                \draw (16,5) arc(180:0:3);
	                \draw[rounded corners, ultra thick] (20,-5) rectangle (30,5);
	                    \node at (25,5){$\bullet$};
	                        \node at(25,0){$\xi$};
	                        \draw[ultra thick] (25,-5) ..controls (25,-10) and (20,-11).. (20,-15);
	                        \draw[densely dashed] (28,5) -- (28,20);
	        \end{tikzpicture}
     	\end{equation*}
     	therefore, by the $B_0$-linearity of $\phi^{-1},$ we get
    	\begin{equation*}
	        \xi=
	        \sum_{j=1}^n
	        \begin{tikzpicture}[scale=1/15 pt, thick,baseline={([yshift=-\the\dimexpr\fontdimen22\textfont2\relax] current bounding box.center)}]
	            
	            \draw[ultra thick, rounded corners] (-15,-5) rectangle (-5,5);
	            \node at (-10,5){$\bullet$};
	                \node at (-10,0){$v_j^*$};
	                \draw (-7,5) -- (-7,15);
	                \draw[densely dashed] (-13,5) --(-13,15);
	                \draw[ultra thick] (-10,-5) -- (-10,-15);
	                
	            \node at (0,0){$\wedge$};
	            
	            \draw[rounded corners, densely dashed, blue] (5,-12) rectangle (36,12);    
	            \node at (20,12){$\bullet$};
	            \draw [gray] (16,15) -- (16,14) arc(180:360:4) -- (24,15);
	                \draw[ultra thick, rounded corners] (8,-5) rectangle (18,5);            
	                \node at (13,5){$\bullet$};
	                   \node at (13,0){$v_j$};
	                   \draw[ultra thick] (13,-5) .. controls (13,-10) and (18,-10).. (18,-15); 
	                    
	                \draw (9,5) arc(180:90:4) -- (28,9) arc (90:0:4);
	                \draw[densely dashed] (16,5) arc (180:90:2) -- (23,7) arc(90:0:2);            
	                            
	                \draw[ultra thick, rounded corners] (23,-5) rectangle (33,5);
	                \node at (28,5){$\bullet$};
	                    \node at (28,0){$\xi$};
	                    \draw[ultra thick] (28,-5) .. controls (28,-10) and (23,-10) .. (23,-15);
	        \end{tikzpicture}
	        =\sum_{j=1}^n  \left[ \phi(v_j)^*\right]\lhd \left\langle [ \phi(v_j)^*]\ \large{|}\ \xi \right\rangle_{B_0}.
    	\end{equation*}
   	 Which proves that $\left\{\phi(v_j)^*  := u_j \right\}_{j=1}^{n}$ is a right $B_0$-basis for $\X$.
\end{proof}

\begin{proposition}\label{prop::normalized}
    	The left and right indices of $\X$ as a $B_{0}-B_{0}$ Hilbert bimodule match. More precisely, we have that r-Ind$(\X) = \delta_x = \text{l-Ind}(\X)$. 
	 Therefore, $\X$ is normalized. (See Definition \ref{index} and the following paragraph.)
\end{proposition}
\begin{proof}
    	Let $\{v_{j}\}_{j = 1}^{n}$ be a left $B_{0}$-basis of $\X = {_{0}B_{1}}$.  Using the decomposition $B_1 = {_1B_0}\wedge{_0B_1}$ described in the 
	proof of Proposition \ref{prop:basis}, one can easily see that in $B_{1},$ the $v_{j}$ can be chosen to satisfy $\sum_{j=1}^{n} v_{j}^{*}v_{j} = 1_{B_{1}}$.  This means that 
    	$$
        \sum_{j=1}^{n} \langle v_{j}\ |\ v_{j}\rangle_{B_{0}} =  \delta_x\cdot\sum_{j=1}^{n}  \mathbb{E}_{1}(v_{j}^{*}\wedge v_{j}) = \delta_x \cdot\mathbb{E}_{1}\left( \sum_{j=1}^{n} v_{j}^{*}\wedge v_{j} \right) = \delta_x\cdot \mathbb{E}_{1}(1_{B_{1}}) = \delta_x.
    	$$
    	As for the right index, use the basis constructed in the proof of Proposition \ref{prop:bases} and proceed as above.
\end{proof}

To show that the bimodule $\X$ is minimal (Definition \ref{minimal}) we will need the next proposition:

\begin{proposition}\label{prop::RightCompacts}
    	Considering the right Hilbert $B_0$-bimodule $\X_{B_{0}} = {_0B_1}_{B_0},$ we have $\mathcal{B}^*(\X_{B_{0}}) = \mathcal{K}(\X_{B_{0}}) = B_{1}$ where $B_{1}$
	 acts with its canonical left action. 
	Picturing $\X$ as a left Hilbert $B_{0}$ module, $_{B_{0}}\X$, we have $\mathcal{B}^*(\, _{B{0}}\X) = \mathcal{K}(\, _{B{0}}\X) = B_{1}^{op},$ where $B_{1}$ acts with its canonical right action. 
\end{proposition}
\begin{proof}
	Consider $_1B_0 = \phi[\X]$ and arbitary elements $\zeta, \xi, \eta\in\X.$ Recall that the right inner product is given by
	 $\langle \phi(\eta)\ |\ \phi(\zeta)\rangle_{B_0} = \phi(\eta)^*\wedge\phi(\zeta).$
	Since we know by Proposition \ref{prop:bases} that $\X$ is fgp, by Lemma 1.7 in \cite{MR1624182}, it follows that
	 $\mathcal{B}^*(\X_{B_0}) = \mathcal{K}^*(\X_{B_0}),$ which in turn is generated by the rank 1 operators
	 $|\phi(\xi)\rangle\langle \phi(\eta)| (\phi(\zeta)) = \phi(\xi)\lhd \langle \phi(\eta)\ |\ \phi(\zeta)\rangle_{B_0} = \phi(\xi) \wedge \phi(\eta) ^* \wedge \phi(\zeta).$
	This means that $|\phi(\xi)\rangle\langle \phi(\eta)|$ is the multiplication operator $\phi(\xi)\wedge \phi(\eta)^*.$
	 As discussed in the proof of Proposition \ref{prop:bases}, $B_{1}$ is the linear space spanned by these products. 
	The first statement follows, and the second statement follows by a similar argument. 
\end{proof}

\begin{remark}
	 Similar reasoning will show that $\mathcal B^*({_lB_r}) \cong \mathcal B^*({_nB_0}) = B_n,$ whenever $n = (l+r).$  If $\cY$ is a $B_{0}-B_{0}$ submodule of $_nB_0$, then 
	 $$
	 	\Y = p[\, _nB_0]
	 $$ 
for some projection $p \in B_{n} \cap i_{n}[B_{0}]' = \Gr_{0, n}$.	It follows that if $x \in \mathcal{B}^{*}(\Y)$, then $x$ can be extended to be a bounded operator $\tilde{x}$ on $ _nB_0$ satisfying $\tilde{x} = 0$ on $ (1 - p) [{_nB_0}]$.  It is easily seen that $\tilde{x}$ is adjointable, $\tilde{x}^{*} = 0$ on $(1-p) [\,_nB_0]$ and  $\tilde{x}^* = x^{*}$ on $\Y$.  Consequently, it follows that $\tilde{x} \in pB_{n}p$.  Conversely, any operator $z \in pB_{n}p$ is in $\mathcal{B}^{*}(\Y)$; therefore $\mathcal{B}^*(\Y) = pB_{n}p$. 
	 
\end{remark}


\begin{corollary}\label{rem::endomorphisms}
	The endomorphism algebra of $B_0$-bimodule maps on $\X = {_0B_1}$ is characterized as
	$$
		\End_{B_0-B_0}\left(\X\right) \cong B_1\cap \Gr_{0,1}.
	$$
\end{corollary}
\begin{proof}
	We replace $\X$ by the isomorphic bimodule $\phi[\X] = {_1B_0}.$  Let $\xi\in {_0B_1}$ be arbitrary, and note that if
	 $T \in \Gr_{0,1} \cap B_{1}$, then clearly $T \in \End_{B_0-B_0}\left({_1B_0}\right)$ via the usual action:
	    \begin{equation*}
		        T\left(\phi(\xi)\right)
 			=
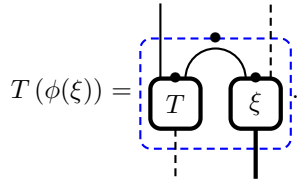
			 \begin{tikzpicture}[scale=1/15 pt, thick,baseline={([yshift=-\the\dimexpr\fontdimen22\textfont2\relax] current bounding box.center)}]  
	            		\draw[densely dashed, rounded corners, blue] (12,-9) rectangle (42,13);
	                	\draw[rounded corners, ultra thick] (40,-5) rectangle (30,5);
	                    	\node at (35,0){$\xi$};
		                 \node at (35,5){$\bullet$};
		                 \draw[densely dashed] (38,5) -- (38,20);
		                 \draw[ultra thick] (35,-5) -- (35,-15);

		                 \node at (27, 13){$\bullet$};
	                 	 \draw[rounded corners, ultra thick] (24,-5) rectangle (14,5);
	                    	 \node at (19,0){$T$};
	                    	 \node at (19,5){$\bullet$};
	                    	 \draw (16,5) -- (16,20);
	                     	 \draw[densely dashed] (19,-5) -- (19,-15);
	                    	 \draw (33,5) arc (0:180:6);
	            \end{tikzpicture}.
	    \end{equation*}
	  Conversely, we claim that if $T \in \End_{B_0-B_0}\left({_1B_0}\right),$ then $T \in \Gr_{0, 1} \cap B_{1}$. On the one hand, in particular 
	$T\in \mathcal B^*({_1B_0}_{B_0}),$ so $T\in B_1,$ by Proposition \ref{prop::RightCompacts}. On the other, $T \in \End_{B_0-B_0}\left({_1B_0}\right)$ 
	implies $T \in \mathcal{B}^*({_1B_0}_{B_{0}}) \cap (i_{1}[B_{0}])'.$ 
	As a consequence, $T\in B_1 \cap \left( \mathcal{B}^*({_1B_0}_{B_{0}}) \cap (i_{i}[B_{0}])'  \right) = B_{1} \cap (i_{i}[B_{0}])' = B_{1} \cap \Gr_{0, 1}$, 
	in light of (the proof of) Proposition \ref{RelCommutants}.
\end{proof}

\begin{proposition}\label{prop::Minimality}
    The Hilbert $B_0$-bimodule $\X$ is minimal (See Definition \ref{minimal}).
\end{proposition}
\begin{proof}
    	Let $T \in \End_{B_0-B_0}( \phi[\X] ) = \Gr_{0, 1} \cap B_{1}$, where again $\phi[\X] = \phi[{_0B_1}] = {_1B_0}$.
	Let $\{v_{j}\}_{j = 1}^{n}$ be the left $B_{0}$-basis for ${_0B_1}$ constructed in the proof of Proposition \ref{prop:bases}, and pretending for for 
	the moment that each $v_{i}$ is actually in $\Gr_{\infty}$, we have (formally):
    	\begin{equation*}
	        \sum_{j=1}^{n} \langle \phi(v_{j})\ |\ T (\phi(v_{j})) \rangle_{B_{0}} = \sum_{j=1}^{n} 
	        \begin{tikzpicture}[scale=1/15 pt, thick,baseline={([yshift=-\the\dimexpr\fontdimen22\textfont2\relax] current bounding box.center)}]  
	            \draw[OhioRed,rounded corners, fill=OhioRed!10, ultra thick] (40,-5) rectangle (30,5);
	            \node at (35,5){$\bullet$};
	                \node at (35,0){$v_{j}$};
	                \draw[OhioRed, ultra thick] (35,-5) -- (35,-12);
	            \draw[rounded corners, ultra thick] (24,-5) rectangle (14,5);
	            \node at (19,5){$\bullet$};
	                \node at (19,0){$T$};
	                \draw (33,5) arc (0:180:5);
	                \draw[densely dashed] (19,-5) -- (19,-12);
	            \draw[OhioRed, ultra thick, rounded corners, fill=OhioRed!10] (-2,-5) rectangle (8,5);
	            \node at (3,5){$\bullet$};
	                \node at (3,0){$v_{j}^{*}$};
	                \draw (6,5) arc(180:0:5);
	                \draw[OhioRed, ultra thick] (3,-5) -- (3,-12);
	        \end{tikzpicture}
	        =
	        \sum_{j=1}^{n}  
	        \begin{tikzpicture}[scale=1/15 pt, thick,baseline={([yshift=-\the\dimexpr\fontdimen22\textfont2\relax] current bounding box.center)}]  
	            \draw[OhioRed,ultra thick, rounded corners,fill=OhioRed!10] (40,-5) rectangle (30,5);
	            \node at (35,5){$\bullet$};
	                \node at (35,0){$v_{j}$};
	                \draw[OhioRed, ultra thick] (35,-5) -- (35,-12);
	            \draw[OhioRed, ultra thick, rounded corners, fill=OhioRed!10] (24,-5) rectangle (14,5);
	            \node at (19,5){$\bullet$};
	                \node at (19,0){$v_{j}^{*}$};
	                \draw[OhioRed, ultra thick] (19,-5) -- (19,-12);
	            \draw[ultra thick,rounded corners] (-2,-5) rectangle (8,5);
	            \node at (3,5){$\bullet$};
	                \node at (3,0){$T$};
	                \draw (6,5) arc(180:90:6)-- (26, 11) arc(90:0:6);
	                \draw (0, 5) arc(0: 180: 2) -- (-4, -5) arc(180:270: 2) -- (9, -7) arc(270:360: 2) -- (11, 5)  arc(180:90:3) -- (19,8) arc(90:0:3);
	          \end{tikzpicture}
	          =
	        \begin{tikzpicture}[scale=1/15 pt, thick,baseline={([yshift=-\the\dimexpr\fontdimen22\textfont2\relax] current bounding box.center)}]
	             \draw[ultra thick, rounded corners] (-2,-5) rectangle (8,5);
	             \node at (3,5){$\bullet$};
	                \node at (3,0){$T$};
	                \draw (0, 5) arc(0: 180: 2) -- (-4, -5) arc(180:270: 2) -- (8, -7) arc(270:360: 2) -- (10, 5) arc(0:180:2);
	        \end{tikzpicture}  \in \mathbb{C}.
	\end{equation*}
    	Notice that we are making use of the maps $\phi$, that \textit{push the dot to the right} $\phi: {_lB_r}\longrightarrow {{_{(l+1)}B_{(r-1)}}}$, defined in
	Notation \ref{remark:changeX}. In our particular case and to make this computation rigorous, we will make use of
	\begin{align*}
	        \phi: B_{1} \longrightarrow \, _{2}B_{0}
	        \hspace{5mm}
	        \text{ given by }
	        \hspace{5mm}
	        \begin{tikzpicture}[scale=1/15 pt, thick,baseline={([yshift=-\the\dimexpr\fontdimen22\textfont2\relax] current bounding box.center)}]
	        \draw[rounded corners, ultra thick] (24,-5) rectangle (14,5);
	                \draw (16, 5)--(16, 12);
	                \draw (22, 5)--(22, 12);
	                \draw[ultra thick] (19, -5)--(19, -10);
	                \node at (19, 5) {$\bullet$};
	        \end{tikzpicture}
	            \longmapsto
	        \begin{tikzpicture}[scale=1/15 pt, thick,baseline={([yshift=-\the\dimexpr\fontdimen22\textfont2\relax] current bounding box.center)}]
	        \draw[rounded corners, ultra thick] (24,-5) rectangle (14,5);
	                \draw[densely dashed] (22,5) -- (22,12);
	                \draw (15, 5)--(15, 12);
	                \draw (17, 5)--(17, 12);
	                \draw[ultra thick] (19, -5)--(19, -10);
	                \node at (19, 5) {$\bullet$};
	        \end{tikzpicture}.
	\end{align*}
    Approximate each $v_{j}$ in norm by a sequence $\{v^{k}_{j}\}_{k=1}^{\infty}\subset \Gr_{\infty} \cap\ {_0B_1}$.  If we let $T'$ be the element
    \begin{equation*}
        T'
        :=
        \begin{tikzpicture}[scale=1/15 pt, thick,baseline={([yshift=-\the\dimexpr\fontdimen22\textfont2\relax] current bounding box.center)}]
            \draw[rounded corners, densely dashed, blue] (-10,10) rectangle (14,-11);
                \draw[ultra thick, rounded corners] (-2,-5) rectangle (8,5);
                \node at (3,0){$T$};
                \draw (0, 5) arc(0: 180: 2) -- (-4, -5) arc(180:270: 2) -- (8, -7) arc(270:360: 2) -- (10, 15);
                \draw (3, 5)--(3, 15);
                \node at (-2, 10) {$\bullet$};
       \end{tikzpicture},
    \end{equation*}
    then the above diagrams imply that for each $k\in \N,$ 
	\begin{equation*}
		\sum_{j=1}^{n} \langle \phi(v_{j}^{k})\ |\ T[\phi ( v_{j}^k)] \rangle_{B_{0}}
	        =
        	T'\wedge \phi \left[(v_{j}^{k})^{*}v_{j}^{k}\right].
	\end{equation*}
    	Letting $k \rightarrow \infty$, we obtain
    	\begin{equation}\label{eqn::mini1}
	        \lim_{k \rightarrow \infty}\sum_{j=1}^{n} \langle \phi(v_{j}^{k}) \ |\ T[\phi(v_{j}^{k})] \rangle_{B_{0}}
	        =
	        T'\wedge
	        \begin{tikzpicture}[scale=1/15 pt, thick,baseline={([yshift=-\the\dimexpr\fontdimen22\textfont2\relax] current bounding box.center)}]
	            	\draw[rounded corners, densely dashed, blue] (-5,-7) rectangle (15,2);
	            	\node at (5, 2) {$\bullet$};
			\draw  (-3, 0) arc (180:360:3);
	                \draw (-3,0) -- (-3,5);
	                \draw (3,0) -- (3,5);
	        \end{tikzpicture}
	        =
	        \begin{tikzpicture}[scale=1/15 pt, thick,baseline={([yshift=-\the\dimexpr\fontdimen22\textfont2\relax] current bounding box.center)}]
			\draw[rounded corners, ultra thick] (-2,-5) rectangle (8,5);
	            	\node at(3,5){$\bullet$};
	                \node at (3,0){$T$};
	            	\draw (0, 5) arc(0: 180: 2) -- (-4, -5) arc(180:270: 2) -- (8, -7) arc(270:360: 2) -- (10, 5) arc(0:180:2);
           	\end{tikzpicture}
	\end{equation}
 	Performing a similar computation for the the right $B_0$-basis for $_1B_0$ given by $\{\phi[\phi(v_j)^* = v_j^* ]\}_{j=1}^n$ we obtain that
	\begin{equation}\label{eqn::mini2}
	        \sum_{j=1^n}\ _{B_0}\langle\ T(\phi[\phi(v_j)^*]), (\phi[\phi(v_j)^*])\ \rangle
	        =
	        \begin{tikzpicture}[scale=1/15 pt, thick,baseline={([yshift=-\the\dimexpr\fontdimen22\textfont2\relax] current bounding box.center)}]
	            \draw[rounded corners, ultra thick] (-5,-5) rectangle (5,5);
	                \node at (0,0){$T$};
	                \draw (4,5) arc (0:180:4);
	        \end{tikzpicture}
	\end{equation}
    	To conclude, we must show that the diagrams obtained in Equation \eqref{eqn::mini1} and Equation \eqref{eqn::mini2} are equal. 
	Since $T$ directly comes from sums of diagrams in the category $\cat$, using Frobenius reciprocity (Proposition \ref{prop::AFalgs}) together
	 with the canonical spherical structure of $\cat$ (Remark \ref{rmk::biinv}), we immediatly see that these two diagrams correspond to the left and right traces of $T,$
	 which are identical. This observation completes the proof.
\end{proof}

We now study tensor powers of our bimodule $_{B_0}\X_{B_0}$. In the following lemma we describe \textbf{the bimodule structure of the} $n$\textbf{-fold fusion} $\X^{\boxtimes_{B_0}n}$, which will turn out extremely useful in the sequel. For a more detailed view on this fusion with respect to $B_0$, we refer the reader to Example \ref{HappyBims} above.

\begin{lemma}[\cite{HaPeI}, Proposition 4.11]\label{lemma::isomorphism}
    For each $n\in \N$, there is a $B_0$-bilinear unitary, which on diagrams is given by (See also Proposition 4.11 in \cite{HaPeI}.)
    \begin{equation}
        \begin{tikzpicture}[scale=1/15 pt, thick,baseline={([yshift=-\the\dimexpr\fontdimen22\textfont2\relax] current bounding box.center)}] 
            \node at (0,30){$\Psi_n:\ _{B_0}(\X^{\boxtimes_{B_0}n})_{B_0} \longrightarrow\ _{B_0}(_0B_n)_{B_0}$};
            
            \draw [rounded corners, ultra thick](-67,-5) rectangle (-57,5);
            \node at (-62,5){$\bullet$};
                \node at (-62,0){$\xi_1$};
                \draw (-59,5) -- (-59,15);
                \draw[densely dashed] (-65,5) -- (-65,15);
                \draw [ultra thick] (-62,-5) -- (-62,-15);
                    \node at (-58,-12){$b_1$};
            \node at(-51,0){$\boxtimes_{B_0}$};
            
            \draw [rounded corners, ultra thick](-45,-5) rectangle (-35,5);
            \node at (-40,5){$\bullet$};
                \node at (-40,0){$\xi_2$};
                \draw (-37,5) -- (-37,15);
                \draw[densely dashed] (-43,5) -- (-43,15);
                \draw [ultra thick] (-40,-5) -- (-40,-15);
                    \node at (-36,-12){$b_2$};
            
            \node at (-20,0){$\boxtimes_{B_0}\cdots\boxtimes_{B_0}$};
            
            \draw [rounded corners, ultra thick](-5,-5) rectangle (5,5);
            \node at (0,5){$\bullet$};
                \node at (0,0){$\xi_n$};
                \draw (3,5) -- (3,15);
                \draw[densely dashed] (-3,5) -- (-3,15);
                \draw[ultra thick] (0,-5) -- (0,-15);
                    \node at (4,-12){$b_n$};
            
            \node at (13, 0){$\longmapsto$};
            
            \draw[densely dashed,blue,rounded corners] (23,-15) rectangle (72, 15);
            \node at (40,15){$\bullet$};
                \draw [rounded corners, ultra thick](25,-5) rectangle (35,5);
                \node at (30,5) {$\bullet$};
                    \node at (30,0){$\xi_1$};
                    \draw[densely dashed] (27,5) .. controls (27,10) and (35,15) .. (35,20);
                    \draw (33,5) .. controls (33,10) and (58,15) ..  (58,20);
                    \draw[ultra thick] (30,-5) .. controls (30,-10) and (45,-15) .. (45,-20);
                
                \draw [rounded corners, ultra thick](40,-5) rectangle (50,5);
                    \node at (45,0){$\xi_2$};
                    \draw (48,5) .. controls (48,10) and (60,15) .. (60,20);
                    \draw[ultra thick] (45,-5) .. controls (45,-10) and (47,-15) .. (47,-20);
                        
                \node at(55,0){$\hdots$};
                
                \draw[rounded corners, ultra thick](70,-5) rectangle (60,5);
                    \node at (65,0){$\xi_n$};
                    \draw (68,5) .. controls (68,10) and (63,15) .. (63,20);
                    \draw[ultra thick] (65,-5) .. controls (65,-10) and (57,-15) .. (57,-20);
            
            \node at (68,18){$n$};
            \node at(80,-18){\footnotesize{$(b_1+b_2+\cdots+b_n)$}};
        \end{tikzpicture}.
    \end{equation}
\end{lemma}
\begin{proof}
    	It is clear that $\Psi_n$ is well-defined and $B_0$-bilinear on $[\Gr_\infty \cap\ {_0B_1}]^{\odot n} .$ Here, we are balancing our tensor product $\odot$ by $\Gr_\infty \cap B_0.$ 
	From the definitions of the left and right $B_0$-valued inner products on $\X^{\boxtimes_{B_0} n}$ and Notation \ref{notation::copies}, it is easy to see 
	that $\Psi_n$ is isometric on diagrams. To extend $\Psi_n$ to a unitary to all of $\X^{\boxtimes_{B_0}n},$ we observe that $[\Gr_\infty \cap\ {_0B_1}]^{\odot n}\subset \X^{\odot n}$ is dense, 
	and that the latter space is clearly dense in $\X^{\boxtimes_{B_0}n}.$ Middle $B_0$-linearity now automatically follows from $\Gr_\infty \cap B_0$-linearity.
    
    	To conclude the proof, it suffices to show that $\Psi_n$ is surjective. To see this, recall Lemma \ref{lemma::simplicity} asserting that $B_1$ is simple and unital, and moreover, 
	the unit element in $B_1$ is a finite sum $1_{B_1} = p_1 = \sum_{j=1}^n v_j^*\wedge v_j$, where each $v_j\in\X.$ (See Equation \ref{eqn::ONEone}.) We shall now proceed by 
	means of induction. If $n = 2,$ for an arbitrary element $b\in {_0B_2},$ we have (formally):
    	\begin{equation}
        	\begin{tikzpicture}[scale=1/10 pt, thick,baseline={([yshift=-\the\dimexpr\fontdimen22\textfont2\relax] current bounding box.center)}] 
            		\draw[rounded corners, ultra thick] (-5,-5) rectangle (5,5);
            		\node at (0,5){$\bullet$};
		                \node at (0,0){$b$};
		                \draw (2,5) -- (2,15);
		                \draw (4,5) -- (4,15);
		                \draw[densely dashed] (-3,5) --(-3,15);
		                \draw[ultra thick] (0,-5) -- (0,-15);
        	\end{tikzpicture}
        	=
        	\begin{tikzpicture}[scale=1/10 pt, thick,baseline={([yshift=-\the\dimexpr\fontdimen22\textfont2\relax] current bounding box.center)}]
		            \draw[densely dashed, rounded corners, blue] (-7,-9) rectangle (18,13);
		            \node at (0,13){$\bullet$};
		                \draw[rounded corners, ultra thick] (-5,-5) rectangle (5,5);
		                \node at (0,5){$\bullet$};
		                    \node at (0,0){$b$};
		                    \draw (2,5) -- (2,15);
		                    \draw[densely dashed] (-2,5) -- (-2,15);
		                    \draw (4,5) -- (4,8);
		                    \draw (10,8) arc(0:180:3);
		                    \draw[OhioRed] (10,8) arc(180:360:3);
		                    \draw (16,8) -- (16,15);
		                    \draw[ultra thick] (0,-5) -- (0,-15);
	        \end{tikzpicture}
	        =
	        \sum_{j=1}^n
	        \begin{tikzpicture}[scale=1/10 pt, thick,baseline={([yshift=-\the\dimexpr\fontdimen22\textfont2\relax] current bounding box.center)}] 
		            \draw [densely dashed, rounded corners, blue] (-8,-8) rectangle (33,15);
		            \node at (10,15){$\bullet$};
		                \draw[dashed, rounded corners, orange] (-6,-7) rectangle (18,11);
		                \node at (4,11){$\bullet$};
		                    \draw[rounded corners, ultra thick] (-5,-5) rectangle (5,5);
		                    \node at (0,5){$\bullet$};
		                        \node at (0,0){$b$};
		                        \draw[ densely dashed] (-3,5) -- (-3,17);
		                        \draw (2,5) .. controls(2,11) and (16,11).. (17,17);
		                        \draw [ultra thick] (0,-5) -- (0,-12);
		                    \draw[rounded corners, ultra thick, OhioRed, fill=OhioRed!10] (7,-5) rectangle (17,5);
		                    \node[OhioRed] at (13,5){$\bullet$};
		                        \node[OhioRed] at (12,0){$v_j^*$};
		                        \draw[OhioRed] (9,5) arc (0:180:3);
		                        \draw[OhioRed, ultra thick] (12,-5) ..controls (12,-9) and (15,-9).. (15,-12);
		                        \draw[OhioRed, densely dashed] (21.5,5) arc (0:180:3);              
		                \draw[dashed, rounded corners, orange] (19,-7) rectangle (31,10); 
		                \node at(22,10){$\bullet$};
		                    \draw[rounded corners,ultra thick, OhioRed,fill=OhioRed!10] (20,-5) rectangle (30,5);
		                    \node[purple] at (25,5){$\bullet$};
		                        \node[purple] at (25,0){$v_j$};
		                        \draw[OhioRed] (29,5) ..controls (29,7) and (24,7).. (24,10);
		                        \draw (24,10) ..controls (24,14) and (21,15).. (21,17);
		                        \draw[ultra thick, OhioRed] (25,-5) ..controls (25,-9) and (21,-9).. (21,-12);
	        \end{tikzpicture}
   	 \end{equation}
    	Notice that in the last term of the above equation, the elements inside the orange boxes are elements in $\X$. Thus, the element $\sum_{j=1}^{n} (b\wedge v_j^*)\boxtimes v_j \in \X{\fuse{B_0}\X}$ maps to the 
	arbitrarily chosen $b$ under $\Psi_2.$ This proves the claim for $n = 2$. The inductive step of the proof is the same as the base case, with the single difference that now $b$ has $n$ strings to the right. 
	These details are simple and thus not explained here.
\end{proof}

\begin{notation}\label{notation::Ltwo}
    	From the fgp $B_0$-bimodules ${_lB_r}$ we can perform a type of \textbf{GNS-construction} to obtain a Hilbert space. 
	For arbitrary $i, j\geq 0$, we define $\Ltwo{{_lB_r}} := \Ltwo{_0B_{l+r}, \Tr_{l+r}}$ as the completion of the pre-Hilbert space whose underlying vector space 
	is given by $_0B_{l+r}$ and its complex-valued inner product is given by 
    	$$
		\langle \eta\ |\ \xi \rangle_{\Ltwo{{_0B_{l+r}}}} :=  \Tr_{l+r}\left(\langle \eta\ |\ \xi \rangle_{B_0}\right).
	$$
    	For $\eta\in{_lB_r},$ we write $\eta\cdot\Omega,$ to denote its corresponding image in the complete space $\Ltwo{{_lB_r}}.$ 
	Notice that using the maps $\phi$ from Notation \ref{remark:changeX} we can turn this space into a $B_n$-$B_m$ bimodule. 
	In fact, it is easily seen that $_{B_l}(\Ltwo{{_lB_r}})_{B_r} \cong _{B_l}(p_l\rhd\Ltwo{B_\infty}\lhd p_r)_{B_r}.$
	Furthermore, for the case $B_n,$ where the module is an algebra, $\phi$ induces a $B_n$-unitary isomorphism between the
	standard GNS-construction $\Ltwo{B_n,\Tr_n}$ and our $\Ltwo{{_0B_{(n+n)}}}$.
\end{notation}

The purpose for introducing these Hilbert spaces is to \textbf{compare the relative commutants of the C*-algebras} $B_n$ \textbf{over different representations}; i.e. the relation between $M_0 = B_0'' \subset \mathcal {B}(\Ltwo{B_0})$ and $i_n[B_0]''\subset \mathcal{B} (\Ltwo{B_n}).$ (These von Neumann algebras were introduced earlier in Lemma \ref{lemma::simplicity}.) In doing so, we invoke a lemma communicated by Andr\'{e} Henriques which we simply state here:

\begin{lemma}\label{lemma::Hilbertifying}
   	Let $B$ be a unital $*$-algebra acting on the Hilbert spaces $\mathcal {H}$ and $\mathcal {K}$, and $M = B''\subset \mathcal{B}(K).$ 
	Assume there exists a dense subspace $\mathcal{H}_\circ\subset\mathcal H$ such that for every $\eta,\xi\in\mathcal{H}_\circ,$ there exists 
	$x,y\in \mathcal K$ such that $\langle b\rhd\xi\ |\ \eta \rangle_{\mathcal H} = \langle b\rhd x\ |\ y \rangle_{\mathcal K}$ holds for all $b\in B$. 
	Then the representation of $B$ on $\mathcal H$ extends to a normal representation of $M$ on $\mathcal H$.
\end{lemma}

We now check that the hypotheses of the previous lemma apply to $B_0$ acting on $\Ltwo{B_0}$ and ${B_0}$ acting on $\Ltwo{_0B_n}.$ 
For an arbitrary $n\geq 0$, any $\xi,\eta\in {_0B_n}\cdot\Omega,$ and any arbitrary $b\in B_0$ we have that
\begin{align*}
	    	\langle b\rhd\xi\ |\ \eta\rangle_{\Ltwo{{_0B_n}}} &= \Tr \left( \langle b\rhd \xi\ |\ \eta  \rangle_{B_0} \right)
	   	 =
		\Tr_n\left(
	    \begin{tikzpicture}[scale=1/15 pt, thick,baseline={([yshift=-\the\dimexpr\fontdimen22\textfont2\relax] current bounding box.center)}] 
	        \draw[rounded corners, densely dashed, blue] (-30,-15) rectangle (31,23);
	            \draw[ultra thick] (-5,25) arc (180:360:5); 
	                \node at (0,23){$\bullet$};
	            \draw[rounded corners, densely dashed, cyan] (-28,16) rectangle (-14,9);
	            \node at (-23,16){$\bullet$};
	                \draw[ultra thick] (-23,5) .. controls (-18,24) and (-7,24) .. (-5,15);
	            \draw[rounded corners, densely dashed, cyan] (29,16) rectangle (13,9);
	            \node at (24,16){$\bullet$};
	                \draw[ultra thick] (23,5) .. controls (22,24) and (7,24) .. (5,15);   
	            \draw[rounded corners, densely dashed, orange] (-9,15) rectangle (9,8);
	                \draw[ultra thick, OhioRed] (-5,15) arc (180:360:5);
	                \node[orange] at (0,15){$\bullet$};
	            \draw[rounded corners, ultra thick] (-25,-5) rectangle (-15,5);
	                \node at (-20,5){$\bullet$};
	                \node at(-20,0){$\xi^*$};
	                \draw[ultra thick] (-20,-5) -- (-20,-20);
	                    \draw[dashed, gray] (-17,5) .. controls (-15,10) and (-7,10) .. (-3,0);
	            \draw[rounded corners, ultra thick] (-5,-10) rectangle (5,0);
	                \node at(0,-5){$b^*$};
	                \node at (0,0){$\bullet$};            
	                \draw[ultra thick] (0,-10) -- (0,-20);
	                    \draw[dashed, gray] (17,5) .. controls (15,10) and (7,10) .. (3,0); 
	            \draw[rounded corners, ultra thick] (25,-5) rectangle (15,5);
	                \node at (20,5){$\bullet$};        
	                \node at(20,0){$\eta$};
	                \draw[ultra thick] (20,-5) -- (20,-20);
	    \end{tikzpicture}
	    \right)\\
	    &=
	    \sum_{k=1}^{K}\Tr_n\left(
	    \begin{tikzpicture}[scale=1/15 pt, thick,baseline={([yshift=-\the\dimexpr\fontdimen22\textfont2\relax] current bounding box.center)}] 
		\draw[ultra thick] (-5,35) arc (180:360:5); 
	                \node at (0,33){$\bullet$};
	        \draw[densely dashed, rounded corners, orange] (-27,-7) rectangle (-4,28);
		\node[orange] at(-16,28){$\bullet$};
	            \draw[OhioRed, ultra thick, rounded corners, fill=OhioRed!10](-15,10) rectangle (-5,20);
	                \node at (-10,15){$w_k^*$};
	                \draw[ultra thick, OhioRed] (-10,10) -- (-10,-20);
	            \draw[OhioRed, ultra thick] (-17,5) .. controls (-17,28) and (-12,32) .. (-7,20);
	                 \draw[rounded corners, ultra thick] (-25,-5) rectangle (-15,5);
	                        \node at(-20,0){$\xi^*$};
	                        \draw[ultra thick] (-20,-5) -- (-20,-20);	            
	                        \draw[rounded corners, ultra thick] (-4,-15) rectangle (4,-7);
	                            \node at(0,-11){$b^*$};
	                            \draw[ultra thick] (0,-15) -- (0,-20);
	        \draw[densely dashed, rounded corners, orange] (27,-7) rectangle (4,28);            
		\node[orange] at(16,28){$\bullet$};
	            \draw[OhioRed, ultra thick, rounded corners, fill=OhioRed!10](15,10) rectangle (5,20);        
	                \node at (10,15){$w_k$};        
	            \draw[ultra thick, OhioRed] (10,10) -- (10,-20);
	            \draw[OhioRed, ultra thick] (17,5) .. controls (17,28) and (12,32) .. (7,20);
%
	                    \draw[rounded corners, ultra thick] (25,-5) rectangle (15,5);
	                        \node at(20,0){$\eta$};
	                        \draw[ultra thick] (20,-5) -- (20,-20);	        
	    \end{tikzpicture}\right)
	=
	\sum_{k =1}^K\Tr_n\left(  \langle b\rhd[w_k\wedge \xi]\ |\ [v_k\wedge \eta] \rangle_{B_0} \right)
\end{align*}
Here, $\{w_k\}_{ k= 1}^K$ is a finite left $B_0$-basis for the fgp bimodule $_nB_0 \cong \X^{\boxtimes n},$ giving us that $1_{B_n} = p_n = \sum_k w_k^*\wedge w_k,$ similarly as in Equation \ref{eqn::one}.

We immediately obtain the following:
\begin{corollary}\label{corollary::representations}
	   For $n\geq 0$, we have an isomorphism of von Neumann algebras 
	   $$\mathcal B(\Ltwo{B_0})\supset M_0 \cong B_0''\subset \mathcal B \left(\Ltwo{_0B_n}\right).$$
\end{corollary}

We close this section with a proposition that will be very useful in Chapter 4:
 \begin{proposition}\label{prop:traceinnerproduct}
	 Let $n = l+r$.  On $_lB_{r}$, one has
	 $$
		 \tr(\langle \xi\ |\ \eta \rangle_{B_{0}}) = \tr( _{B_{0}}\langle \eta, \xi \rangle)
	 $$
  \end{proposition}
 \begin{proof}
	We will work on $ _nB_0$.  Note that in this case, if $\xi, \eta \in {_nB_0}$, then $\eta^{*}\xi \in B_{0}$ and $\xi\wedge\eta^{*} \in B_{n}$.   
	If we utilize the maps $\Phi$ and $\mathbb{E}$ from the discussion before Proposition \ref{prop:weights}, we have: 
	$$
		 \tr(\langle \xi\ |\ \eta \rangle_{B_{0}}) = \tr(\eta^{*}\wedge\xi) = (\Phi\circ \mathbb{E})(\eta^{*}\wedge\xi) = (\Phi\circ \mathbb{E})(\xi\wedge\eta^{*}) 
		= \delta_{x}^{n} \cdot \tr_{n}(\xi\wedge\eta^{*})) = \delta_{x}^{n} \cdot \tr(\mathbb{E}_{n}(\xi\wedge\eta^{*})) =  \tr_n( _{B_{0}}\langle \eta, \xi \rangle)
	$$
 \end{proof}

\section{Realizing a RC*TC as Hilbert C*-Bimodules}\label{Sec::Representing}
	In this chapter we will show that the constructions of Section \ref{section:bimodule} induce a full  bi-involutive (strong-)monoidal functor of the form 
	$$
		\ff:\mathcal \cat_x\hookrightarrow \Bim_{\mathsf{fgp}}(B_0).
	$$  
	 As a reminder, $B_0$ is a \textbf{separable simple tracial unital exact C*-algebra}, (see Proposition \ref{lemma::simplicity}) as in Notation  
	\ref{notation::corners}, and every $\X\in\Bim_{\mathsf{fgp}}(B_0)$ is a \textbf{minimal} (Definition \ref{minimal} and Proposition \ref{prop::Minimality}) \textbf{normalized} (Definition \ref{index} and 
	Proposition \ref{prop::normalized}) \textbf{finitely generated projective (fgp)} (Definition \ref{prop:basis} and Proposition \ref{prop:bases}) 
	\textbf{Hilbert C*-bimodule} (Definition \ref{bimodule}).
	 We direct the reader to Section \ref{sec::Bimodules} with special attention to Example \ref{HappyBims} for a more 
	detailed viewpoint on the RC*TC of bimodules $\Bim_{\mathsf{fgp}}(B_0)$. We moreover assume that $\cat$ is an \textbf{essentially small strict RC*TC with simple unit and its canonical 
	bi-involutive and spherical structures}. We moreover \textbf{fix a symmetrically self-dual object} $x  = \overline{x}\in \cat$, (Lemma \ref{lemma::ssdObject}) of dimension $\delta_x > 1$ and we denote 
	by $\cat_x$ the full RC*TC of $\cat$, whose objects are all tensor powers of $x.$

 Let us now start constructing the functor $\ff$. For a given $n\in \N$ we define (see Lemma \ref{lemma::isomorphism})
\begin{equation}
    \ff(x^{\otimes n}) :=\ _{B_0}\left(_0B_n\right)_{B_0} (\cong\ _{B_0}(\X^{\boxtimes_{B_0}n})_{B_0}).
\end{equation}

We shall now define $\ff$ on the morphism spaces. To do so, we make use of the embedding of diagrams coming from $\cat$ into the C*-algebra
 $A_\infty$, as described in Proposition $\ref{prop::AFalgs}$. (We remind the reader that the construction of $A_\infty$ was done in the paragraphs
 preceeding Proposition \ref{prop:prehilbert}.) For $n,m\in\N$ and $f\in\cat(x^{\otimes n}\rightarrow x^{\otimes m}),$ we define the action of $\ff(f)$ on 
diagrams by only using the multiplication internal to the C*-algebra $A_\infty \subseteq B_\infty\subseteq \mathcal B(\X_\infty)$. In the sequel, we will often 
simply write $f$ instead of $FR(f)$, directly identifying diagrams in the category with those in $A_\infty.$

For a diagram $\xi\in {_0B_n}\cap\Gr_{\infty}$, we define 
$$\ff(f)(\xi) := \xi\wedge f,$$
corresponding to the (linear extension of the) following diagrammatic computation:

$$\ff(f): (\Gr_{\infty}\cap\ _0B_n) \longrightarrow (\Gr_{\infty}\cap\ _0B_m)$$
\begin{equation}
    \begin{tikzpicture}[scale=1/13 pt, thick,baseline={([yshift=-\the\dimexpr\fontdimen22\textfont2\relax] current bounding box.center)}] 
        \draw[rounded corners, ultra thick] (-40,-5) rectangle (-30,5);
        \node at (-35, 5){$\bullet$};
            \node at(-35, 0){$\xi$};
            \draw[densely dashed, thick] (-38,5) -- (-38, 15);
            \draw[ultra thick] (-32,5) -- (-32, 15);
                \node at (-30, 12){$n$};
            \draw[ultra thick] (-35,-5) -- (-35,-15);
                \node at (-33,-12){$b$};
    \end{tikzpicture}
            \longmapsto
    \begin{tikzpicture}[scale=1/13 pt, thick,baseline={([yshift=-\the\dimexpr\fontdimen22\textfont2\relax] current bounding box.center)}] 
        \node at (-15,10){$\bullet$};
        \draw[densely dashed, rounded corners, blue] (-25,-11) rectangle (-5,10);
            \draw[ultra thick, rounded corners] (-20,-5) rectangle (-10,5);
            \draw[ultra thick] (-12,5) -- (-12,15);
            \node at (-15,0){$\xi$};
            \draw[densely dashed] (-18,5) -- (-18,15);
            \draw[ultra thick] (-15,-5) -- (-15,-15);
        
        \node at (0,0){$\wedge$};
        
        \node at (14,10){$\bullet$};
        \draw [densely dashed, rounded corners, blue] (5,-11) rectangle (25,10);
            \draw[ultra thick, rounded corners] (22,5) rectangle (12,-5);
            \draw[ultra thick] (18,5) -- (18,15);
            \node at (17,0){$f$};
            \draw[ultra thick] (9,-5) arc(180:360:4);
            \draw[ultra thick] (9,-5) -- (9,15);
    \end{tikzpicture}
    =
        \begin{tikzpicture}[scale=1/18 pt, thick,baseline={([yshift=-\the\dimexpr\fontdimen22\textfont2\relax] current bounding box.center)}] 
        \draw[densely dashed, rounded corners, blue] (13,-25) rectangle (40,27);
        \node at (23,27){$\bullet$};
            \draw[ultra thick] (33,20) .. controls (33,23) and (25,27).. (25,34);
            \draw[rounded corners, ultra thick] (28,10) rectangle (38,20);
                \node at (30,30){$m$};
            \node at(33,15){$f$};
                
            \draw[ultra thick] (21,-5) .. controls (22,5) and (32,6) .. (33,10);
                \node at (26,0){$n$};
            
            \draw[rounded corners, ultra thick] (15,-5) rectangle (25,-15);
            \node at (20,-10){$\xi$};
            \draw[ultra thick] (20,-15) .. controls (20,-20) and (27,-27) .. (27,-32);
            \draw[densely dashed] (17,-5) -- (17,33);
            \node at (19,-5) {$\bullet$};
                \node at (30,-30){$b$};
    \end{tikzpicture}.
\end{equation}
The next result will allow us to extend the definition of $\ff(f)$ to all of $_0B_n$.
\begin{lemma}\label{lemma::bound}
   	For every $n,m\in\N,$ we have that
    	\begin{enumerate}
        	\item the subset $\Gr_\infty \cap {_0B_n}\subset\ {_0B_n}$ is dense. 
        	\item Given $f\in \cat(x^{\otimes n}\rightarrow x^{\otimes m}),$ the map $\ff(f):\Gr_{\infty}\cap\ {_0B_n}\rightarrow \Gr_{\infty}\cap\ {_0B_m}$ is bounded and $||f||_{A_\infty} = ||f||_\cat$ (see Definition \ref{def::C*Cat}), and therefore extends to a bounded (right-adjointable) map\\ $\ff(f)\in\Bim_{\mathsf{fgp}}(B_0)({_0B_n}\rightarrow{_0B_m}),$
        	with $\ff(f)^* = \ff(f^\dagger),$ so $\ff$ is a dagger functor.
    	\end{enumerate}
\end{lemma}
\begin{proof}
Statement (1) was proven above. (See Remark \ref{remark::density} and the Proposition preceeding it.) We shall now prove (2). This proof can also be found in Proposition 4.16 in \cite{HaPeI}, but we give it here for the convenience of the reader. By definition, on diagrams we have that $\ff(f)(-) = (-)\wedge (f)$. The bound is now obvious since this map is given by multiplication in the Banach algebra $A_\infty,$ and it moreover corresponds to the operator of right creation by $f$ in the full Fock space picture. (Definition \ref{dfn::Fock} and Proposition \ref{prop::intertwiner}.)
A diagrammatic computation will reveal that $\langle \ff(f)\xi\ |\ \eta \rangle_{B_0} = \langle \xi\ |\ \ff(f^\dagger)\eta \rangle_{B_0}.$ Thus, $\ff(f)$ is right-adjointable with $\ff(f)^* = \ff(f^\dagger).$
\end{proof}

We now describe the \textbf{tensorator and unit} data for $\ff$:
$$\iota^{\ff}: \one_{\Bim_{\mathsf{fgp}}(B_0)} \longrightarrow \ff(\one_{\cat}),$$
given by the identity on $B_0$. For arbitrary $n,\ m\geq 0$ we define
$$\mu^{\ff}_{n,m} : {_0B_n}\fuse{B_0}{_0B_m}\longrightarrow {_0B_{(n+m)}},$$
by extending the following diagramatic composition (and omitting the unitors):
\begin{equation}
    	\begin{tikzpicture}[scale=1/15 pt, thick,baseline={([yshift=-\the\dimexpr\fontdimen22\textfont2\relax] current bounding box.center)}] 
    	        \draw[rounded corners, ultra thick] (22,-5) rectangle (32,-15);
	        \node at (27,-10){$\xi$};
	        \draw[ultra thick] (27,-15) --(27,-22);
	        \draw[densely dashed] (24,-5) -- (24,5);
		\draw[ultra thick] (30,-5) --(30,5);
	        \node at (27,-5) {$\bullet$};
	        \node at (30,-20){$b$};
		\node at (33,3){$n$};
	        \node at (42,-10) {$\boxtimes$};
	        \draw[rounded corners, ultra thick] (52,-5) rectangle (62,-15);
	        \node at (57,-10){$\eta$};
	        \draw[ultra thick] (57,-15) -- (57,-22);
		\draw[ultra thick](60,-5) --(60,5);
	        \draw[densely dashed] (54,-5) -- (54,5);
	        \node at (57,-5) {$\bullet$};
	        \node at (61,-20){$b'$};
		\node at (64,3){$n'$};
	\end{tikzpicture}
             	\longmapsto
		\hspace{0.5cm}
    	\begin{tikzpicture}[scale=1/15 pt, thick,baseline={([yshift=-\the\dimexpr\fontdimen22\textfont2\relax] current bounding box.center)}] 
	             \draw[densely dashed, rounded corners, blue] (-18,-8) rectangle (13,8);
	                        \draw[rounded corners, ultra thick] (-15,-5) rectangle (-5,5); 
	                        \node at(-10,-0){$\xi$};
	                        \node at (-10,8){$\bullet$};
	                        \node at (5,0){$\eta$};
	                        \draw[rounded corners, ultra thick] (0,-5) rectangle (10,5); 
	                        
	                \draw[densely dashed] (-14,5) -- (-14,22);
	                \draw[ultra thick] (-10,-5) .. controls (-10,-10) and (-4,-11) .. (-4,-18);
	                \draw[ultra thick] (5,-5) .. controls (5,-9) and (0,-14) .. (0,-18);
	                \node at (10,-16){$(b+b')$};

	                    \draw[ultra thick] (-7,5) ..controls (-7,15) and (-4,17).. (-4,22);
	                    \draw[ultra thick] (8,5) ..controls (8,12) and (-3,15).. (-2,22);
	                    \node at (9,20){($n+m)$};
	    \end{tikzpicture}.
\end{equation}
Notice that if $f$ and $g$ are morphisms in $\cat,$ then $\mu^{\ff}(f,g) = {\id_{\one}} \otimes f \otimes g$ on diagrams. So far, in defining $\mu^{\ff}_{n,m}$ we have only used the (unitary) associators and 
the unitors in $\Bim_{\mathsf{fgp}}(B_0)$, so naturality in $n$ and $m$ automatically holds, alongside with unitarity. Therefore, by Lemma \ref{lemma::bound}, $\mu^{\ff}_{n, m}$ extends to a unitary isomorphism 
${_0B_n}\fuse{B_0}{_0B_m} \xrightarrow[]{\sim} {_0B_{(n+m)}},$ which is again natural in both $n$ and $m.$ 
Similarly, we can also extend $\mu^{\ff}(f,g)$ from diagrams to all of ${_0B_n}\fuse{B_0}{_0B_m},$ since this map is just a series of vertical and horizontal compositions in $\cat$.
It is clear that bi-adjointability for $\mu^{\ff}(f,g)$ holds on diagrams, so the extension of $\mu^\ff$ is also bi-adjointable. This therefore establishes the following Proposition.
\begin{proposition}
	The functor $(\ff,\mu^{\ff},\iota^{\ff})$ is strong monoidal.
\end{proposition}

\begin{remark}\label{rmk::faithful}
	A direct diagrammatic argument demonstrates the functor $\ff$ is faithful. This, however, is a general fact about strong-monoidal dagger functors between rigid C*-categories with simple unit. 
\end{remark}

We can further unveil more structure on the functor $\ff$. 

\begin{proposition}
	The dagger tensor functor $(\ff,\mu^\ff,\iota^\ff)$ is bi-involutive with structure maps
	\begin{align*}
		\chi_n: \ff(\overline{x^{\otimes n}}) = &{\, _0B_n}\longrightarrow \overline{\ff(x^{\otimes n})} = \overline{{_0B_n}}\\	
		&\eta\longmapsto \phi^{-n}(\eta^*),
	\end{align*}
	for every $n\in\N$.
\end{proposition}
\begin{proof}
	Clearly each $\chi_n$ is a bounded
	 bi-adjointable bimodule isomorphism, since it is the composition of such maps. That $\{ \chi_n \}$ is a monoidal natural transformation for which Diagrams \eqref{eqn::Chi} 
	commute is a straightforward diagrammatic verification, and showing each $\chi_n$ is unitary follows from the properties of the map $\phi$ introduced in Notation
	 \ref{remark:changeX}  and simple diagrammatic considerations.
\end{proof}

We shall postpone the proof that this functor is full until the end of the next chapter, where we develop the tools necessary to prove it.

\section{Hilbertifying C*-Bimodules}\label{section:Hilbert}
	For the reminder of this article, we will further restrict our scope to the category $\Bim_{\mathsf{fgp}}^{\mathsf{tr}}(B_0) 
\subset \Bim_{\mathsf{fgp}}(B_0)$ \textbf{ consisting of those bimodules which are compatible with the trace} \cite[Definition 5.7]{MR1624182}; i.e., 
$\Y\in \Bim_{\mathsf{fgp}}^{\mathsf{tr}}(B_0)$ if and only if for each $\xi,\eta\in\Y$ the following identity holds:
\begin{equation}\label{eqn::ComTrace}
	\tr_{B_0}\left( \langle \eta \ |\  \xi\rangle_{B_0} \right) = \tr_{B_0}\left(\, _{B_0}\langle \xi, \eta \rangle\right).
\end{equation}
Notice this subcategory is still a RC*TC, since the fusion of bimodules compatible with the trace is again compatible.

The goal of this section is to construct a fully-faithful bi-involutive strong monoidal functor of the form:
$$_{M_0}(-\fuse{B_0}\Ltwo{B_0})_{M_0}: \Bim_{\mathsf{fgp}}^{\mathsf{tr}}(B_0)\longrightarrow \Bim_{\mathsf{bf}}^{\mathsf{sp}}(M_0)$$
\begin{align}
    	\Y\longmapsto\ _{M_0}(\Y\fuse{B_0}\Ltwo{B_0})_{M_0}\\
   \	\Bim_{\mathsf{fgp}}^{\mathsf{tr}}(B_0) (\Y\rightarrow \mathcal{Z})\owns f\mapsto f\boxtimes\id_{\Ltwo{B_0}}
\end{align}
(See Example \ref{BifiniteBimodules} for a description of $\Bim_{\mathsf{bf}}^{\mathsf{sp}}(M_0.)$)


We shall now explain how to obtain the $M_0$-bimodules. For a fixed but arbitrary $\Y\in\Bim_{\mathsf{fgp}}^{\mathsf{tr}}(B_0)$, consider the \textit{algebraic} $B_0$-balanced \textit{tensor product} $\Y\fuse{B_0}\Ltwo{B_0}$ endowed with the (densely-defined) \textbf{sesquilinear form}:
\begin{equation}
    	\langle \xi\boxtimes a\Omega,\ \eta\boxtimes b\Omega\rangle := \langle \langle \eta\ |\ \xi \rangle_{B_0}\rhd a\Omega,\ b\Omega\rangle_{\Ltwo{B_0}} := \tr_{B_0}\left(\ b^*\langle \eta\ |\ \xi\rangle_{B_0} a\ \right),  \hspace{2mm} \text{for all } a,b\in B_0, \text{and } \xi,\eta\in\Y.
\end{equation}
Notice that $\langle \xi\boxtimes a\Omega,\ \eta\boxtimes \Omega\rangle = \langle [\xi\lhd a]\boxtimes \Omega, [\eta\lhd b]\boxtimes \Omega\rangle.$ We can moreover endow this space with a right $B_0$-action given by $$(\xi\boxtimes a\Omega)\lhd b := (\xi\boxtimes ab\Omega).$$

This action is bounded. Indeed, for arbitrary $\xi\in\Y$ and $b\in B_0,$ we have 
$$||\xi\boxtimes b\Omega ||_2^2 = \tr_{B_0}( \langle \xi\lhd b\ |\ \xi\lhd \rangle_{B_0} ) = \tr_{B_0}(\, _{B_0}\langle \xi\lhd b, \xi\lhd b\rangle ) \leq ||b||^2\cdot \tr_{B_0} (\langle \xi\ |\ \xi\rangle_{B_0}) = ||b||^2\cdot ||\xi\boxtimes\Omega ||_2^2.$$
Here, we used the repeatedly used the compatibility with the trace condition, Equation \eqref{eqn::ComTrace}, together with  Lemma 1.26 in \cite{MR1624182}.

\begin{remark}
	The algebraic $B_0$-balanced tensor product space $\Y\fuse{B_0}\Ltwo{B_0},$ is a Hilbert space so it needs no further completion. Indeed, since $\Y$ is fgp, it then has a finite 
	right $B_0$-basis $\{u_i\}_{i=1}^{n}.$ By observing that $ p = \sum_{i} |u_i\rangle\langle u_i | $ is a self-adjoint idempotent, we can then realize $\Y\cong p[{B_0}] \subset {B_0}^{\oplus n}$
	as right $B_0$-modules via 
		$$y = \sum_i u_i \lhd \langle u_i\ |\ y\rangle_{B_0} \longmapsto \left(\sum_{j=1}^n \langle u_i\ |\ u_j \rangle_{B_0}\cdot\langle u_j\ |\ y\rangle_{B_0}\right)_{i=1}^n.$$ 
	We now observe that the orthogonal projection 
	$p\boxtimes\id_{\Ltwo{B_0}}\in \mathcal B^* \left({B_0}^{\oplus n}\fuse{B_0}\Ltwo{B_0}\right)$ realizes $\Y\fuse{B_0}\Ltwo{B_0}$ as a (necessarily closed) direct summand of the
	Hilbert space ${B_0}^{\oplus n}\fuse{B_0}\Ltwo{B_0} \cong \Ltwo{B_0}^{\oplus n}.$
\end{remark}

We can therefore extend this action from $\Y\boxtimes\Omega$ to the whole Hilbert space $\Y\fuse{B_0}\Ltwo{B_0}.$ To see this \textbf{right }$B_0$-\textbf{action extends to a (normal) right action of }$M_0$, we use Lemma \ref{lemma::Hilbertifying}, observing that for the dense subspace  $\Y\boxtimes\Ltwo{B_0}$ spanned by the vectors $\xi\otimes b\Omega$, the following identity $\langle\xi\boxtimes a\Omega\lhd b,\  \eta\boxtimes c\Omega \rangle = \tr(c^*\langle \eta\ |\ \xi\rangle_{B_0} ab) = \langle \left(\langle\eta\ | \ \xi\rangle_{B_0}\rhd a\Omega\right)\lhd b,\ c\Omega\rangle$ holds. 

There is a densely defined \textbf{left }$B_0$-\textbf{action} given by 
$$
	b\rhd(\xi\boxtimes c\Omega) := (b\rhd\xi)\boxtimes (c\Omega).
$$ 
In the following we will show that this action is bounded on $\Y\fuse{B_0}\Ltwo{B_0}$ and how to extend this to a normal left $M_0$-action. The distinctive elements that make this extension work are \textbf{the existence of a a positive definite trace on the \textit{right-adjointable operators}} and a \textbf{faithful, trace-preserving conditional expectation}, which we shall respectively denote by:
\begin{align*}
    	\Tr:\mathcal B^*(\Y)\longrightarrow \C \hspace{10 mm}  \mathcal E: \mathcal B^*(\Y)\longrightarrow B_0, 
\end{align*}

\begin{remark}
    	Recall that since $\Y$ is fgp, it then follows that $\mathcal B^*(\Y) = \mathcal K^*(\Y) = \text{FinRan} (\Y).$ We can therefore characterize right-adjointable operators as finite combinations of \textit{ket-bras}, $|\xi\rangle\langle\eta|,$ for $\xi,\eta\in\Y.$
\end{remark}

In lights of the above Remark, we can explicitly define these maps as the $\C$-linear extensions of
\begin{align}
    \Tr\left(\ |\xi\rangle\langle\eta|\ \right) :=  \frac{1}{\sqrt{\text{Ind}(\Y)}}\cdot \tr_{B_0}\left[\ \langle \eta\ |\ \xi\rangle_{B_0} \right]\hspace{3mm} \text{ and }\hspace{3mm}
    \mathcal E\left(\ |\xi\rangle\langle\eta|\ \right)  := \frac{1}{\sqrt{\text{Ind}(\Y)}}\cdot{_{B_0}\langle\xi,\eta \rangle}.
\end{align}

Recall that since $\Y\in\Bim_{\mathsf{fgp}}^{\mathsf{tr}}(B_0),$ it is already normalized (see Definition \ref{index} and Example \ref{HappyBims}), so that $\text{l-Ind}(\Y) =  \sqrt{\text{Ind}(\Y)} = \text{r-Ind}(\Y)\in Z(B_0)^+ = \R_{\geq0}.$

We provide a summary of the properties of $\Tr$ and $\mathcal E$ in the form of a proposition:
\begin{proposition}\label{prop::HilbTrace}
    The following statements hold on $\Bim_{\mathsf{fgp}}^{\mathsf{tr}}(B_0)$:
    \begin{enumerate}
        \item The map $\Tr$ is tracial;  
        \item The trace $\Tr$ is positive definite;
        \item The mapping $\mathcal E$ is  $B_0$-bilinear and unital;
        \item In fact, $\mathcal E$ is a fully faithful map; and
        \item The conditional expectation $\mathcal E$ preserves the trace, and therefore, for each $\xi,\eta\in \Y,$ and each $a,b\in B_0$ 
        \begin{equation}\label{eqn::MatchingStates}
            \langle \xi\boxtimes a\Omega,\ \eta\boxtimes b\Omega\rangle = \tr\left(\ _{B_0}\langle \xi\lhd a, \ \eta\lhd b \rangle\ \right) = \langle \Omega,\ _{B_0}\langle \eta\lhd b,\ \xi\lhd a \rangle\Omega \rangle.
        \end{equation}
    	\end{enumerate}
   	 Notice that by Lemma \ref{lemma::simplicity}, $B_0$ has a unique trace, so for some $\lambda > 0,$  we have that $\Tr|_{B_0} = \lambda\cdot\tr_{B_0}.$
	Since $\id_\Y =   \sum_j | v_j\rangle\langle v_j |$ for a left $B_0$-basis $\{v_j\}_{j =1}^{n}$ for $\Y$, $\lambda = \sqrt{\text{Ind}(\Y)}$.
\end{proposition}
\begin{proof} 
	\begin{enumerate}
		\item This is a direct computation.\\
		
		\item An arbitrary element $x\in\mathcal B^*(\Y) = \text{FinRan}(\Y)$ can be written as $x = \sum_{r=1}^{t} |\xi_r\rangle\langle \eta_r|.$ 
		Suppose $\{u_i\}_{i=1}^m$ is a right basis for $\Y_{B_0}.$ Assuming $x \cdot x^*$ has zero trace and by first expanding the $\eta$'s with respect to this basis, and then expanding the $\xi$'s, we obtain that
		\begin{align*}
		        0 &=  \frac{1}{\sqrt{\text{Ind}(\Y)}}\cdot\Tr(x\cdot x^*) =  \frac{1}{\sqrt{\text{Ind}(\Y)}}\cdot\sum_{r,s=1}^t\Tr\left(|\xi_r\rangle\langle \eta_r|\circ |\eta_s\rangle\langle \xi_s| \right)\\
			&= \sum_{r,s=1}^t \tr_{B_0}\left( \langle\xi_s \ |\ \xi_r\rangle_{B_0} \cdot\sum_{i,j=1}^m \langle u_i \ |\ \eta_r\rangle_{B_0}^* \cdot\langle u_i\ |\ u_j \rangle_{B_0}\cdot \langle u_j\ |\ \eta_s \rangle_{B_0} \right)\\ 
		        &= \sum_{r,s=1}^t\tr_{B_0}\left( \sum_{k=1}^m \left[\sum_{j=1}^m \langle \eta_s\ |\ u_j\rangle_{B_0} \cdot \langle u_j\ |\ u_k \rangle_{B_0}\right]^* \cdot\langle \xi_s\ |\ \xi_r\rangle_{B_0} \cdot\left[\sum_{i=1}^m \langle \eta_r\ |\ u_i\rangle_{B_0} \cdot \langle u_i\ |\ u_k \rangle_{B_0}\right] \right)\\
		        &= \sum_{r,s=1}^t\sum_{k=1}^m\tr_{B_0}\left( b_{ks}^* \cdot\langle \xi_s\ |\ \xi_r\rangle_{B_0}\cdot b_{rk}\right) = \sum_{k=1}^m \tr_{B_0}\left(\sum_{l,l'}^m c_{lk}^*\cdot \langle u_i\ |\ u_j \rangle_{B_0}\cdot  c_{l'k} \right).
		\end{align*}
		Here, $\left(\sum_{i=1}^m \langle \eta_r\ |\ u_i\rangle_{B_0} \cdot \langle u_i\ |\ u_k \rangle_{B_0}\right) =: (b_{rk})\in\text{Mat}_{T\times m}({B_0}),$
		 and $ \sum_{r=1}^t \langle u_l\ |\ \xi_r \rangle_{B_0}\cdot b_{rk} =:  c_{lk} \in\text{Mat}_{m\times m}({B_0}).$ 
		Notice that in the third equality we used the fact that the matrix $(\langle u_i\ |\ u_j \rangle_{B_0})_{ij}$ is an idempotent.
		 The last equality is then a positive definite quadratic form in $B^m,$ and therefore conclude that each of the $k$ terms in the sum is zero, by the faithfulness of $\tr_{B_0}$. We then obtain that for each $k,i\in\{1,\hdots, m\}$
		    $$0 = \sum_{j=1}^m \langle u_i\ |\ u_j\rangle_{B_0}\cdot c_{jk} =  \sum_{r=1}^t \langle u_i\ |\ \xi_r \rangle_{B_0}\cdot\langle \eta_r\ |\ u_k \rangle_{B_0}.$$
		    This implies that $\sum_{r=1}^{t} |\xi_r\rangle\langle \eta_r| = 0,$ thus establishing that $\Tr$ is positive definite.
		    
		\item The left ${B_0}$-linearity is obvious.  The right bilinearity follows from
			\begin{align*}
				\mathcal{E}( |\xi \rangle \langle \eta | \cdot b) = \mathcal{E}( |\xi \rangle \langle b^{*}\rhd\eta |) &=  \,  _{B_{0}}\langle \xi, b^{*}\rhd\eta \rangle = \frac{1}{\sqrt{\Ind(\cY)}} \, _{B_{0}}\langle b^{*}\rhd\eta,\xi \rangle^{*} \\
				&=   \frac{1}{\sqrt{\Ind(\cY)}}  (b^{*} \cdot_{B_{0}}\langle \eta,\xi \rangle)^{*} =  \frac{1}{\sqrt{\Ind(\cY)}} \,  _{B_{0}}\langle \xi, \eta \rangle\cdot b = \mathcal{E}( |\xi \rangle \langle \eta |)\cdot b,
			\end{align*}
			for all $\xi, \eta \in \Y$ and $b \in B_{0}$.  Unitality follows from the fact that $\id_\Y = \sum_{i =1}^{m}|u_i\rangle\langle u_i|,$ where $\{u_i\}_{i = 1,...m}$ is a right $B_0$-basis for $\Y_B.$\\
		\item Applying Lemma 1.26 in \cite{MR1624182} to the full bimodule $_{B_0}\Y_{B_0}$ implies $\mathcal E$ is a conditional expectation (of index finite type).
			 (The bimodule $_{B_0}\Y_{B_0}$ is full since $B_0$ is simple.) This map is clearly fully faithful, from the definition of the inner products.\\
		\item By compatibility with the trace, Equation \eqref{eqn::ComTrace}, we have that
		    	     $$ \Tr \left(\ |\xi\rangle\langle\eta|   \ \right)
			    =  \frac{1}{\sqrt{\text{Ind}(\Y)}}\cdot\tr_{B_0} \left(\ \langle\eta\ |\ \xi\rangle_{B_0} \ \right)
			    =  \frac{1}{\sqrt{\text{Ind}(\Y)}}\cdot\tr_{B_0} \left(\ _{B_0}\langle \xi,\ \eta\rangle \ \right) 
			    = \tr_{B_0} \left(\ \mathcal E (|\xi\rangle\langle \eta|) \ \right).$$
	   \end{enumerate}
\end{proof}
 \begin{remark}
	Notice that all the bimodules $_lB_r$ introduced in Notation \ref{notation::copies} define objects in $\Bim_{\mathsf{fgp}}^{\mathsf{tr}}(B_0)$, in light of Proposition \ref{prop:traceinnerproduct}.
\end{remark}
 
\begin{proposition}
    	The left $B_0$-action on $\Y\boxtimes\Omega$ extends to a bounded left $B_0$-action on $\Y\fuse{B_0}\Ltwo{B_0}$.
\end{proposition}
\begin{proof}
    	This follows from the fact that $\Y$ is fgp. Indeed, for arbitrary $b\in B_0$ and $y\in\Y$, we have that
	 $||b\rhd(y\boxtimes\Omega)||_2^2 = \tr (\langle b\rhd y\ |\ b\rhd y\rangle_{B_0})\leq \tr (||b||^2\cdot\langle y\ |\ y\rangle_{B_0}) = ||b||^2\cdot ||y\boxtimes\Omega||_2^2.$	
	The inequality is a direct application of (3) in \cite[Proposition 1.12]{MR1624182}, stating that $\langle b\rhd y\ |\ b\rhd y \rangle_{B_0}\leq ||b||^2\cdot \langle y\ |\ y\rangle_{B_0}.$ 
\end{proof}

We now extend the left $B_0$-action on $\Y\fuse{B_0} \Ltwo{B_0}$ to a normal left-$M_0$ action by using Lemma \ref{lemma::Hilbertifying}. 
 For the dense subspace $\Y\boxtimes\Omega$ spanned by the vectors $\xi\boxtimes a\Omega,$ since the conditional expectation preserves the trace, we obtain
$$\langle b\rhd(\xi\boxtimes a\Omega),\ \eta\boxtimes c\Omega  \rangle = \tr_{B_0}\left( _{B_0}\langle b\rhd\xi\lhd a,\ \eta\lhd c  \right)  
= \tr_{B_0} \left(b\cdot {_{B_0}\langle} \xi\lhd a,\ \eta\lhd c\rangle\right) = \langle b\Omega,\ _{B_0}\langle \eta\lhd c,\ \xi\lhd a\rangle\cdot\Omega\rangle$$
Hence \textbf{the left }$B_0$-\textbf{action on }$\Y\fuse{B_0}\Ltwo{B_0}$ \textbf{extends to a normal left} $M_0$-\text{action}.

\begin{remark}\label{rmk::BoundVects}
    	For the remainder of this chapter will need to further exploit the structure of $M_0$-bimodules of the form $\mathcal Z \fuse{B_0}\Ltwo{B_0},$ 
	and their tensor products. We say that \textbf{ a vector }$\xi\in \mathcal Z \fuse{B_0}\Ltwo{B_0}$ \textbf{is (left) bounded} if and only if the mappping 
	$$
		L_\xi: M_0\Omega\rightarrow \mathcal Z \fuse{B_0}\Ltwo{B_0}, \text{ given by }
			m\Omega\mapsto m\blacktriangleright\xi
	$$
	 extends to a bounded map defined on  $\Ltwo{B_0}\cong \Ltwo{M_0}.$ If $\xi,\eta$ are bounded vectors, then the composite map  $L_\eta^*\circ L_\xi$ makes
	 sense and takes values  in ${M_0}\Omega.$ This defines a \textbf{left }$M_0$-\textbf{valued inner product} on
	 the set bounded vectors, denoted $\langle\eta\ |\ \xi\rangle_{M_0}.$  (For a proof see \cite{Bisch} or \cite{JP17}.)
	 One can easily see that for an arbitrary $z\in \mathcal Z$, the vector $z\boxtimes \Omega\in \mathcal Z \fuse{B_0}\Ltwo{B_0}$ is bounded,
	 and hence $\langle z\ |\ z\rangle_{M_0} \in M_0$. A computation will reveal that $L_{z\boxtimes \Omega}^*: \mathcal Z \fuse{B_0}\Ltwo{B_0}\rightarrow \Ltwo{M_0}$ 
	is given by $y\boxtimes\Omega\mapsto\ _{B_0}\langle y, z\rangle\Omega .$ 
	 It then follows that $\langle z\boxtimes\Omega\ |\ z\boxtimes\Omega \rangle_{M_0} =\ _{B_0}\langle z, z\rangle$.
\end{remark}

We now show that $_{M_0}(-\fuse{B_0}\Ltwo{B_0})_{M_0}$ \textbf{defines a bi-involutive fully-faithful strong monoidal functor} into the claimed target category.
 Again, we need to be certain that the von Neumann closure of $B_0$ is independent of the space over which we represent it.

\begin{lemma}
	The $\sigma$-weak topologies on the left (resp. right) action of $B_0''$ coming from the left (resp. right) action on $\Ltwo{B_0}$ and the action on $\Y \fuse{B_0} \Ltwo{B_0}$ agree. 
\end{lemma}
\begin{proof}
    	Let's first show the statement corresponding to the left actions. Let $(B_0\rhd)$ denote the image of $B_0$ as a subset of $\mathcal B(\Ltwo{B_0}),$ and let $(B_0\blacktriangleright)$
	 that of $B_0$ considered as a subset of $\mathcal B(\Y \fuse{B_0} \Ltwo{B_0})$. By Lemma \ref{lemma::Hilbertifying}, we can consider the von Neumann algebra
	 $(M_0\blacktriangleright ) \subset \mathcal B(\Y \fuse{B_0} \Ltwo{B_0}).$ Moreover, this lemma tells us that for every $b\in B_0$ we have that $(b\blacktriangleright) = ((b\rhd)\blacktriangleright).$ 
	So we automatically obtain that $(B\blacktriangleright)''\subseteq (M\blacktriangleright),$ since both are von Neumann algebras containing the same copy of $B_0.$ Now, since $\blacktriangleright$ 
	defines a $\sigma$-weak continuous map $\blacktriangleright: M\rightarrow\mathcal B(\Y \tens{B_0} \Ltwo{B_0}),$ it then follows that $\blacktriangleright$ commutes with taking $\sigma$-weak-limits. 
	Thus, for an arbitrary $m\in M$ and a net $b_\lambda$ such that $b_\lambda\rhd\longrightarrow m$ $\sigma$-weak, we have the following chain of equalities:
    	$$
		\text{WOT}-\lim_{\lambda} b_\lambda\blacktriangleright = \text{WOT}-\lim_\lambda ((b_\lambda \rhd)\blacktriangleright) = (\text{WOT}-\lim_\lambda b_\lambda\rhd)\blacktriangleright =  m \blacktriangleright,
	$$
    	thus establishing that $(B_0\blacktriangleright)'' = (M\blacktriangleright)\subset \mathcal B(\Y \tens{B_0} \Ltwo{B_0}).$
    	The proof of the corresponding statement for the right actions is straightforward and thus we omit it.
\end{proof}

\begin{proposition}
    	The assignment $_{M_0}(-\fuse{B_0}\Ltwo{B_0})_{M_0}$ defines a faithful dagger functor out of $\Bim_{\mathsf{fgp}}^{\mathsf{tr}}(B_0)$.
\end{proposition}
\begin{proof}
    	Let $\Y, \mathcal Z \in \Bim_{\mathsf{fgp}}^{\mathsf{tr}}(B_0)$ and $f: \Y\rightarrow \mathcal Z$ be a fixed but arbitrary $B_0$-bimodule map. We denote $\tilde f:= f\boxtimes\id_{\Ltwo{B_0}}$. 
	It is clear that $\tilde f$ is $B_0$-$M_0$ bilinear and bounded. It is also evident that $_{M_0}(-\fuse{B_0}\Ltwo{B_0})_{M_0}$ respects composition of maps and the dagger structures, 
	so it only remains to check that $\tilde f$ is in fact $M_0$-bilinear. 
	Let $m$ be an arbitrary element in $M_0$, and $(b_n)\subset B$ a sequence such that $\text{SOT}-\lim_{n} (b_n\blacktriangleright) = m\blacktriangleright.$ we then have that
	\begin{align*}
	        m\blacktriangleright\tilde f(y\boxtimes\Omega) &= {\|\cdot\|_2}-\lim_{n} b_{n}\blacktriangleright \tilde f(y\boxtimes\Omega) 
	                                                        = {\|\cdot\|_2}-\lim_{n} \tilde f ((b_n\rhd y)\boxtimes \Omega)\\
	                                                        &= \tilde{f} \left( {\|\cdot\|_2}-\lim_{n}((b_{n}\rhd y)\otimes \Omega) \right) 
	                                                        = \tilde{f} \left( {||\cdot||_2}-\lim_{n} b_n\blacktriangleright(y\otimes\Omega)\right)\\
	                                                        &= \tilde{f} (m\blacktriangleright (y\boxtimes\Omega)). 
     \end{align*}
	Faithfulness follows since the trace $\tr_{B_0}$ is faithful. Notice that we are arguing using sequences converging in the SOT, as opposed to nets. 
	Indeed, on $||\cdot||_\infty$-bounded sets, the $||\cdot||_2$-norm topology on $B_0\Omega$ is equivalent to the SOT on $B_0$. Thus, by Kaplansky Density Theorem,
	we can replace SOT converging nets by $||\cdot||_2$-converging sequences which are automatically bounded in $||\cdot||_\infty.$
\end{proof}

\begin{proposition}\label{prop::propertiesHilbertify}
    	The functor $_{M_0}(-\fuse{B_0}\Ltwo{B_0})_{M_0}$ is strong monoidal with tensorator given by the linear extension of:
    	$$
		\mu_{\Y,\mathcal Z}: (\Y\fuse{B_0}\Ltwo{B_0})\fuse{M_0} (\mathcal Z\fuse{B_0}\Ltwo{B_0}) \longrightarrow ((\Y\fuse{B_0}\mathcal Z)\fuse{B_0}\Ltwo{B_0})
	$$
    	mapping	$(\xi\boxtimes\Omega) \boxtimes (\eta\boxtimes\Omega) \mapsto ( \xi\boxtimes\eta) \boxtimes\Omega.$
\end{proposition}
\begin{proof}
	Let $y\in\Y$ and $x\in\mathcal Z$ be fixed but arbitrary. We shall first show that this map is well-defined.
	For an arbitrary operator $m\blacktriangleright\in M_0$, consider a bounded net $(b_i)\subset B_0$ such that $b_i\blacktriangleright\rightarrow m\blacktriangleright$ SOT, as $i\rightarrow\infty$. 
	We then obtain that 
	$$
		|| ((y\lhd m)\boxtimes z))\boxtimes\Omega - y\boxtimes(b_i\rhd z)\boxtimes\Omega||_2^2 = \lim_i|| (y\boxtimes(b_i-m)\rhd z)\boxtimes\Omega||_2^2\leq ||y||_B^2\cdot ||(b_i-m)\rhd z\boxtimes\Omega||_2^2 = 0.
	$$
    	We shall next show that $\mu$ is an isometry. By Remark \ref{rmk::BoundVects} and Equation (1.1) in \cite{Bisch} defining the product 
	on the $M_0$-relative fusion of two bimodules in $\Bim_{\mathsf{bf}}^{\mathsf{sp}}(M_0)$, we have the following chain of equalities:
    	\begin{align*}
      		\langle (y\boxtimes\Omega)\boxtimes(z\boxtimes\Omega),\ (y\boxtimes\Omega)\boxtimes(z\boxtimes\Omega) \rangle  &=  \langle (y\boxtimes\Omega)\blacktriangleleft \langle z\boxtimes\Omega\ |\ z\boxtimes\Omega\rangle_{M_0},\ y\boxtimes\Omega\rangle\\ 
              &= \tr_{B_0} \left(  \langle y\ |\ y\lhd \,_{B_0}\langle z, z\rangle\rangle_{B_0} \right)
                          = \tr_{B_0} (\langle y\ |\ y \rangle_{B_0} \cdot\, _{B_0}\langle z,  z\rangle)\\
              &= \tr_{B_0} (\, _{B_0}\langle \langle y\ |\ y\rangle_{B_0}\rhd z\ |\ z\rangle ) = \tr_{B_0} ( \langle z\ |\ \langle y\ |\ y\rangle_{B_0}\rhd z\rangle_{B_0} )\\ 
              &= \tr_{B_0} (\langle y\boxtimes z\ |\ y\boxtimes z\rangle_{B_0} ) = \langle ( y\boxtimes z)\boxtimes\Omega, (y\boxtimes z)\boxtimes\Omega \rangle.
    	\end{align*}
    	Thus, this map defines an isometry. Since $\mu_{\Y,\mathcal Z}$ obviously has dense range, then it extends to a unitary $M_0$-bimodule map. 
	As a word of warning, to use the inner product on the $M_0$-fusion bimodule, one needs to know \textit{a priori} that $\Y\fuse{B_0}\Ltwo{B_0}$ and $\mathcal Z\fuse{B_0}\Ltwo{B_0}$
	define spherical/extremal bifinite $M_0$-bimodules; i.e. they are objects in $\Bim_{\mathsf{bf}}^{\mathsf{sp}}(M_0).$ This is independently demonstrated as part of the proof of Proposition \ref{prop::CommTrian}, but we
	postpone it not to slow down the exposition.

    	It remains to verify that $\mu$ defines a natural transformation. 
	A simple computation will show that naturality holds when restricted to the dense subspace $(\Y\fuse{B_0}\Omega)\fuse{B_0}(\mathcal Z\fuse{B_0}\Omega)$, 
	and this will suffice, since all the transformations involved are bounded. This completes the proof.
\end{proof}

\begin{proposition}\label{prop::HilbBiInv}
	The functor $-\fuse{B_0}\Ltwo{B_0}$ is bi-involutive with structure maps
	\begin{align*}
		\chi_\Y:  \overline{\Y}\fuse{B_0}\Ltwo{B_0}&\longrightarrow \overline{\Y\fuse{B_0}\Ltwo{B_0}} \text{ given on } \overline{\Y}\boxtimes\Omega  \text{ by}\\
			&\overline{y}\boxtimes\Omega\longmapsto \overline{y\boxtimes\Omega},
	\end{align*}
	for $\Y\in\Bim_{\mathsf{fgp}}^{\mathsf{tr}}(B_0).$
\end{proposition}
\begin{proof}
	We need to show this is a family of unitary isomorphisms satisfying the conditions in Definition \eqref{dfn::BiInvolutive}. 
	By compatibility with the trace condition (Equation \ref{eqn::ComTrace}) we have the following chain of equalities:
	\begin{align*}
	        \langle\chi_{\Y}(\overline{y}\boxtimes\Omega)\ |\ \overline{x\boxtimes\Omega} \rangle &= \langle \overline{y\boxtimes\Omega}\ |\ \overline{x\boxtimes\Omega}\rangle = \langle x\boxtimes\Omega\ |\ y\boxtimes\Omega\rangle\\
	        &= \tr_{B_0}\left(\langle x\ |\ y\rangle_{B_0}\right) = \tr_{B_0}\left( _{B_0}\langle y,\ x \rangle\right)     = \tr_{B_0}\left( \langle \overline y \ |\ \overline x \rangle_{B_0}\right)\\
	        &= \langle \overline y \boxtimes\Omega \ |\ \overline x\boxtimes \Omega \rangle = \langle \overline{y}\boxtimes\Omega \ |\ \chi_{\Y}^{-1}(\overline{x\boxtimes\Omega})\rangle,
	\end{align*}
	proving that $\chi_\Y$ is right-adjointable and that it extends to a unitary. 
	To show that that this family of unitaries $\{\chi_\Y\}$ is natural, monoidal and that Diagrams \eqref{eqn::Chi} commute is a matter of simple diagrammatic computation. This completes the proof.
\end{proof}

\begin{proposition}\label{prop::CommTrian}
    	The following 2-cell commutes:
       \begin{equation}\label{eqn::triangle}
           	\begin{tikzcd}[column sep=large, row sep=large]
        	& \cat_x \arrow[dl, "\ff"', hook] \arrow[dr,hook, two heads, "\mathbb G"{name=G}]  & \\
        	|[alias=B]|\Bim_{\mathsf{fgp}}^{\mathsf{tr}}(B_0) \arrow["-\fuse{B_0}\Ltwo{B_0}"', hook]{rr}\arrow[shorten >=7mm, shorten <=7mm, Rightarrow, from = B, to = G, "T", "\cong" below]  & & \Bim_{\mathsf{bf}}^{\mathsf{sp}}(M_0). 
           	\end{tikzcd}
      \end{equation} 
     	Here, $T = \{T_n\}_{n\geq 0}$ is a monoidal unitary natural isomorphism whose components are
	\begin{align*}
		T_n: ({_0B_n})&\fuse{B_0}\Ltwo{B_0}\longrightarrow  \mathbb G(x^{\otimes n}) = p_0 \rhd\Ltwo{{_0M_n}} \lhd p_n \text{ given on } {_0B_n}\boxtimes\Omega \text{ by } \\
			&(p_0\wedge b\wedge p_n)\boxtimes\Omega \longmapsto p_0\rhd b\Omega \lhd p_n =(p_0\wedge b\wedge p_n)\Omega,
	\end{align*}
 	and $\mathbb G$ is the fully faithful bi-involutive strong-monoidal functor constructed on \cite{BHP12}. 
	In in our language, $\mathbb G$ is given on objects by
    \begin{equation*}
        x^{\otimes n}\longmapsto {_{M_0}\left(p_0\rhd \Ltwo{_0M_n} \lhd p_n \right)} _{M_0}
    \end{equation*}
    and on morphisms
	    \begin{equation}
	        \begin{tikzpicture}[scale=1/10 pt, thick,baseline={([yshift=-\the\dimexpr\fontdimen22\textfont2\relax] current bounding box.center)}] 	        
		            \node at(-21,4){$(f: x^{\otimes n}\longrightarrow x^{\otimes m}) \longmapsto \ \ \mathbb G(f):$};
		            
		            \draw[rounded corners, ultra thick] (5,-1) rectangle (15,9);
		            \node at (10,9){$\bullet$};
		                \node at (10,4){$\xi$};
		                \draw[ultra thick] (14,9) --(14,16);
		                    \node at(16,14){$n$};
		                \draw[densely dashed] (6,9) --(6,16);
		                \draw[ultra thick] (10,-1) -- (10,-8);
		                    \node at(12,-6){$b$};
		                
		            \node at (20,4){$\longmapsto$};
		            
		            \draw[rounded corners, ultra thick] (26,-4) rectangle (34,4);
		            \node at (30,4){$\bullet$};
		                \node at (30,0){$\xi$};
		                \draw[ultra thick] (32,4) --(32,8);
		                    \node at(34,6){$n$};
		                \draw[densely dashed] (28,4) --(28,19);
		                \draw[ultra thick] (30,-4) -- (30,-10);
		                    \node at(31.5,-8){$b$};
		                		            
		            \draw[rounded corners, ultra thick] (29.5,8) rectangle (34.5,14);
		            \node at (32,11){$f$};
		                \draw[ultra thick] (32,14) -- (32,19);
		                    \node at(34.5,17.5){$m$};
	        \end{tikzpicture}.
    	\end{equation}
	It then follows that $\ff$ is full; i.e., $$\ff[\cat(x^{\otimes n}\rightarrow x^{\otimes m})] = \Bim_{\mathsf{fgp}}^{\mathsf{tr}}(B_0) (\X^{\fuse{B_0} n}\rightarrow \X^{\fuse{B_0} m}).$$
\end{proposition}
\begin{proof}
	Each of the maps $T_n$ is densely defined and is clearly $B_0$-bilinear and isometric with dense range. Thus, each $T_n$ extends to a $B_0$-unitary. 
	 For an arbitrary map 
	$f\in\cat(x^{\otimes n}\rightarrow x^{\otimes m}),$ checking that $(T_m \circ (-\fuse{B_0}\Ltwo{B_0}) \circ \ff)(f) = (\mathbb G \circ T_n)(f)$ is an easy computation. 
	The tensorator $\mu^T$ is given by extending juxtaposition of diagrams, similarly as was previously done for $\mu^\ff.$
	Therefore, establishing Diagram \ref{eqn::triangle} commutes, as well as proving that $\ff$ does take values in the bifinite spherical bimodules over $M_0$ up to a monoidal
	 unitary natural isomorphism.
    
    	To show that $\ff$ is full we implement a finite-dimensional linear algebraic trick: Since $\mathbb G$ is full and faithful between rigid categories, we have that 
	$$
		\text{dim}[\cat(x^{\otimes n}\rightarrow x^{\otimes m})] = \text{dim}[\Bim_{\mathsf{bf}}^{\mathsf{sp}}(M_0) (\mathbb G(x^{\otimes n})\rightarrow \mathbb G(x^{\otimes m}))] =: D <\infty.
	$$
	Since $\ff$ and $(-\fuse{B_0}\Ltwo{B_0})$ are faithful, we obtain 
	$$
 		D = \text{dim}[\ff[\cat(x^{\otimes n}\rightarrow x^{\otimes m})])] \leq \text{dim} [\Bim_{\mathsf{fgp}}^{\mathsf{tr}}({_0B_n}\rightarrow{_0B_m})]
		= \text{dim} [\Bim_{\mathsf{fgp}}^{\mathsf{tr}}({_0B_n}\rightarrow{_0B_m})\fuse{B_0}\Ltwo{B_0}] \leq D.
	$$
	Thus all inequalities must be simultaneously satified, proving $\ff$ is full. This completes the proof.
\end{proof}
 
Now we \textbf{unitarily Cauchy-complete} $\ff$ on $\cat_x$: \textbf{on subobjects} of tensor powers of $x$:  for a given $n\in \N$, a projection
 $p\in \cat(x^{\otimes n}\rightarrow x^{\otimes n})$ determines a sub-object of $x^{\otimes n},$ denoted $p[x^{\otimes n}]\in\cat_x,$ we then define
$$\ff(p[x^{\otimes n}]) := \ff(p)[{ _0B_n}] \in \Bim_{\mathsf{fgp}}^{\mathsf{tr}}(B_0),$$
and \textbf{linearly extend} the definition of $\ff$ \textbf{over direct sums}. Here, $\ff(p)[\ _0B_n] \in \Bim_{\mathsf{fgp}}^{\mathsf{tr}}(B_0)$ denotes the range of $\ff(p).$
We shall now define $\ff(f),$ where $f:p[x^{\otimes n}]\rightarrow q[x^{\otimes m}]$ is a morphism between projections $p\in \cat(x^{\otimes n})$ and $q\in\cat(x^{\otimes m}).$ We define
$$\ff(f): \ff(p)[{_0B_n}] \longrightarrow \ff(q)[{_0B_m}]$$
as the extension of the mapping acting on a diagram $\xi\in \Gr_{\infty}\cap\ {_0B_n}$ by:
\begin{equation}
    \begin{tikzpicture}[scale=1/13 pt, thick,baseline={([yshift=-\the\dimexpr\fontdimen22\textfont2\relax] current bounding box.center)}]
        \draw[rounded corners, densely dashed, blue](-42,-9) rectangle (-23,24);
        \node at (-35,24){$\bullet$};
            \draw[rounded corners, ultra thick] (-40,-5) rectangle (-30,5);
            \node at (-35, 5){$\bullet$};
                \node at(-35, 0){$\xi$};
                \draw[densely dashed, thick] (-38,5) -- (-38, 28);
                \draw[ultra thick] (-35,-5) -- (-35,-15);
                    \node at (-33,-12){$b$};
            \draw[ultra thick] (-33,5) ..controls (-33, 8) and (-30,9) .. (-30,12);
                    
            \draw[rounded corners, ultra thick](-35,12) rectangle(-25,18);
            \node at(-30,15){$p$};
                \draw[ultra thick] (-30,18) -- (-30,28);
                    \node at(-27,26){$n$};
    \end{tikzpicture}
            \longmapsto
    \begin{tikzpicture}[scale=1/16 pt, thick,baseline={([yshift=-\the\dimexpr\fontdimen22\textfont2\relax] current bounding box.center)}] 
        \draw[rounded corners, densely dashed, blue](-42,-9) rectangle (-23,48);
        \node at (-35,48){$\bullet$};
            \draw[rounded corners, ultra thick] (-40,-5) rectangle (-30,5);
            \node at (-35, 5){$\bullet$};
                \node at(-35, 0){$\xi$};
                \draw[densely dashed, thick] (-38,5) -- (-38, 54);
                \draw[ultra thick] (-35,-5) -- (-35,-15);
                    \node at (-33,-12){$b$};
            \draw[ultra thick] (-33,5) ..controls (-33, 8) and (-30,9) .. (-30,12);
                    
            \draw[rounded corners, ultra thick](-35,12) rectangle(-25,18);
            \node at(-30,15){$p$};
                \draw[ultra thick] (-30,18) -- (-30,24);
                    
            \draw[rounded corners, ultra thick](-34,24) rectangle(-26,32);
            \node at(-30,28){$f$};
                \draw[ultra thick] (-30,32) -- (-30,38);    
                    
            \draw[rounded corners, ultra thick](-35,38) rectangle (-25,44);
            \node at(-30,41){$q$};
                \draw[ultra thick] (-30,44) -- (-30,54);
                    \node at (-26,51){$m$};
    \end{tikzpicture}.
\end{equation}

We now extend the \textbf{tensorator data} for $\ff$. For arbitrary $n,\ m\geq 0$, and projections $p\in \cat(x^{\otimes n} \rightarrow x^{\otimes n})$ and $q\in\cat(x^{\otimes m}\rightarrow x^{\otimes m})$ we define
$$\mu^{\ff}_{p,q} :\ff(p)[\ {_0B_n}]\boxtimes_{B_0}\ff(q)[\ {_0B_m}]\longrightarrow \ff(p\otimes  q)[\ {_0B_{n+m}}],$$
by extending the following diagramatic composition (and omitting the unitors):
\begin{equation}
    	\begin{tikzpicture}[scale=1/13 pt, thick,baseline={([yshift=-\the\dimexpr\fontdimen22\textfont2\relax] current bounding box.center)}] 
    	\draw[densely dashed, rounded corners, blue] (13,-27) rectangle (40,24);
	        \node at (23,24){$\bullet$};
	        \draw[ultra thick] (33,18) .. controls (33,18) and (27,27).. (27,30);
	        \draw[rounded corners, ultra thick] (28,12) rectangle (38,18);
	        \node at(33,15){$p$};
	        \node at (32,27){$n$};
	        
	        \draw[ultra thick] (22,-5) .. controls (22,0) and (33,6) .. (33,12);
	        \node at (27,0){$n$};
	        
	        \draw[rounded corners, ultra thick] (15,-5) rectangle (25,-15);
	        \node at (20,-10){$\xi$};
	        \draw[ultra thick] (20,-15) .. controls (20,-18) and (26,-20) .. (27,-31);
	        \draw[densely dashed] (17,-5) -- (17,30);
	        \node at (29,-30){$b$};
	        \node at (45,0) {$\boxtimes$};
	        \draw[densely dashed, rounded corners, blue] (50,-27) rectangle (77,24);
	        \node at (60,24){$\bullet$};
	        \draw[ultra thick] (70,18) .. controls (70,23) and (64,27).. (64,30);
	        \draw[rounded corners, ultra thick] (65,12) rectangle (75,18);
	            \node at(70,15){$q$};
	            \node at (69,27){$m$};
	        
	        \draw[ultra thick] (58,-5) .. controls (58,0) and (70,5) .. (70,12);
	        \node at (65,0){$m$};
	        
	        \draw[rounded corners, ultra thick] (52,-5) rectangle (62,-15);
	        \node at (57,-10){$\eta$};
	        \draw[ultra thick] (57,-15) .. controls (58,-24) and (63,-26) .. (64,-32);
	        \draw[densely dashed] (54,-5) -- (54,30);
	        \node at (67,-30){$b'$};
		\end{tikzpicture}
             	\longmapsto
    		\begin{tikzpicture}[scale=1/13 pt, thick,baseline={([yshift=-\the\dimexpr\fontdimen22\textfont2\relax] current bounding box.center)}] 
		            \draw[densely dashed, rounded corners, blue] (-24,-27) rectangle (20,24);
		            \node at (-11,24){$\bullet$};
		                \draw[densely dashed, rounded corners, blue] (-22,-22) rectangle (7,-8);
		                \node at (-17,-8){$\bullet$};
		                        \draw[rounded corners, ultra thick] (-20,-20) rectangle (-10,-10); 
		                        \node at(-15,-15){$\xi$};
		                        \node at (0,-15){$\eta$};
		                        \draw[rounded corners, ultra thick] (-5,-20) rectangle (5,-10); 
		                        
		                \draw[densely dashed] (-19,-10) -- (-19,30);
		                \draw[ultra thick] (-15,-20) .. controls (-15,-25) and (0,-26) .. (0,-33);
		                \draw[ultra thick] (0,-20) .. controls (0,-24) and (3,-29) .. (3,-33);
		                \node at (12,-31){$(b+b')$};
		                \draw[ultra thick] (-15,-10) ..controls (-15,-4) and (-5,-3).. (-5,2);
		                \draw[ultra thick] (0,-10) ..controls (0,-4) and (11,-3).. (11,2);
		                
		                \draw[densely dashed, rounded corners, blue] (-12,-2) rectangle (18,12);
		                    \draw [rounded corners, ultra thick] (-10,2) rectangle (0,8);
		                      \node at (-5,5){$p$};
		                    \draw[rounded corners, ultra thick] (6,2) rectangle (16,8);
		                        \node at (11,5){$q$};
		                    \draw[ultra thick] (-5,8) ..controls (-5,17) and (-2,19).. (-2,30);
		                    \draw[ultra thick] (11,8) ..controls (11,17) and (2,19).. (2,30);
		                    \node at (12,28){($n+m)$};
	    \end{tikzpicture}
\end{equation}
Notice that if $p'$ and $q'$ are projections in the Cauchy completion of $\cat_x,$ and if $f:p\rightarrow p'$ and $g: q\rightarrow q'$ are morphisms in $\cat_x,$ then 
$\mu^{\ff}(f,g) = {\id_{\one}} \otimes f \otimes g$ on diagrams. By Lemma \ref{lemma::bound}, $\mu^{\ff}_{p, q}$
 extends to a unitary isomorphism $\ff(p)[{_0B_n}]\boxtimes_{B_0}\ff(q)[{_0B_m}] \xrightarrow[]{\sim} \ff(p\otimes  q)[{_0B_{(n+m)}}],$ and is clearly natural in $p$ and $q.$ 
 We can also extend $\mu^{\ff}(f,g)$ from diagrams to all of $\ff(p)[{_0B_n}]\boxtimes_{B_0} \ff(q)[{_0B_m}],$ since this map is just a series of vertical and horizontal compositions in the category.
It is clear that bi-adjointability for $\mu^{\ff}(f,g)$ holds on diagrams, so the extension of $\mu^\ff$ is also bi-adjointable.

We are now ready to state the final form of the most important result of this manuscript, which summarizes this whole article: 

\begin{theorem}\label{thm::punchline}
    	Given a singly-generated small RC*TC, $\cat_x$ with simple unit with generator $x\in\cat_x$,  there exists a unital, simple, separable, exact C*-algebra $B_0$ with unique trace, 
	and a fully-faithful bi-involutive strong monoidal functor 
	$$
		\ff:\cat_x\hookrightarrow \Bim_{\mathsf{fgp}}^{\mathsf{tr}}(B_0),
	$$ into a full subcategory of fgp minimal and normalized $B_0$-bimodules.  Moreover, the $K_{0}$ group of $B_0$ is the free abelian group on the classes
	 of simple objects in $\cat_x$, and the image of the simple objects of $\cat_x$ are precisely the canonical generators of $K_{0}(B_0).$
\end{theorem}

\begin{remark}\label{remark:countable}
	One can extend the result of Theorem \ref{thm::punchline} by allowing $\cat$ to be countably generated. 
	 Indeed if $\cat$ is countably generated, then it is generated by a countable set of symmetric self-dual objects $x_{1}, \, x_{2}, \, \cdots$.  
	Given a word $w$ a word in $x_{1}, x_{2}, \cdots$ we define $x_{w}$ by
	$$
		x_{w} = x_{i_{1}}\otimes \cdots \otimes x_{i_{k}} \text{ provided } w = x_{i_{1}}\cdots x_{i_{j}}.
	$$
	Given words $u, v, $ and $w$, one now defines $V_{u, v, w}$ by $\cat(x_{u} \rightarrow x_{w} \otimes x_{w})$.  If we define $W$ to be the set of all words in $x_{1}, \, x_{2}, \, \cdots$, one can form
	$$
		\Gr_{\infty} = \bigoplus_{u, v, w \in W} V_{u, v, w},
	$$ 
	and endow it with the $*$-structure and multiplication as described in Sections \ref{section:GJS} and \ref{section:weight}, where we now color the strings according to the labels $x_{1}, x_{2}, \cdots$. 
 	Let $B_{\infty}$ be the C*-algebra generated by $\Gr_{\infty}$, and let $B_{0} = p_{0}\wedge B_{\infty}\wedge p_{0},$ where $p_{0}$ is the empty diagram.  It is also true
	that $B_{0}$ is a unital, separable, simple, exact C*-algebra.
 
 	If $w \in W$, we define
 	$$
 		X_{w} = \overline{ \bigoplus_{v \in W} V_{v, w, \emptyset}}^{\|\cdot\|}
 	$$  
 	then $X_{w} \in \Bim_{\mathsf{fgp}}^{\mathsf{tr}}(B_{0})$ via the actions and inner products as described in Section \ref{section:bimodule}.  
	Moreover, the mapping $x_w \longmapsto  X_{w}$ induces a fully-faithful bi-involutive strong monoidal functor 
	$$
		\ff:\cat\hookrightarrow \Bim_{\mathsf{fgp}}^{\mathsf{tr}}(B_0).
	$$
\end{remark}

We note that there is no separable C*-algebra, $B$ with unique trace, over which we can fully realize every RC*TC as a full subcategory of Hilbert C*-bimodules.  
This stems from the fact that the $K_{0}$ group of any such $B$ must be countable, and thus there are only countably many values for the traces of projections in 
$B \otimes K(\mathcal{H})$. On the other hand, amongst all of the countably generated fusion categories, any $\delta \geq 2$ can be the dimension of a generating object.
However, our main theorem does allow us to establish this result, if we restrict our scope to \textbf{unitary fusion categories}:

\begin{corollary}\label{cor::universalFusion}
	There exists a unital simple exact separable C*-algebra $B$ with unique trace over which we can represent every unitary
	 fusion category $\cat$ in the spirit of Theorem \ref{thm::punchline}.
\end{corollary}
\begin{proof}
	Up to equivalence, there are countably many fusion categories.  Let $\cat_{1}, \cat_{2}, \cdots$ be a set of representatives for these fusion categories, 
	let $\cat$ be their free product.  Then $\cat$ is a small RC*TC, $\cat$ with simple unit, and each $\cat_{k}$ embeds into $\cat$ in the obvious way.  
	We can therefore apply the construction in Remark \ref{remark:countable} to $\cat$ to obtain $B$.
\end{proof}


\bibliographystyle{amsalpha}
\bibliography{bibliography}

\end{document}